\documentclass[10pt]{article}%

\usepackage[show]{ed}

\usepackage{hyperref}
\definecolor{myblue}{rgb}{0,0,0.6}
\hypersetup{colorlinks=true, linkcolor=myblue,citecolor=myblue,filecolor=myblue,urlcolor=myblue}

\usepackage{pgf,tikz}
\usetikzlibrary{arrows}
\usetikzlibrary{calc}
\usetikzlibrary{%
    decorations.pathreplacing,%
    decorations.pathmorphing%
}

\usepackage{rotating} 

\usepackage{mathrsfs}
\DeclareMathAlphabet{\mathpzc}{OT1}{pzc}{m}{it}

\usepackage{color}
\usepackage{amsmath}
\usepackage{graphicx}
\usepackage{amsfonts}
\usepackage{amssymb}%
\usepackage{latexsym}
\usepackage{psfrag}
\usepackage{subfigure}
\usepackage{accents}

\newtheorem{theorem}{Theorem}[section]
\newtheorem{corollary}[theorem]{Corollary}
\newtheorem{definition}[theorem]{Definition}
\newenvironment{proof}[1][Proof]{\noindent \emph{#1.} }
{\hfill \ \rule{0.5em}{0.5em}}
\newtheorem{lemma}[theorem]{Lemma}
\newtheorem{proposition}[theorem]{Proposition}

\newtheorem{example}[theorem]{Example}
\newtheorem{remark}[theorem]{Remark}

\numberwithin{equation}{section}
\numberwithin{figure}{section}
\numberwithin{table}{section}

\usepackage[a4paper]{geometry}
\newcommand{\noi}{\noindent}

\newcommand{\la}{\lambda}

\newcommand{\bzero}{\mathbf{0}}
\newcommand{\R}{\mathbb{R}}

\newcommand{\NN}{\mathbb{N}}
\newcommand{\Z}{\mathbb{Z}}
\newcommand{\cA}{{\cal A}}

\newcommand{\cG}{{\cal G}}

\newcommand{\cH}{{\cal H}}
\newcommand{\cI}{{\cal I}}

\newcommand{\cP}{{\cal P}}
\newcommand{\cQ}{{\cal Q}}
\newcommand{\cR}{{\cal R}}
\newcommand{\cS}{{\cal S}}
\newcommand{\cT}{{\cal T}}
\newcommand{\cK}{{\cal K}}

\newcommand{\cV}{{\cal V}}
\newcommand{\cW}{{\cal W}}

\newcommand{\cN}{{\cal N}}
\newcommand{\cB}{{\cal B}}

\newcommand{\bx}{\mathbf{x}}

\newcommand{\bn}{\mathbf{n}}

\newcommand{\ba}{\mathbf{a}}
\newcommand{\be}{\mathbf{e}}
\newcommand{\bb}{\mathbf{b}}

\newcommand{\by}{\mathbf{y}}
\newcommand{\tby}{\widetilde{\mathbf{y}}}
\newcommand{\bz}{\mathbf{z}}
\newcommand{\tbz}{\widetilde{\mathbf{z}}}

\newcommand{\bze}{\mathbf{0}}

\newcommand{\C}{\mathbb{C}}

\newcommand{\re}{{\rm e}}
\newcommand{\ri}{{\rm i}}
\newcommand{\rd}{{\rm d}}

\newcommand{\beq}{\begin{equation}}
\newcommand{\eeq}{\end{equation}}
\newcommand{\beqs}{\begin{equation*}}
\newcommand{\eeqs}{\end{equation*}}
\newcommand{\bit}{\begin{itemize}}
\newcommand{\eit}{\end{itemize}}
\newcommand{\ben}{\begin{enumerate}}
\newcommand{\een}{\end{enumerate}}
\newcommand{\bal}{\begin{align}}
\newcommand{\eal}{\end{align}}
\newcommand{\bals}{\begin{align*}}
\newcommand{\eals}{\end{align*}}
\newcommand{\bse}{\begin{subequations}}
\newcommand{\ese}{\end{subequations}}
\newcommand{\bpr}{\begin{proposition}}
\newcommand{\epr}{\end{proposition}}
\newcommand{\bre}{\begin{remark}}
\newcommand{\ere}{\end{remark}}
\newcommand{\bpf}{\begin{proof}}
\newcommand{\epf}{\end{proof}}
\newcommand{\ble}{\begin{lemma}}
\newcommand{\ele}{\end{lemma}}
\newcommand{\bco}{\begin{corollary}}
\newcommand{\eco}{\end{corollary}}
\newcommand{\bex}{\begin{example}}
\newcommand{\eex}{\end{example}}
\newcommand{\bth}{\begin{theorem}}
\newcommand{\enth}{\end{theorem}}

\newcommand{\Rea}{\mathbb{R}}
\newcommand{\Com}{\mathbb{C}}

\newcommand{\sign}{\mathop{{\rm sign}}}
\newcommand{\spec}{\mathop{{\rm spec}}}
\newcommand{\conv}{\mathop{{\rm conv}}}
\newcommand{\esssup}{\mathop{{\rm ess} \sup}}
\newcommand{\essinf}{\mathop{{\rm ess} \inf}}

\newcommand{\Oi}{{\Omega_-}}

\newcommand{\pdiff}[2]{\frac{\partial #1}{\partial #2}}

\newcommand{\half}{\frac{1}{2}}

\newcommand{\LtG}{{L^2(\Gamma)}}

\newcommand{\tendi}{\rightarrow \infty}
\newcommand{\tendo}{\rightarrow 0}

\newcommand{\nxy}{|\bx-\by|}

\def\XXint#1#2#3{{\setbox0=\hbox{$#1{#2#3}{\int}$}
     \vcenter{\hbox{$#2#3$}}\kern-.5\wd0}}

\newcommand*{\N}[1]{\left\|#1\right\|}

\allowdisplaybreaks[4]

\newcommand{\tfa}{\text{ for all }}
\newcommand{\tfor}{\text{ for }}

\newcommand{\tas}{\text{ as }}
\newcommand{\tand}{\text{ and }}
\newcommand{\tst}{\text{ such that }}
\newcommand{\tfind}{\text{ find }}

\newcommand{\Hilb}{\cH}

\newcommand{\vertiii}[1]{{\left\vert\kern-0.25ex\left\vert\kern-0.25ex\left\vert #1
    \right\vert\kern-0.25ex\right\vert\kern-0.25ex\right\vert}}

\newcommand{\ess}{\mathrm{ess}}

\newcommand{\newdelta}{{\rho}}
\newcommand{\newalpha}{{\beta}}

\newcommand{\newbx}{\by}
\newcommand{\newbz}{\bz}
\newcommand{\newangle}{\theta}

\usepackage{soul}

\definecolor{amcol}{rgb}{0.8,0,0}

\definecolor{escol}{rgb}{0,0,0.8}
\definecolor{estcol}{rgb}{0,0.6,0}

\definecolor{cwcol}{rgb}{0.5,0,0.5}
\definecolor{cwstcol}{rgb}{0,0.6,0.6}

\begin{document}

\title{
Coercivity, essential norms, and the Galerkin method for second-kind integral equations on polyhedral and Lipschitz domains
}

\author{S. N.~Chandler-Wilde\footnotemark[1]\,\,, E. A. Spence\footnotemark[2]}
\date{}



\footnotetext[1]{Department of Mathematics and Statistics, University of Reading,
Whiteknights, PO Box 220, Reading, RG6 6AX, UK, \tt S.N.Chandler-Wilde@reading.ac.uk}
\footnotetext[2]{Department of Mathematical Sciences, University of Bath, Bath, BA2 7AY, UK, \tt E.A.Spence@bath.ac.uk }

\maketitle

\begin{center}
{\em Dedicated to Wolfgang Wendland on the occasion of his 85th birthday}
\end{center}

\begin{abstract}
It is well known that, with a particular choice of norm, the classical double-layer potential operator $D$ has essential norm $<1/2$ as an operator on the natural trace space $H^{1/2}(\Gamma)$ whenever $\Gamma$ is the boundary of a bounded Lipschitz domain. This implies, for the standard second-kind boundary integral equations for the interior and exterior Dirichlet and Neumann problems in potential theory,  convergence of the Galerkin method in $H^{1/2}(\Gamma)$ for any sequence of finite-dimensional subspaces $(\cH_N)_{N=1}^\infty$ that is asymptotically dense in $H^{1/2}(\Gamma)$. Long-standing open questions are whether the essential norm is also $<1/2$ for $D$ as an operator on $L^2(\Gamma)$ for all Lipschitz $\Gamma$ in 2-d; or whether, for all Lipschitz $\Gamma$ in 2-d and 3-d, or at least for the smaller class of Lipschitz polyhedra in 3-d, the weaker condition holds that the operators $\pm \frac{1}{2}I+D$ are compact perturbations of coercive operators -- this a necessary and sufficient condition for the convergence of the Galerkin method for every sequence of subspaces $(\cH_N)_{N=1}^\infty$ that is asymptotically dense in $L^2(\Gamma)$. We settle these open questions negatively. We give examples of 2-d and 3-d Lipschitz domains with Lipschitz constant equal to one for which the essential norm of $D$ is $\geq 1/2$, and examples with Lipschitz constant two for which the operators $\pm \frac{1}{2}I +D$ are not coercive plus compact. We also give, for every $C>0$, examples of Lipschitz polyhedra for which the essential norm is $\geq C$ and for which $\lambda I+D$ is not a compact perturbation of a coercive operator for any real or complex $\lambda$ with $|\lambda|\leq C$. We then, via a new result on the Galerkin method in Hilbert spaces, explore the implications of these results for the convergence of Galerkin boundary element methods in the $L^2(\Gamma)$ setting. Finally, we resolve negatively a related open question in the convergence theory for collocation methods, showing that, for our polyhedral examples, there is no weighted norm on $C(\Gamma)$, equivalent to the standard supremum norm, for which the essential norm of $D$ on $C(\Gamma)$ is $<1/2$.
\end{abstract}

\section{Introduction}\label{sec:intro}

Layer potentials and boundary integral equations have long been an important tool in the mathematics of PDEs (e.g., \cite{Ke:94,Ta:00,MeCo:00,Mc:00,Me:18}), and have been, and continue to be, of equal importance for practical scientific and engineering computation. In particular, numerical methods based on Galerkin, collocation, or numerical quadrature discretisation, coupled with fast matrix-vector multiply and compression algorithms, and iterative solvers such as GMRES, provide spectacularly effective computational tools for solving a range of linear boundary value problems, for example in potential theory, elasticity, and acoustic and electromagnetic wave scattering (e.g., \cite{LaSc:99,BrKu:01,ChSoCuVeHa:04,BoSa:04,XiTaWe:08,GrGuMaRo09,SaSc:11}).

Despite the significant role boundary integral equations (BIEs) play in the analysis of PDEs, and their importance for numerical computation, there remain many open problems for analysis and numerical analysis. Second-kind integral equation formulations, dating back to Gauss and the work of Carl Neumann (see \cite{Co:07, We:09}), continue to be hugely popular in computational practice because they lead naturally to well-conditioned linear systems that can be solved by iterative methods in a small number of iterations (see, e.g., \cite{Rokhlin83,At:97,LaSc:99,BrKu:01,ContETAL02,EpGr09,EpGrON13,ChDaLo:17}). However, even for the classical second-kind integral equations of potential theory there exists no complete convergence theory for Galerkin methods for general Lipschitz domains (or even for general Lipschitz polyhedra in 3-d), set in the Hilbert space of $L^2$ functions on the boundary $\Gamma$, carrying out integration against test functions using the natural $L^2(\Gamma)$ inner product,  despite the utility of such Galerkin methods for large-scale computations (e.g., \cite{LaSc:99,BoSa:04,XiTaWe:08}). Before giving further details, including details of the open questions that we tackle in this paper, we introduce some of the notation that we use.

Throughout, $\Omega_- \subset \Rea^d$, $d=2,3$, is a bounded Lipschitz domain\footnote{We refer to a subset $\Omega\subset \R^d$ as a \emph{domain} if it is open; we do not require additionally that it is connected. A Lipschitz domain is one for which, in some neighbourhood of each point on the boundary, $\Gamma$ can be written, in some rotated coordinate system, as the graph of a Lipschitz continuous function with the domain only on one side of $\Gamma$ (see, e.g., \cite[Definition 3.28]{Mc:00} for details).}, with boundary $\Gamma$ and outward-pointing unit normal vector $\bn$, and
$\Omega_+:= \R^d\setminus \overline{\Omega_-}$ is the exterior of $\overline{\Omega_-}$, also a Lipschitz domain with boundary $\Gamma$. The interior and exterior (in $\Omega_-$ and $\Omega_+$) Dirichlet and Neumann problems for Laplace's equation can be reformulated as BIEs involving the operators
\beq\label{eq:secondkindBIEs}
\half I \pm D \quad\tand\quad \half I \pm D'
\eeq
(see Table \ref{tab:bies}), where the \emph{double-layer operator} $D$ and the \emph{adjoint double-layer operator} $D'$ are defined by
\beq\label{eq:DD'}
D \phi(\bx) = \int_\Gamma \pdiff{\Phi(\bx,\by)}{n(\by)} \phi(\by)\, \rd s(\by) \quad \tand\quad
D' \phi(\bx) = \int_\Gamma \pdiff{\Phi(\bx,\by)}{n(\bx)} \phi(\by)\, \rd s(\by),
\eeq
for $\phi\in \LtG$ and almost all $\bx \in \Gamma$, where $\Phi(\bx,\by)$ is the fundamental solution for Laplace's equation,
\beq\label{eq:fund}
\Phi(\bx,\by):=
\left\{
\begin{array}{cc}
\displaystyle{\frac{1}{2\pi} \log \left(\frac{1}{\nxy}\right),} & d= 2,\\
\displaystyle{\frac{1}{4\pi \nxy}},& d=3,
\end{array}
\right.
\eeq
for $\bx,\by\in \R^d$ with $\bx\neq \by$. Explicitly,
\beq\label{eq:DD'2}
D \phi(\bx) = \frac{1}{c_d}\int_\Gamma \frac{(\bx-\by)\cdot \bn(\by)}{|\bx-\by|^d}\, \phi(\by) \,\rd s(\by) \quad \tand\quad
D' \phi(\bx) = \frac{1}{c_d}\int_\Gamma  \frac{(\by-\bx)\cdot \bn(\bx)}{|\bx-\by|^d}\, \phi(\by)\, \rd s(\by),
\eeq
where $c_d$ is the surface measure of the unit sphere in $\R^d$ ($c_2=2\pi$, $c_3=4\pi$).

For general Lipschitz $\Gamma$, the integrals in the definitions of $D$ and $D'$ are understood as Cauchy principal values, and $D$ and $D'$  are bounded
 on $L^2(\Gamma)$ by  the results on boundedness of the Cauchy integral on Lipschitz $\Gamma$ of Coifman, McIntosh, and Meyer \cite{CoMcMe:82}, following earlier work by Calder\'on \cite{Ca:77} on boundedness for $\Gamma$ with small Lipschitz constant (in the sense of Definition \ref{def:Lipschitz} below). As shown by Verchota \cite{Ve:84} (and see \cite[Appendix A]{El:92}, \cite{MitreaD:97}, \cite[Thm.~2.25]{ChGrLaSp:12}) the operators in \eqref{eq:secondkindBIEs} are also Fredholm of index zero on $L^2(\Gamma)$. Indeed, when $\Gamma$ is connected, $\half I - D$ and $\half I- D'$ are invertible on $L^2(\Gamma)$ and $\half I+ D'$ is invertible on $L^2_0(\Gamma)$, the set of $\phi\in L^2(\Gamma)$ with mean value zero, so that one-rank perturbations of $\half I+ D'$ and $\half I+ D$ are invertible on $L^2(\Gamma)$ \cite{Ve:84}. More generally (see \cite[\S5.15]{Me:18} and \cite{MitreaD:97,KulkarniEtal:05,St:08,SaSc:11,Kr:14}), whatever the topology of $\Gamma$, the interior and exterior Dirichlet problems can be formulated as BIEs of the form
 \begin{equation} \label{eq:GenBIE}
 A \phi = g \quad \mbox{where} \quad A :=  \half I + D_*, \quad \phi, g\in L^2(\Gamma),
 \end{equation}
 the operator $D_*$ is a finite rank or compact perturbation of $\pm D$ or $\pm D'$ as indicated in Table \ref{tab:bies}, and $A:L^2(\Gamma)\to L^2(\Gamma)$ is invertible. The same holds for the Neumann problems provided the Neumann data are square integrable\footnote{For example, when $\Omega_-$ is connected, the interior Neumann problem with data $g\in L_0^2(\Gamma)$ can be formulated as \eqref{eq:GenBIE} with $D_*=D'+P$, where $P$ is orthogonal projection onto the orthogonal complement of $L_0^2(\Gamma)$, and $A:L^2(\Gamma)\to L^2(\Gamma)$ given by \eqref{eq:GenBIE} is invertible (see, e.g., proof of \cite[Thm.~2.25]{ChGrLaSp:12}).}.

\begin{table}
\begin{center}
\begin{tabular}{|c|c|c|c|c|}
\hline
&Interior Dirichlet   & Interior Neumann & Exterior Dirichlet  & Exterior Neumann  \\
& problem & problem & problem & problem \\
\hline
 Direct   & $\frac{1}{2}I- D^\prime$   & $\frac{1}{2}I+ D$  &  $\frac{1}{2}I+ D'$ & $ \frac{1}{2}I- D$\\
 \hline
 Indirect & $ \frac{1}{2}I- D$ & $ \frac{1}{2}I+ D'$& $\frac{1}{2}I+ D$ & $ \frac{1}{2}I- D'$\\
\hline
     \end{tabular}
     \end{center}
     \caption{
\emph{The integral operators involved in the standard second-kind integral-equation formulations of the interior and exterior Dirichlet and Neumann problems for Laplace's equation. Direct and indirect have the standard meanings (e.g.\ \cite{SaSc:11}): direct formulations are those obtained from a Green's representation theorem by taking traces; indirect formulations are those obtained by using a layer potential as an ansatz and taking traces.
}   \label{tab:bies}}
\end{table}

The Galerkin method for \eqref{eq:GenBIE} in $L^2(\Gamma)$ requires first choosing a sequence of finite-dimensional approximation spaces $(\cH_N)_{N=1}^\infty$ in $L^2(\Gamma)$ that is asymptotically dense in $L^2(\Gamma)$, meaning that
$$
\inf_{\phi_N\in \cH_N} \|\psi-\phi_N\|_{L^2(\Gamma)} \to 0 \quad\tas \quad N\to \infty,
$$
for every $\psi\in L^2(\Gamma)$. Then, for each $N$, we seek an approximation $\phi_N\in \cH_N$ such that
\begin{equation} \label{eq:GalL2}
(A\phi_N,\psi_N)_{L^2(\Gamma)} = (g, \psi_N)_{L^2(\Gamma)} \quad \mbox{for all } \psi_N\in \cH_N.
\end{equation}
We say that this Galerkin method converges if, for some $N_0\in \NN$, $\phi_N$ is well-defined by \eqref{eq:GalL2} for all $N\geq N_0$,  and all $g\in \LtG$,
and $\phi_N\to \phi=A^{-1}g$ in $L^2(\Gamma)$ as $N\to \infty$ for all $g\in L^2(\Gamma)$.

It follows from existing, general results on the Galerkin method in a Hilbert space setting (Theorem \ref{thm:Galerkin} below) that the Galerkin method \eqref{eq:GalL2} converges for every asymptotically dense approximation sequence $(\cH_N)_{N=1}^\infty\subset \LtG$ if and only if $A$ can be written as the sum of a coercive and a compact operator (coercive in the sense of \eqref{eq:coer} below). In particular, $A$ is coercive plus compact if $\|D\|_{L^2(\Gamma), \ess} < 1/2$, where
\beq\label{eq:essnorm}
\big\|D\big\|_{L^2(\Gamma), \ess} := \inf_{K \text{ compact }} \big\| D - K \big\|_{\LtG}
\eeq
is the essential norm of $D$ as an operator on $L^2(\Gamma)$, for then $D=D_\dag + K$ with $\|D_\dag\|_{\LtG}<1/2$ and $K$ compact, so that $\half I+D_\dag$ is coercive.

Wendland \cite{We:09} has reviewed the state-of-the-art in numerical analysis of Galerkin and collocation methods for solution of \eqref{eq:GenBIE}, and the state-of-the-art in related analysis questions for the double-layer potential operator $D$, with some emphasis on the case when $\Gamma$ is Lipschitz polyhedral (meaning that $\Omega_-$ is a Lipschitz polyhedron). As he notes, most of the existing proofs of convergence for the Galerkin method \eqref{eq:GalL2} (all the proofs in 2-d) rely on establishing that $\|D\|_{\LtG,\ess} < 1/2$. These comprise the cases where:
\begin{itemize}
\item[(i)] $\Gamma$ is $C^1$, when $D$ (and its $L^2(\Gamma)$ adjoint $D'$) are compact by \cite{FaJoRi:78}, so that $\|D\|_{\LtG,\ess} = 0$;
\item[(ii)] $\Gamma$ is a 2-d curvilinear polygon with each side $C^{1,\alpha}$ for some $0<\alpha<1$ and with each corner angle in the range $(0,2\pi)$; this result was announced in \cite{Sh:69}, with details of the proof given in \cite{Sh:91}, and with the analogous result for polygons following from the result of \cite[\S3]{Ch:84}, (see, e.g., \cite[Lemma 1.5]{BoMaNiPa:90}).
\end{itemize}
Additionally, Wendland suggests that $\|D\|_{\LtG, \ess}<1/2$ if the Lipschitz character of $\Gamma$ (in the sense of Definition \ref{def:Lipschitz})  is small enough ``due to the results of \cite{Mi:99}". The results in \cite[Lemma 1, Page 392]{Mi:99} concern the essential spectral radius but the arguments can be adapted to prove that $\|D\|_{\LtG, \ess}<1/2$ if the Lipschitz character of $\Gamma$ is small enough, and we do this below in Corollary \ref{cor:smallM}.

As Wendland notes,  the Galerkin method \eqref{eq:GalL2} has also been studied by Elschner \cite{El:92a,El:95}, who analyses spline-Galerkin methods when $\Gamma$ is Lipschitz polyhedral and $A=\half I-D$. Importantly, Elschner's analysis does not assume that $\|D\|_{\LtG,\ess}<1/2$; indeed he announces in \cite[Remark 3.4(i)]{El:95} that, even in the case when $\Omega_-$ is a convex polyhedron, it can hold that  $\|D\|_{\LtG,\ess}>1/2$. Nevertheless, he is able to prove convergence, for a certain class of Lipschitz polyhedra, of $h$- and $hp$-Galerkin boundary element methods (in \cite{El:92a} and \cite{El:95}, respectively), with approximation spaces carefully adapted to the singularities of the solution at corners and edges of the polyhedron. His analysis reduces proof of stability and convergence to a requirement of injectivity of $\half I - D$ either on the spaces $L^2(\Gamma_*^{(j)})$ or on the spaces $L^\infty(\Gamma_*^{(j)})$, where $\Gamma_*^{(j)}$ is an infinite cone associated to the $j$th corner of $\Gamma$ but with strips along the edges of the cone deleted \cite[Equation (2.3)]{El:92a}. This requirement is satisfied if $\|D\|_{\LtG,\ess}<1/2$, but also if $\|D\|_{C(\Gamma),\ess}<1/2$, so in particular (see the discussion below \eqref{eq:ineq2}) when $\Omega_-$ is convex. Convergence of these Galerkin methods would hold for all Lipschitz polyhedra if we could show that $A$ is coercive plus compact on $L^2(\Gamma)$ whenever $\Gamma$ is Lipschitz polyhedral\footnote{Additionally, as Elschner makes clear (\cite[Remark 4.4]{El:92a}, \cite[Remark 3.6]{El:95}), the injectivity condition \cite[Equation (2.3)]{El:92a} is satisfied if there exists a weight function $w\in L^\infty(\Gamma)$ satisfying \eqref{eq:wbound} below for some $c_->0$ such that $\|D\|_{C_w(\Gamma),\ess}<1/2$.  Thus a positive answer to either of the open questions Q2 or Q3 below would complete a proof of convergence of Elschner's Galerkin methods for all Lipschitz polyhedral $\Gamma$.}.

\subsection{The open questions we address} \label{sec:open}

Wendland \cite[\S1, \S3.2, \S4.2]{We:09} (and see Elschner \cite[Remark A.3]{El:92}) flags the following long-standing open questions that are the focus of our paper:

\bit
\item[Q1.] Is $\|D\|_{\LtG, \ess}<1/2$ in 2-d whenever $\Gamma$ is Lipschitz?
\item[Q2.] Does the Galerkin method \eqref{eq:GalL2} converge for every asymptotically dense sequence of finite-dimensional approximation spaces $(\cH_N)_{N=1}^\infty\subset \LtG$ whenever $\Gamma$ is Lipschitz ($d=2$ or $3$), in particular whenever $\Gamma$ is Lipschitz polyhedral ($d=3$)?
\eit

As discussed above, as a consequence of Theorem \ref{thm:Galerkin}, Q2 can be rephrased equivalently as:

\bit
\item[Q2$^\prime$.] Can the operators $\half I\pm D$ and $\half I\pm D'$ be written as the sum of a coercive operator and a compact operator whenever $\Gamma$ is Lipschitz ($d=2$ or $3$), in particular whenever $\Gamma$ is Lipschitz polyhedral ($d=3$)?
\eit
As we have noted above, a positive answer to Q1 implies a positive answer to Q2$^\prime$, equivalently a positive answer to Q2.

We address Q2 and Q2$^\prime$ via a further reformulation of these questions in terms of $W_\ess(D)$, the essential numerical range of $D$ (the definitions of the numerical range and essential numerical range of a linear operator are recalled below in \eqref{eq:numrangedef} and \eqref{eq:essnumrangedef}). As a consequence of a general property of bounded linear operators on Hilbert spaces that we recall in Corollary \ref{cor:essnum}, and since $D'$ is the $L^2(\Gamma)$ adjoint of $D$ so that $W_\ess(D')=W_\ess(D)$, Q2 can also be rephrased as follows:

\bit
\item[Q2$^{\prime\prime}$.] Are the points $\pm 1/2$ outside the essential numerical range $W_{\ess}(D)$ of $D$ on $L^2(\Gamma)$ when $\Gamma$ is Lipschitz ($d=2$ or 3), in particular when $\Gamma$ is Lipschitz polyhedral ($d=3$)?
\eit
We note that Q2$^{\prime\prime}$ has a positive answer if $w_\ess(D)<1/2$, where
$$
w_\ess(D) := \sup_{z\in W_\ess(D)} |z|
$$
is the essential numerical radius of $D$.

There are at least two reasons for anticipating that the above questions might have positive answers.
Firstly, thanks to Steinbach and Wendland \cite{StWe:01} (and see \cite[Remark A.3]{El:92a}), provided we equip the natural trace space $H^{1/2}(\Gamma)$ with the appropriate norm, $D$ has essential norm $<1/2$ as an operator on $H^{1/2}(\Gamma)$ for every Lipschitz $\Gamma$. Thus also the Galerkin method \eqref{eq:GalL2} converges if we replace the $L^2(\Gamma)$ inner product in \eqref{eq:GalL2} by an inner product on $H^{1/2}(\Gamma)$. This Galerkin method, using a non-local $H^{1/2}(\Gamma)$ inner product, is less attractive for numerical computation, but these positive answers to Q1 and Q2 for $H^{1/2}(\Gamma)$ might encourage a hope of positive answers also for $\LtG$.

Secondly, there has been progress on a related long-standing open problem
concerning the essential spectrum of $D$ as an operator on $\LtG$,  $\spec_{\LtG,\ess}(D)$, and specifically the essential spectral radius $r_{\LtG,\ess}(D) := \sup_{z\in \spec_\ess(D)} |z|$. This open problem \cite[Problem 3.2.12]{Ke:94}\footnote{The phrasing in Kenig \cite[Problem 3.2.12]{Ke:94} is different, namely that [in the case when $\Gamma$ is connected] the spectral radius of $D'$ as an operator on $L^2_0(\Gamma)$ is $<1/2$. But, recalling that $r_{\LtG,\ess}(D)=r_{\LtG,\ess}(D')$, this is equivalent to \eqref{eq:ress}, since any eigenvalues of $D'$ lie in $[-1/2,1/2]$  \cite[Theorem 1.1]{FaSaSe:92}, and $\pm \half I + D'$ is known to be Fredholm on $L^2(\Gamma)$, and also invertible on $L_0^2(\Gamma)$ as long as $\Gamma$ is connected \cite[Theorem 4.1]{MitreaD:97}.} is to show that
\begin{equation} \label{eq:ress}
r_{\LtG,\ess}(D) < \half \quad \mbox{for all Lipschitz $\Gamma$.}
\end{equation}
This bound has been shown to hold when $\Omega_-$ is convex \cite{FaSaSe:92} (and see \cite{ChLe08} for extensions to locally convex domains), for all $\Gamma$ with sufficiently small Lipschitz character \cite{Mi:99}, and for $\Gamma$ Lipschitz polyhedral \cite[Theorem 4.1]{El:92} (and see \cite{GrMa88,Ma12,GrMa13}). Since
\begin{equation} \label{eq:ineq}
r_{\LtG,\ess}(D) \leq w_\ess(D) \leq \|D\|_{\LtG,\ess},
\end{equation}
the bound \eqref{eq:ress} also holds for the cases cited above where it is known that $\|D\|_{\LtG,\ess}<1/2$. One might hope that, at least in some of the cases where it is known that $r_{\LtG,\ess}(D) < \half$, it holds also that $\|D\|_{\LtG,\ess}<1/2$, or at least that $w_\ess(D)<1/2$, either of these enough to give a positive answer to Q2.

The final long-standing open question that we address, flagged by Wendland \cite[\S4.1]{We:09} (and see \cite{KrWe86,AnKlKr88,KrWe88,Ha:01}), is concerned specifically with the case when $\Gamma$ is Lipschitz polyhedral, in which case $D$ is well-defined also as a bounded operator on $C(\Gamma)$. To explain this conjecture we note the following general relationship, in a Banach space $X$ equipped with a norm $\|\cdot\|_X$, between the essential spectral radius $r_{X,\ess}(\cA)$ of a bounded linear operator $\cA$ and its essential norm $\|\cA\|_{X,\ess}$. Generalising \eqref{eq:ineq} it trivially holds that $r_{X,\ess}(\cA) \leq \|\cA\|_{X,\ess}$. But it can also be shown \cite{GoMa60} that, for every $\varepsilon>0$ there exists an equivalent norm $\|\cdot\|^\prime$ on $X$ such that
\begin{equation} \label{eq:ineq2}
\|\cA\|_{X^\prime,\ess} \leq r_{X,\ess}(\cA)+\varepsilon,
\end{equation}
where $X^\prime$ denotes $X$ equipped with $\|\cdot\|^\prime$.

We can apply this observation to the case that $\Gamma$ is Lipschitz polyhedral. In that case (see \cite{We:09}) $\|D\|_{C(\Gamma),\ess}<1/2$ when $\Omega_-$ is convex, but not for all non-convex $\Omega_-$. However, $r_{C(\Gamma),\ess}(D)<1/2$ (see \cite[Theorem 0.1]{Ra:92}, \cite{Ra:95}, and cf.\ \cite{GrMa88,Ma12,GrMa13}), so that there must exist a norm $\|.\|^\prime$ on $C(\Gamma)$, equivalent to the standard maximum norm, for which the induced essential norm of $D$ is also $<1/2$. Motivated by the numerical analysis of collocation methods for \eqref{eq:GenBIE} in the case $D_*=\pm D$, Kr\'al and Wendland \cite{KrWe86} consider, specifically, weighted norms equivalent to the standard maximum norm. Given $w\in L^\infty(\Gamma)$ with, for some $c_->0$,
\begin{equation} \label{eq:wbound}
c_-\leq w(x) \leq 1, \quad \mbox{for almost all } x\in \Gamma,
\end{equation}
they define the norm $\|\cdot\|_{C_w(\Gamma)}$ by
\begin{equation} \label{eq:wnormDef}
\|\psi\|_{C_w(\Gamma)} := \esssup_{x\in \Gamma} |\psi(x)/w(x)|, \quad \psi\in C(\Gamma).
\end{equation}
(Of course $\|\cdot\|_{C_1(\Gamma)}$ is the standard supremum norm, and $\|\cdot\|_{C_w(\Gamma)}$ and $\|\cdot\|_{C_1(\Gamma)}$ are equivalent by \eqref{eq:wbound}.) In \cite{KrWe86} they construct examples of Lipschitz (and non-Lipschitz) polyhedral $\Gamma$ and $w$ for which  $\|D\|_{C_w(\Gamma),\ess}<1/2$ (although $\|D\|_{C_1(\Gamma),\ess}>1/2$). Generalising these examples, Angell et al.\ \cite{AnKlKr88} and Kr\'al and Wendland \cite{KrWe88} show that, whenever $\Omega_-$ is a so-called rectangular domain, meaning that each side of $\Gamma$ lies in one of the Cartesian coordinate planes, a piecewise constant weight $w$ can be constructed so that  $\|D\|_{C_w(\Gamma),\ess}<1/2$. Extending further, Hansen developed in \cite{Ha:01} a procedure for general polyhedral $\Gamma$ to systematically generate piecewise constant weight functions; the class of polyhedral $\Gamma$ for which this procedure generates a $w$ with $\|D\|_{C_w(\Gamma),\ess}<1/2$ is termed {\em Hansen's class} in \cite{We:09}.  As Wendland \cite{We:09} notes: ``It is still not clear whether Hansen's procedure always provides a weight function with $\|\half I+D\|_{C_w(\Gamma),\ess}<1$ for any arbitrary polyhedral domain. Hence the stability and convergence of the collocation method for piecewise smooth $\Gamma$ is in part open.'' This, and the discussion above of the convergence theory for Elschner's spline-Galerkin methods,  motivate the final open question that we address in this paper:
\bit
\item[Q3.] For every Lipschitz polyhedral $\Gamma$, does there exist a weight function $w\in L^\infty(\Gamma)$ satisfying \eqref{eq:wbound} for some $c_->0$ such that $\|D\|_{C_w(\Gamma),\ess}<1/2$?
\eit

\subsection{The main results and their implications}

Our main results are that we answer, in the negative, questions Q1--Q3, and hence we also answer Q2$^\prime$ and Q2$^{\prime\prime}$ negatively. Our first main result addresses Q1 (we recall the definition of the Lipschitz constant of a Lipschitz domain in Definition \ref{def:Lipschitz} below).

\begin{theorem}[Answer to Q1] \label{thm:Q1} For every $M>0$ there exists, for both $d=2$ and $3$ (i.e., in both 2-d and 3-d), a bounded Lipschitz domain $\Omega_d^{M}$, with Lipschitz constant $M$, such that if $\Gamma$ is the boundary of $\Omega_d^{M}$ then
$$
\|D\|_{\LtG,\ess} \geq M/2.
$$
In particular if $\Gamma$ is the boundary of $\Omega_d^{1}$, which has Lipschitz constant one, then $\|D\|_{\LtG,\ess} \geq 1/2$.
\end{theorem}
Our method of proof is constructive. Indeed, the particular domain $\Omega_d^{M}$ that we use to prove this result is shown, for $d=2$ and $M=1$, in Figure \ref{fig:OmegaM}, and is specified in Definitions \ref{def:OmegaM2d} and \ref{def:OmegaM3d}, for the 2-d and 3-d cases, respectively. We note that, complementing this result, we show below in Corollary \ref{cor:smallM} that, for every $M_0> 0$ and every bounded Lipschitz domain $\Omega_-$ with Lipschitz character $M\leq M_0$,
$$
\|D\|_{\LtG,\ess} \leq C\,M,
$$
where the constant $C$ depends only on $d$ and $M_0$ (but must be $\geq 1/2$ by the above theorem), so that $\|D\|_{\LtG,\ess} < 1/2$ if the Lipschitz character $M$ of $\Omega_-$ is small enough.

The same domains $\Omega_d^{M}$ provide a negative answer to Q2$^{\prime\prime}$, and so also a negative answer to Q2$^\prime$ and Q2.

\begin{theorem}[Answer to Q2, Q2$^\prime$, Q2$^{\prime\prime}$] \label{thm:Q2} For $M>0$ and $d=2,3$, if $\Gamma$ is the boundary of $\Omega_d^{M}$, which has Lipschitz constant $M$, then
$$
\big\{\lambda\in \C:|\lambda| \leq M/4\big\} \subset W_\ess(D),
$$
so that $\lambda I + D$ and $\lambda I + D'$ are not coercive plus compact for any $\lambda$ with $|\lambda|\leq M/4$.
In particular, if $\Gamma$ is the boundary of $\Omega_d^{2}$, which has Lipschitz constant two, then $\half I \pm D$ and $\half I \pm D'$ cannot be written as the sum of coercive and compact operators, so that there exists an asymptotically dense sequence of finite-dimensional spaces $(\cH_N)_{N=1}^\infty\subset \LtG$ such that the Galerkin method \eqref{eq:GalL2} does not converge.
\end{theorem}

The domains $\Omega_d^M$ that feature in the above results can be thought of as curvilinear polygons with infinitely many sides when $d=2$ and curvilinear polyhedra with infinitely many sides when $d=3$;
see Figure \ref{fig:OmegaM} for $\Omega_2^1$, and Figures \ref{fig:OmegaM3d} and \ref{fig:GM3} for the key portions of $\Omega_3^2$.
 In the 3-d case similar results can be obtained (though without such an explicit dependence on the Lipschitz constant) for the case when $\Gamma$ is Lipschitz polyhedral. The next result features the family of Lipschitz polyhedra, $\{\Omega_{\theta,n}:0<\theta \leq \pi, n \in \{2,3,...\}\}$, specified in Definition \ref{def:Gamma3-d}, that we term {\em open-book polyhedra}; precisely, we refer to $\Omega_{\theta,n}$ as the {\em open-book polyhedron with $n$ pages and opening angle $\theta$}. Figure \ref{fig:book} shows $\Omega_{\theta,n}$ for $\theta=\pi/2$ and $n=4$ (and see Figure \ref{fig:book2}). In the following theorem and hereafter, $\conv(T)$ denotes the convex hull of $T\subset \C$.

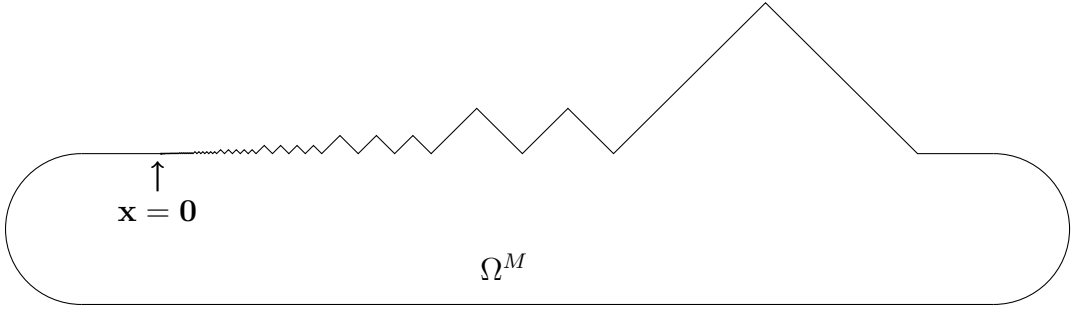
\begin{figure}
\begin{tikzpicture}[scale=10]
\draw (-0.1,0) arc(90:270:0.1);
\draw (1.1,-0.2) arc(-90:90:0.1);
\draw plot coordinates {(-0.1,-0.2) (1.1,-0.2)};
\draw plot coordinates {(-0.1,0)
(0,0) (0.00047018,0) (0.00048063,1.0449e-05) (0.00049108,0) (0.00050153,1.0449e-05) (0.00051198,0) (0.00052243,1.0449e-05) (0.00053288,0) (0.00054332,1.0449e-05) (0.00055377,0) (0.00056422,1.0449e-05) (0.00057467,0) (0.00058512,1.0449e-05) (0.00059557,0) (0.00060602,1.0449e-05) (0.00061646,0) (0.00062691,1.0449e-05) (0.00063736,0) (0.00064781,1.0449e-05) (0.00065826,0) (0.00066871,1.0449e-05) (0.00067916,0) (0.0006896,1.0449e-05) (0.00070005,0) (0.0007105,1.0449e-05) (0.00072095,0) (0.0007314,1.0449e-05) (0.00074185,0) (0.0007523,1.0449e-05) (0.00076274,0) (0.00077319,1.0449e-05) (0.00078364,0) (0.0008023,1.8658e-05) (0.00082096,0) (0.00083962,1.8658e-05) (0.00085827,0) (0.00087693,1.8658e-05) (0.00089559,0) (0.00091425,1.8658e-05) (0.00093291,0) (0.00095156,1.8658e-05) (0.00097022,0) (0.00098888,1.8658e-05) (0.0010075,0) (0.0010262,1.8658e-05) (0.0010449,0) (0.0010635,1.8658e-05) (0.0010822,0) (0.0011008,1.8658e-05) (0.0011195,0) (0.0011381,1.8658e-05) (0.0011568,0) (0.0011755,1.8658e-05) (0.0011941,0) (0.0012128,1.8658e-05) (0.0012314,0) (0.0012501,1.8658e-05) (0.0012688,0) (0.0012874,1.8658e-05) (0.0013061,0) (0.0013396,3.3489e-05) (0.001373,0) (0.0014065,3.3489e-05) (0.00144,0) (0.0014735,3.3489e-05) (0.001507,0) (0.0015405,3.3489e-05) (0.001574,0) (0.0016075,3.3489e-05) (0.001641,0) (0.0016744,3.3489e-05) (0.0017079,0) (0.0017414,3.3489e-05) (0.0017749,0) (0.0018084,3.3489e-05) (0.0018419,0) (0.0018754,3.3489e-05) (0.0019089,0) (0.0019424,3.3489e-05) (0.0019758,0) (0.0020093,3.3489e-05) (0.0020428,0) (0.0020763,3.3489e-05) (0.0021098,0) (0.0021433,3.3489e-05) (0.0021768,0) (0.0022372,6.0466e-05) (0.0022977,0) (0.0023582,6.0466e-05) (0.0024186,0) (0.0024791,6.0466e-05) (0.0025396,0) (0.0026,6.0466e-05) (0.0026605,0) (0.002721,6.0466e-05) (0.0027814,0) (0.0028419,6.0466e-05) (0.0029024,0) (0.0029628,6.0466e-05) (0.0030233,0) (0.0030838,6.0466e-05) (0.0031442,0) (0.0032047,6.0466e-05) (0.0032652,0) (0.0033256,6.0466e-05) (0.0033861,0) (0.0034466,6.0466e-05) (0.003507,0) (0.0035675,6.0466e-05) (0.003628,0) (0.0037379,0.00010994) (0.0038478,0) (0.0039578,0.00010994) (0.0040677,0) (0.0041777,0.00010994) (0.0042876,0) (0.0043975,0.00010994) (0.0045075,0) (0.0046174,0.00010994) (0.0047274,0) (0.0048373,0.00010994) (0.0049472,0) (0.0050572,0.00010994) (0.0051671,0) (0.005277,0.00010994) (0.005387,0) (0.0054969,0.00010994) (0.0056069,0) (0.0057168,0.00010994) (0.0058267,0) (0.0059367,0.00010994) (0.0060466,0) (0.0062482,0.00020155) (0.0064497,0) (0.0066513,0.00020155) (0.0068528,0) (0.0070544,0.00020155) (0.0072559,0) (0.0074575,0.00020155) (0.007659,0) (0.0078606,0.00020155) (0.0080622,0) (0.0082637,0.00020155) (0.0084653,0) (0.0086668,0.00020155) (0.0088684,0) (0.0090699,0.00020155) (0.0092715,0) (0.009473,0.00020155) (0.0096746,0) (0.0098761,0.00020155) (0.010078,0) (0.010451,0.00037325) (0.010824,0) (0.011197,0.00037325) (0.011571,0) (0.011944,0.00037325) (0.012317,0) (0.01269,0.00037325) (0.013064,0) (0.013437,0.00037325) (0.01381,0) (0.014183,0.00037325) (0.014557,0) (0.01493,0.00037325) (0.015303,0) (0.015676,0.00037325) (0.01605,0) (0.016423,0.00037325) (0.016796,0) (0.017496,0.00069984) (0.018196,0) (0.018896,0.00069984) (0.019596,0) (0.020295,0.00069984) (0.020995,0) (0.021695,0.00069984) (0.022395,0) (0.023095,0.00069984) (0.023795,0) (0.024494,0.00069984) (0.025194,0) (0.025894,0.00069984) (0.026594,0) (0.027294,0.00069984) (0.027994,0) (0.029327,0.001333) (0.03066,0) (0.031993,0.001333) (0.033326,0) (0.034659,0.001333) (0.035992,0) (0.037325,0.001333) (0.038658,0) (0.039991,0.001333) (0.041324,0) (0.042657,0.001333) (0.04399,0) (0.045323,0.001333) (0.046656,0) (0.049248,0.002592) (0.05184,0) (0.054432,0.002592) (0.057024,0) (0.059616,0.002592) (0.062208,0) (0.0648,0.002592) (0.067392,0) (0.069984,0.002592) (0.072576,0) (0.075168,0.002592) (0.07776,0) (0.082944,0.005184) (0.088128,0) (0.093312,0.005184) (0.098496,0) (0.10368,0.005184) (0.10886,0) (0.11405,0.005184) (0.11923,0) (0.12442,0.005184) (0.1296,0) (0.1404,0.0108) (0.1512,0) (0.162,0.0108) (0.1728,0) (0.1836,0.0108) (0.1944,0) (0.2052,0.0108) (0.216,0) (0.24,0.024) (0.264,0) (0.288,0.024) (0.312,0) (0.336,0.024) (0.36,0) (0.42,0.06) (0.48,0) (0.54,0.06) (0.6,0) (0.8,0.2) (1,0)
(1.1,0)};
\draw (0.5,-0.15) node[anchor=east] {\large $\Omega^M$};
\draw [thick,->] (0,-0.05) -- ++(90:0.04);
\draw (0,-0.05) node[anchor=north] {\large $\bx = \bze$};
\end{tikzpicture}
\caption{\label{fig:OmegaM} The domain $\Omega_d^M$ in 2-d (i.e., $d=2$), as specified in Definition \ref{def:OmegaM2d}, for Lipschitz constant $M=1$ (and $\beta=0.6$). When $\Gamma = \partial \Omega_d^M$ is the boundary of this domain, $\|D\|_{\LtG,\ess}\geq 1/2$ because of the oscillatory, self-similar geometry locally to the point $\bx=\bze \in\Gamma$ indicated by the arrow.}
\end{figure}

\begin{figure}
\includegraphics[width=.5\textwidth]{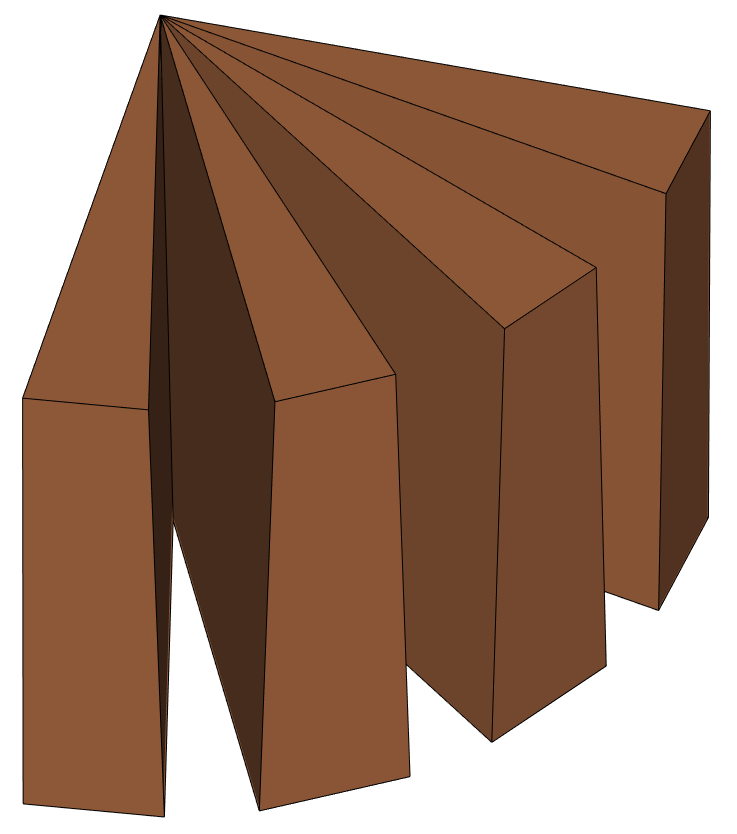}\includegraphics[width=.5\textwidth]{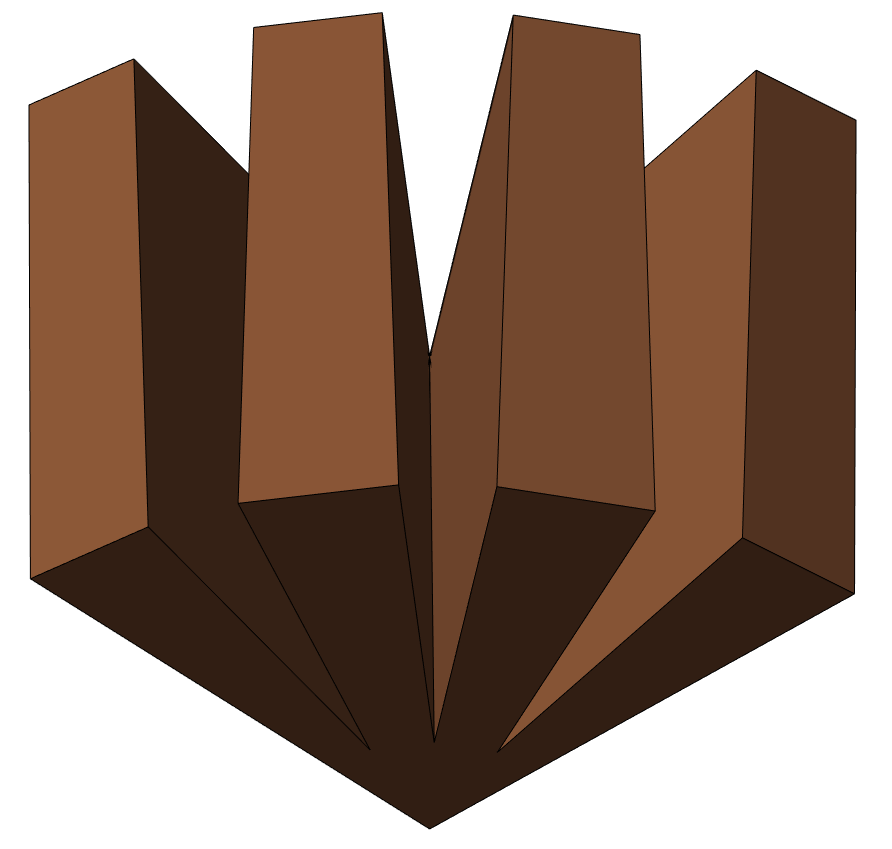}
\caption{\label{fig:book} Views from above and below of the open-book polyhedron $\Omega_{\theta,n}$, with $n=4$ pages and opening angle $\newangle = \pi/2$}.
\end{figure}

\begin{theorem}[Answer to Q2, Q2$^\prime$, Q2$^{\prime\prime}$ for Lipschitz polyhedral $\Gamma$] \label{thm:Q2Poly} Suppose that $n\in \NN$ with $n\geq 2$, and that $\Gamma = \Gamma_{\theta,n}$, the boundary of the Lipschitz polyhedron $\Omega_{\theta,n}$ (the open-book polyhedron with $n$ pages and opening angle $\theta$) given by Definition \ref{def:Gamma3-d}. Then, for every $\epsilon>0$, there exists $\theta_0\in (0,\pi]$ such that
\begin{eqnarray*}
\conv\Big([-\sqrt{n}/2+\epsilon,\sqrt{n}/2-\epsilon] \cup \{\lambda\in \C: |\lambda| \leq \sqrt{n-1}/2-\epsilon\}\Big) \subset W_\ess(D) \quad \mbox{and}\\
\|D\|_{\LtG,\ess} \geq w_\ess(D)\geq \sqrt{n}/2-\epsilon \quad \mbox{for }0<\theta\leq \theta_0.
\end{eqnarray*}
Thus, for all sufficiently small $\theta\in (0,\pi]$, $\half I \pm D$ and $\half I \pm D'$ cannot be written as the sum of coercive and compact operators, so that there exists an asymptotically dense sequence of finite-dimensional spaces $(\cH_N)_{N=1}^\infty\subset \LtG$ such that the Galerkin method \eqref{eq:GalL2} does not converge.
\end{theorem}

Our proofs of the above results depend on:
\bit
\item[(a)] the equivalence of  Q2, Q2$^\prime$, and Q2$^{\prime\prime}$, discussed in a general Hilbert space context in \S\ref{sec:GM};
\item[(b)] localisation results for the essential norm and essential numerical range of $D$ that we prove as Theorems \ref{lem:local} and \ref{thm:loccon}, adapting essential spectrum localisation arguments, in particular those from \cite{Mi:99} (and see \cite{FaSaSe:92,El:92,Ra:92,ChLe08,LeCoPer:21});
\item[(c)] the simple observation that, if $D^\dag$ is a restriction of $D$ from $L^2(\Gamma)$ to some subspace, then $\|D\|_{\LtG}\geq \|D^\dag\|_{\LtG}$ and the numerical range $W(D)\supset W(D^\dag)$ (see \S\ref{sec:CompNR});
\item[(d)] that if $D$ and $\widetilde D$ are the double-layer potential operators on surfaces $\Gamma$ and $\widetilde \Gamma$ that are geometrically similar ($\widetilde \Gamma$ is the image of $\Gamma$ under an affine transformation), then $D$ and $\widetilde D$ are unitarily equivalent, so that $\|D\|_\LtG = \|\widetilde D\|_{L^2(\widetilde \Gamma)}$ and $W(D)=W(\widetilde \Gamma)$ (Lemmas \ref{lem:Visometry}, \ref{lem:Vcommute}, and \ref{lem:dil2});
\item[(e)] for Theorems \ref{thm:Q1} and \ref{thm:Q2}, lower bounds for the norm and numerical range of $D$ in the case when $\Gamma$ is the graph $\Gamma^M$ of a particular Lipschitz continuous function with Lipschitz constant $M$ (see Definition \ref{def:GammaM} and Figure \ref{fig:Saw}), these lower bounds obtained by relating $D$ on a subspace of $\LtG$ to particular infinite Toeplitz matrices with piecewise continuous symbols, and computing the jumps in those symbols (see \S\ref{sec:pergraph});
\item[(f)]  for Theorem \ref{thm:Q2Poly}, explicit computations of the asymptotics of particular finite-dimensional discretisations of $D$ for the open-book polyhedron $\Omega_{\theta,n}$ in the ``closing-the-book'' limit $\theta\to 0$ (see \S\ref{sec:Thm13Proof}).
\eit

The implications of Theorems \ref{thm:Q2} and \ref{thm:Q2Poly} for the numerical analysis of the Galerkin method \eqref{eq:GalL2} for the standard second-kind integral equations of potential theory are significant. The negative answers that these results give to Q2$^\prime$ mean that there is no longer hope of proving convergence of particular Galerkin methods for all Lipschitz $\Gamma$, or even for all Lipschitz polyhedral $\Gamma$, by showing that the operators $\half I \pm D$ or $\half I \pm D'$ are coercive plus compact on $L^2(\Gamma)$; this contrasts with the situation for the same operators on $H^{\pm 1/2}(\Gamma)$, and for the situation for the standard first kind integral equations of potential theory \cite{Co:88,Mc:00}.

On the other hand the implications for the Galerkin method in computational practice are at first sight more modest: our proofs, that use the equivalence of Q2, Q2$^\prime$, and Q2$^{\prime\prime}$, show initially only that  the Galerkin method \eqref{eq:GalL2} does not converge for every sequence  $(\cH_N)_{N=1}^\infty\subset \LtG$. This leaves open the possibility that Galerkin methods used in practice, notably all Galerkin methods based on boundary element method discretisation \cite{St:08,SaSc:11}, are in fact convergent.

Our next main result clarifies that this is not the case. As a corollary of a new general result for the Galerkin method in Hilbert spaces that we prove as Theorem \ref{thm:GalGenNew} below, drawing inspiration from arguments of Markus \cite{Ma:74} used to prove the equivalence of Q2 and Q2$^\prime$, we obtain the following result.

\begin{theorem} \label{thm:Q2Ext} Suppose that $\Gamma$ is one of the geometries in Theorem \ref{thm:Q2} or \ref{thm:Q2Poly} for which $\half I \pm D$ and $\half I \pm D'$ cannot be written as the sum of coercive and compact operators, and that $(\cH^*_N)_{N=1}^\infty$ is a sequence of finite-dimensional subspaces of $\LtG$ with $\cH^*_1\subset \cH^*_2\subset ...$ that is asymptotically dense in $\LtG$, i.e.
$$
\LtG = \overline{\cup_{N=1}^\infty \cH^*_N}.
$$
Then there exists a sequence $(\cH_N)_{N=1}^\infty$ of finite-dimensional subspaces of $\LtG$ with $\cH_1\subset \cH_2\subset ...$ such that the Galerkin method \eqref{eq:GalL2} is not convergent but, for each $N\in \NN$,
\begin{equation} \label{eq:embed}
\cH_N^*\subset \cH_N \subset \cH^*_{M_N} \quad \mbox{for some } M_N\in \NN.
\end{equation}
\end{theorem}

We can apply this result in the case that $(\cH^*_N)_{N=1}^\infty$ is an asymptotically dense sequence of boundary element subspaces, in which case $(\cH_N)_{N=1}^\infty$, satisfying \eqref{eq:embed}, is also a sequence of boundary element subspaces (since $\cH_N\subset \cH_{M_N}^*$) and is also asymptotically dense (since $\cH^*_N\subset \cH_N$). Thus this result implies that there exist Lipschitz and polyhedral boundaries $\Gamma$ for which there are Galerkin methods \eqref{eq:GalL2} based on asymptotically dense sequences $(\cH_N)_{N=1}^\infty$ of boundary element subspaces that do not converge.

We present elsewhere \cite{ChSp:12} alternative second-kind integral equations for the interior and exterior Laplace Dirichlet problems for general Lipschitz domains. These take the form $A\phi=g$, with $A$ coercive on $\LtG$, so that the Galerkin method \eqref{eq:GalL2} converges for every asymptotically dense sequence $(\cH_N)_{N=1}^\infty$, indeed the C\'ea's lemma estimate of Theorem \ref{thm:Galerkin}(b) applies. Other convergent methods for general Lipschitz domains  have been developed by Dahlberg and Verchota \cite{DaVe:90} and by Adolfsson et al.\ \cite{AdGoJaLe:92}, based on Galerkin solution of second-kind integral equations on the boundaries of a sequence of smooth domains approximating $\Gamma$.

Our final main result answers Q3 negatively, again using the open-book polyhedra as counterexamples.

\begin{theorem}[Answer to Q3] \label{thm:Q3}
For every  $n\in \{2,3,...\}$ and $\epsilon>0$ there exists  $\theta_0\in (0,\pi/4]$ such that, if $\Gamma$ is the boundary of the open-book Lipschitz polyhedron $\Omega_{\theta,n}$ and $0<\theta\leq \theta_0$, then
$$
\|D\|_{C_w(\Gamma),\ess}\geq \frac{2n-1}{4}-\epsilon,
$$
for every weight function $w\in L^\infty(\Gamma)$ satisfying \eqref{eq:wbound} for some $c_->0$.
\end{theorem}
Since $(2n-1)/4 \geq 3/4>1/2$ for $n\geq 2$, this result, applied with any $n\geq 2$, means that Hansen's class of polyhedra does not include all polyhedra. Related to this observation, this negative result also closes off the main potential route to completing convergence proofs, for all polyhedra, of collocation and quadrature methods for \eqref{eq:GenBIE} (see \cite[p.~172]{Ra:92}, \cite[Remark 2.3]{Ra:93}, \cite[\S4]{Ha:01}, \cite[\S4.1]{We:09}), and closes off a main potential route for finalising the proof of convergence for all polyhedra of Elschner's spline-Galerkin methods (see \cite[Remark 4.4]{El:92a}, \cite[Remark 3.6]{El:95} and the discussion above in \S\ref{sec:intro}). Our proof of this result depends on a non-trivial extension of a result of Kr\'al and Medkov\'a \cite{KrMe:00} which we prove as Theorem \ref{thm:ENweighted}, showing that the localisation formula \eqref{eq:ENweighted} for $\|D\|_{C_w(\Gamma),\ess}$ that Kr\'al and Medkov\'a prove for lower semicontinuous $w$ satisfying \eqref{eq:wbound} holds more generally for $w\in L^\infty(\Gamma)$ satisfying \eqref{eq:wbound}.

\subsubsection{Implications for other second-kind boundary integral equations and their Galerkin solution} \label{sec:OtherImp}

We have discussed the implications of our results for the boundary integral equations \eqref{eq:GenBIE} for the Dirichlet and Neumann problems in potential theory and their Galerkin solution \eqref{eq:GalL2}. Our results also have implications for BIE formulations of other elliptic boundary value problems.

\paragraph{Transmission problems.} The equation $A\phi=g$, with
\begin{equation} \label{eq:Alambda}
A = \lambda I-D^\prime
\end{equation}
and $\lambda\in \C\setminus\{-\half,\half\}$, arises in the transmission problem for the Laplace equation (e.g., \cite[\S5.12]{Me:18}), which in turn arises as the low-frequency (quasi-static) limit of transmission problems for the Maxwell system \cite{AmDeMi:16}. The case when $\lambda\in \R$ with $|\lambda|>1/2$ (e.g., \cite[\S5.12]{Me:18}), is classical, and $A$ is known to be invertible on $\LtG$ for $\lambda$ in this range \cite[Theorem 1.3]{FaSaSe:92}. More recently,
complex $\lambda$ with $|\lambda|< 1/2$ have been studied, motivated by applications to nanoparticle plasmonic resonances \cite{AmDeMi:16}, specifically the case where the particle has negative permittivity. This application has prompted much work on computation of the spectrum of $D^\prime$ on $H^{-1/2}(\Gamma)$, $\LtG$ and other spaces for particular geometries (e.g., \cite{AmDeMi:16,HePe:17,Schnitzer:20,LeCoPer:21}).

Theorems \ref{thm:Q2} and \ref{thm:Q2Poly} make clear that, for any $\lambda\in \C$, there exist Lipschitz polyhedra and 2-d and 3-d Lipschitz domains with Lipschitz constant as small as $4|\lambda|$, such that $A=\lambda I-D^\prime$ is not coercive plus compact. Thus, when $A$ is given by \eqref{eq:Alambda} with $\lambda\in \C\setminus\{-\half,\half\}$,  similar conclusions to those of Theorem \ref{thm:Q2Ext} follow regarding the non-convergence of the Galerkin method \eqref{eq:GenBIE} for all asymptotically dense subsequences $(\cH_N)_{N=1}^\infty$, as a corollary of Theorem \ref{thm:GalGenNew}. (We note, on the other hand, that, for the particular sequence of approximating subspaces proposed by Elschner \cite{El:92a}, the Galerkin method \eqref{eq:GalL2} has been shown  to converge \cite{El:92a} for every Lipschitz polyhedral $\Gamma$ for all but finitely many (unknown, $\Gamma$-dependent) values of $\lambda$ with $|\lambda|\geq 1/2$, indeed for all $\lambda \neq -1/2$ with $|\lambda|\geq 1/2$ if $\Omega_-$ is a convex polyhedron.)

\paragraph{Helmholtz problems.} Theorems \ref{thm:Q2} and \ref{thm:Q2Poly} exhibit geometries for which $A = \half I + D_*$ is not coercive plus compact whenever $D_*$ is a compact perturbation of $\pm D$ or $\pm D'$. Second-kind BIEs of the form \eqref{eq:GenBIE} with such $D_*$ are widely used for computation of the solution of interior and exterior boundary value problems for the Helmholtz equation $\Delta u + k^2u=0$, modelling time-harmonic acoustic and electromagnetic problems (e.g., \cite{CoKr:83}, \cite[\S2.5]{ChGrLaSp:12}). In particular, the standard BIEs to solve the exterior Dirichlet Helmholtz problem, due to  Brakhage and Werner \cite{BrWe:65}, Leis \cite{Le:65}, Panich \cite{Pa:65}, and Burton and Miller \cite{BuMi:71}, take the form \eqref{eq:GenBIE} with
\begin{equation} \label{eq:D*k}
D_* = D_k-\ri \eta k S_k \quad \mbox{ or } \quad D_* = D_k'-\ri \eta k S_k.
\end{equation}
Here $\eta>0$ is a positive constant and $S_k$ and $D_k$ are the standard acoustic single- and double-layer potential operators, defined by
\begin{equation} \label{eq:DLPSLPk}
D_k \phi(\bx) = \int_\Gamma \pdiff{\Phi_k(\bx,\by)}{n(\by)} \phi(\by)\, \rd s(\by) \quad \tand\quad S_k \phi(\bx) = \int_\Gamma \Phi_k(\bx,\by) \phi(\by)\, \rd s(\by), \quad \bx \in \Gamma,
\end{equation}
where $\Phi_k(\bx,\by)$ is the Helmholtz fundamental solution, given in 3-d ($d=3$) by
\beq\label{eq:fundk}
\Phi_k(\bx,\by):= \frac{\re^{\ri k|\bx-\by|}}{4\pi \nxy}, \quad \bx,\by\in \R^d, \quad \bx\neq \by.
\eeq
The operator $D_k'$ is the adjoint of $D_k$ with respect to the real $L^2$-inner product, given by the formula \eqref{eq:DD'} for $D'$ with $\Phi$ replaced by $\Phi_k$.

It is well-known that $D_k$, $D_k'$, and $S_k$ are bounded operators on $\LtG$ (e.g.\ \cite{TorWel93}, \cite[Theorem 2.17]{ChGrLaSp:12}). Indeed \cite{TorWel93}, $D_k-D$, $D_k'-D'$, and $S_k$ are integral operators with weakly singular kernels and so are compact on $\LtG$.  Moreover, if $A=\half+D_*$ and $D_*$ is given by \eqref{eq:D*k}, then $A$ is invertible on $\LtG$ (for the general Lipschitz case see \cite[Theorem 2.7]{ChLa:07} or \cite[Theorem 2.27]{ChGrLaSp:12}). But $A$ may not be coercive plus compact.  Since  $D_*$ given by \eqref{eq:D*k} is a compact perturbation of $D$ or $D'$, we have the following corollary of Theorems \ref{thm:Q2} and \ref{thm:Q2Poly}.

\begin{corollary} \label{cor:Helm}
Suppose that $d=2$ or $3$ and $\Gamma$ is the boundary of $\Omega_d^M$, for some $M\geq 2$, or that $d=3$ and $\Gamma=\Gamma_{\theta,n}$ is the boundary of the Lipschitz open-book polyhedron $\Omega_{\theta,n}$, for some $n\geq 2$. Then, provided also that $\theta>0$ is sufficiently small in the case $\Gamma=\Gamma_{\theta,n}$,
$$
A = \half I + D_*,
$$
with $D_*$ given by \eqref{eq:D*k}, is not the sum of a coercive and a compact operator for any $k>0$ and $\eta>0$.
\end{corollary}

Surprisingly, this corollary is relevant to conjectures in the literature regarding the  large-$k$ limit. When $\Omega_-$ is Lipschitz and star-shaped (in the sense of Definition \ref{def:ss}), e.g. a convex polyhedron, and when $\Gamma$ is $C^\infty$ and non-trapping (in the sense of \cite[Definition 1.1]{BaSpWu:16}),  it has been shown \cite{ChMo:08,BaSpWu:16} that $\|A^{-1}\|_{\LtG}=O(1)$ as $k\to\infty$ with $\eta$ fixed. When $\Omega_-$ is $C^3$ and is strictly convex with strictly positive curvature, and with $\Gamma$ piecewise analytic, the stronger result  that $A$ is coercive on $\LtG$, uniformly in $k$, has been shown; precisely, by \cite[Theorem 1.2]{SpKaSm:15}, there exist $\eta_0$, $k_0$, and $\alpha_0>0$ such that
\begin{equation} \label{eq:coercivityk}
|(A\phi,\phi)_{\LtG}| \geq \alpha_0\|\phi\|^2_\LtG, \quad \mbox{for all } \quad \phi\in\LtG, \;\; \eta \geq \eta_0, \;\; k\geq k_0.
\end{equation}

This strong result is surprising as the Helmholtz equation is considered a prime example of an indefinite problem (see the discussion in \cite{MoSp:14}) for which the most one can hope is that the associated operators are compact perturbations of coercive operators. The bound \eqref{eq:coercivityk} guarantees, as part (ii) of Theorem \ref{thm:Galerkin} makes clear,  that the Galerkin equations \eqref{eq:GalL2} are well-posed, moreover with stability constants that are independent of $k$ for $k\geq k_0$, and independent of the chosen approximation space sequence $(\cH_N)_{N=1}^\infty$. This is highly attractive for high-frequency numerical solution methods as discussed in \cite{ChGrLaSp:12,SpKaSm:15}.

It is natural to speculate whether \eqref{eq:coercivityk} holds for more general classes of $\Gamma$, and to investigate this by numerical experiment, and this was done via  approximate computation of the numerical range of $A$ in \cite{BeSp:11} (using the equivalence of a) and d) in Lemma \ref{lem:coer_equiv}), leading to conjectures that:
\begin{itemize}
\item[(i)] The bound \eqref{eq:coercivityk} holds (with the choice $\eta_0=1$) for all non-trapping domains \cite[Conjecture 6.2]{BeSp:11};
\item[(ii)] $A$ is coercive (with the choice $\eta_0=1$) for all positive $k>0$ that are not close, in the sense of \cite[Remark 5.2]{BeSp:11}, to resonances of the exterior Helmholtz Dirichlet problem \cite[Conjecture 6.1]{BeSp:11}.
\end{itemize}
The open-book polyhedra are star-shaped with respect to a ball (Definition \ref{def:ss}), and so certainly non-trapping in the sense of \cite{BeSp:11}. Thus Corollary \ref{cor:Helm}, which shows that there exist open-book polyhedra for which $A$ (for all $k>0$ and $\eta>0$) is not even a compact perturbation of a coercive operator, establishes that these conjectures do not hold in the 3-d case\footnote{A quite different counterexample to the first of these conjectures is given in \cite[\S6.3.2]{ChSpGiSm:20}. The example constructed (in both 2-d and 3-d) is a $\Gamma$ that is $C^\infty$ and  non-trapping for which $\inf_{\phi\in \LtG, \|\phi\|_{\LtG}=1}|(A\phi,\phi)_{\LtG}|\to 0$ as $k\to \infty$ at least as fast as $k^{-1}$, so that the bound \eqref{eq:coercivityk} is not satisfied with an $\alpha_0$ independent of $k$. (But it might still be the case for the counterexample in \cite[\S6.3.2]{ChSpGiSm:20} that $A$ is coercive for each fixed $k$.)}.

\section{The Galerkin method and essential numerical range in Hilbert spaces} \label{sec:GM}

In this section we recall, in a Hilbert space setting, the definitions of the coercivity, numerical range, and essential numerical range of a bounded linear operator, the relationships between these concepts, and their role in the convergence theory of the Galerkin method. These results provide the justification for the equivalence of the open questions Q2, Q2$^\prime$,  and  Q2$^{\prime\prime}$ in \S\ref{sec:open}. In particular we attack questions Q2 (about convergence of the Galerkin method for the standard 2nd kind integral equations of potential theory) and Q2$^\prime$ (coercivity plus compactness of the related operators), via the  reformulation Q2$^{\prime\prime}$ in terms of the essential numerical range of the double-layer potential operator. The equivalence of Q2$^\prime$  and  Q2$^{\prime\prime}$ is provided by Corollary \ref{cor:essnum} below, and that between Q2 and Q2$^\prime$ by Theorem \ref{thm:Galerkin}. The main new result in this section is Theorem \ref{thm:GalGenNew}, a significant strengthening of part (c) of Theorem \ref{thm:Galerkin}, that further explores the relationship between the coercivity plus compactness of an operator $\cA$ and convergence of the associated Galerkin method.

We say that a linear operator $\cA:\cH \rightarrow \cH$, where $\cH$ is a Hilbert space, is \emph{coercive}
\footnote{In the literature, the property \eqref{eq:coer} (and its analogue for operators $A: \cH\rightarrow \cH'$, where $\cH'$ is the dual of $\cH$) is sometimes called ``$\Hilb$-ellipticity" (as in, e.g., \cite[Page 39]{SaSc:11}, \cite[\S3.2]{St:08}, and \cite[Definition 5.2.2]{HsWe:08}) or ``strict coercivity'' (e.g., \cite[Definition 13.28]{Kr:14}),  with ``coercivity" then used to mean \emph{either} that $\cA$ is the sum of a coercive operator and a compact operator (as in, e.g., \cite[\S3.6]{St:08} and \cite[\S5.2]{HsWe:08}) \emph{or} that $\cA$ satisfies a G\aa rding inequality (as in \cite[Definition 2.1.54]{SaSc:11}).} if there exists an $\alpha>0$ such that
\beq\label{eq:coer}
\big|(\cA\phi,\phi)_{\cH}\big|\geq \alpha \N{\phi}^2_{\cH} \quad \tfa \phi \in \cH.
\eeq
A closely related concept is positive definiteness. We say that a bounded linear operator $\cB:\cH\to \cH$ is \emph{strictly} \emph{positive definite} if it is self-adjoint and if, for some $\beta>0$, $(\cB\phi,\phi)_{\cH} \geq \beta \|\phi\|_\cH^2$, for all $\phi\in \cH$.

Importantly, coercivity of an operator is also related to its numerical range. Recall that the \emph{numerical range} (also known as the \emph{field of values}) of a bounded linear operator $\cA:\cH\rightarrow \cH$ is defined by
\beq \label{eq:numrangedef}
W(\cA):= \Big\{
(\cA \psi,\psi)_\cH\,:\, \N{\psi}_\cH =1
\Big\},
\eeq
and the \emph{essential numerical range} is defined by
\beq\label{eq:essnumrangedef}
W_{\ess}(\cA) = \bigcap_{\cK \text{ compact }} \overline{W(\cA+\cK)}.
\eeq
For every bounded linear operator $\cA$, $W(\cA)$ and $W_\ess(\cA)$ are convex, bounded sets (e.g., \cite{GuRa:97,BDav,bonsall1973numerical}); we call
$$
w(\cA) := \sup_{z\in W(\cA)} |z| \quad \mbox{and} \quad w_\ess(\cA) := \sup_{z\in W_\ess(\cA)} |z|
$$
the numerical radius and essential numerical radius of $\cA$, and note that \cite[\S1.3]{GuRa:97}
\begin{equation} \label{eq:NormEssNorm}
\half \|\cA\|_\cH \leq w(\cA) \leq \|\cA\|_\cH \quad \mbox{so that also} \quad \half \|\cA\|_{\cH,\ess} \leq w_\ess(\cA) \leq \|\cA\|_{\cH,\ess},
\end{equation}
where
$$
\|\cA\|_{\cH,\ess} := \inf_{\cK \text{ compact }} \big\| \cA - \cK \big\|_{\cH}
$$
is the essential norm of $\cA$.
We recall that the spectrum of $\cA$ is the set of $\lambda\in \Com$ such that $\cA-\lambda\cI$ is not invertible, and the essential spectrum of $\cA$ is the
set of $\lambda\in \Com$ such that $\cA-\lambda\cI$ is not Fredholm.
We recall also that (e.g., \cite{GuRa:97,BDav,bonsall1973numerical}) $\overline{W(\cA)}$ contains the spectrum of $\cA$ and $W_\ess(\cA)$ its essential spectrum, so that also
$$
r_{\cH}(\cA) \leq w(\cA) \quad \mbox{and} \quad r_{\cH,\ess}(\cA) \leq w_\ess(\cA),
$$
where $r_{\cH}(\cA):= \max_{z\in \spec(\cA)}|z|$ and $r_{\cH,\ess}(\cA):= \max_{z\in \ess \, \spec(\cA)}|z|$ are the spectral radius and essential spectral radius of $\cA$. We note also the following elementary but key observation: if $\cV$ is a closed subspace of $\cH$, $\cP:\cH\to\cV$ is orthogonal projection, and $\widetilde \cA:= \cP\cA|_\cV$, then
\begin{equation} \label{eq:subspaces}
W(\widetilde \cA) \subset W(\cA), \;\;\;  W_\ess(\widetilde \cA) \subset W_\ess(\cA), \;\;\; \|\widetilde \cA\|_\cV \leq \|\cA\|_\cH, \;\;\; \mbox{and } \;\;\|\widetilde \cA\|_{\cV,\ess} \leq \|\cA\|_{\cH,\ess}.
\end{equation}

\begin{lemma}[Equivalent formulations of coercivity] \label{lem:coer_equiv} Given a bounded linear operator $\cA:\cH \rightarrow \cH$ on a Hilbert space $\cH$, the following are equivalent:
\begin{enumerate}
\item[a)] $\cA$ is coercive;
\item[b)] $\cA=c(\cB+\ri \cG)$, for some $c\in \C\setminus \{0\}$ and bounded linear operators $\cB$ and $\cG$ such that $\cG$ is self-adjoint and $\cB$ is strictly positive definite;
\item[c)] $\cA=c(\cI+\cB)$, for some $c\in \C\setminus \{0\}$ and bounded linear operator $\cB$ such that $\|\cB\|_\cH<1$;
\item[d)] $0 \notin \overline{W(\cA)}$;
\item[e)] For some $\theta\in \R$ and $\alpha >0$, $\Re \left[\re^{\ri\theta} (\cA\phi,\phi)_\cH \right]\geq \alpha \|\phi\|_\cH^2$, for all $\phi\in \cH$.
\end{enumerate}
\end{lemma}
\begin{proof}[References for the proof] The equivalence of b)-d) is shown, for example, in the discussion around \cite[Chapter II, Lemma 5.1]{GoFe:74}, and the equivalence of a) and d), for example, in \cite[Lemma 3.3]{BeSp:11}. That d) and e) are equivalent is immediate from the convexity of $W(\cA)$ (e.g., \cite[Theorem 1.1-2]{GuRa:97}).
\end{proof}

We are interested, in particular, in the equivalence of a) and d) in the above lemma, and in the following straightforward corollary of that equivalence, which implies the equivalence of Q2$^\prime$ and Q2$^{\prime\prime}$ in \S\ref{sec:open}. For completeness we include the short proof.

\begin{corollary} \label{cor:essnum}  If $\cA:\cH\to \cH$ is a bounded linear operator, then $\cA$ is the sum of a coercive operator plus a compact operator if and only if $0 \notin W_{\ess}(\cA).$
\end{corollary}
\begin{proof}
If $\cA+\cK$ is coercive and $\cK$ is compact, then $0\not\in \overline{W(\cA+\cK)}$ by the above lemma, and so $0 \notin W_{\ess}(\cA).$ Conversely, if $0 \notin W_{\ess}(\cA)$ then, for some compact $\cK$, $0\not\in \overline{W(\cA+\cK)}$, so that by the above lemma $\cA+\cK$ is coercive.
\end{proof}

\subsection{The Galerkin method in Hilbert spaces}
Recall the definition of the \emph{Galerkin method} for approximating solutions of the operator equation $\cA\phi = g$, where $\phi, g \in \cH$, $\cA:\cH\rightarrow \cH$ is a bounded linear operator, and $\cH$ is a Hilbert space: given a sequence $(\cH_N)_{N=1}^\infty$ of finite-dimensional subspaces of $\cH$ with $\dim(\cH_N)\tendi$ as $N\tendi$,
\beq\label{eq:G}
\tfind \phi_N \in\cH_N \tst \big(\cA\phi_N,\psi_N)_\cH=\big(g,\psi_N\big)_\cH \quad \tfa \psi_N\in \cH_N.
\eeq
The Galerkin method \emph{is convergent for the sequence $(\cH_N)_{N=1}^\infty$} if, for every $g\in \cH$, the Galerkin equations \eqref{eq:G} have a unique solution for all sufficiently large $N$ and $\phi_N\to \cA^{-1}g$ as $N\to\infty$.

Closely related to this definition, $(\cH_N)_{N=1}^\infty$ is \emph{asymptotically dense in $\cH$} or \emph{converges to $\cH$} if, for every $\phi\in \cH$,
$$
\inf_{\psi_N\in \cH_N}\|\phi-\psi_N\|_\cH \to 0 \quad \mbox{as} \quad N\to \infty.
$$
It is easy to see that if the Galerkin method converges then $(\cH_N)_{N=1}^\infty$ converges to $\cH$. Indeed, a standard necessary and sufficient condition (e.g., \cite[Chapter II, Theorem 2.1]{GoFe:74}) for convergence of the Galerkin method is that $(\cH_N)_{N=1}^\infty$ converges to $\cH$ and that, for some $N_0\in \NN$ and $c>0$,
\begin{equation} \label{eq:lower}
\frac{\|\cP_N\cA\phi_N\|_\cH}{\|\phi_N\|_\cH} \geq c, \quad \mbox{for all non-zero } \phi_N\in \cH_N \mbox{ and } N\geq N_0,
\end{equation}
where $\cP_N$ is orthogonal projection of $\cH$ onto $\cH_N$.

The equivalence of Q2 and Q2$^\prime$ in \S\ref{sec:open} is given by Part (c) of the following theorem.

\begin{theorem}[The Main Abstract Theorem on the Galerkin Method.]\label{thm:Galerkin}

\

(a) If $\cA$ is invertible then there exists a sequence $(\cH_N)_{N=1}^\infty$ for which the Galerkin method converges.

(b)  If $\cA$ is coercive (i.e.~\eqref{eq:coer} holds) then, for every sequence $(\cH_N)_{N=1}^\infty$, the Galerkin equations \eqref{eq:G} have a unique solution $\phi_N$ for every $N$ and
\beqs
\big\|\phi-\phi_N\big\|_\cH \leq \frac{\|\cA\|_\cH}{\alpha} \,  \inf_{\psi \in \cH_N}\big\|\phi-\psi\big\|_\cH,
\eeqs
so $\phi_N\to \cA^{-1}g$ as $N\to\infty$ if $(\cH_N)_{N=1}^\infty$ converges to $\cH$.

(c) If $\cA$ is invertible then the following are equivalent:
\begin{itemize}
\item The Galerkin method converges for every sequence $(\cH_N)_{N=1}^\infty$ that converges to $\cH$.
\item $\cA=\cA_0+\cK$ where $\cA_0$ is coercive and $\cK$ is compact.
\end{itemize}
\end{theorem}

\bpf[References for the proof]
Part (a) was first proved in \cite[Theorem 1]{Ma:74}; see also \cite[Chapter II, Theorem 4.1]{GoFe:74}.
Part (b) is C\'ea's Lemma; see \cite{Ce:64}.
Part (c) was first proved in \cite[Theorem 2]{Ma:74}, with this result building on results in \cite{Va:65}; see also \cite[Chapter II, Lemma 5.1 and Theorem 5.1]{GoFe:74}
\epf

Part (c) of the above theorem implies that, if $\cA$ is invertible but not coercive plus compact, then there exists {\em at least one} sequence of finite-dimensional subspaces  $(\cH_N)_{N=1}^\infty$ converging to $\cH$ for which the Galerkin method is not convergent. The following stronger result, Theorem \ref{thm:GalGenNew}, which implies Theorem \ref{thm:Q2Ext}, builds on the arguments in \cite{Ma:74} to show that, if $\cA$ is not coercive plus compact, then the Galerkin method fails for many subspace sequences: precisely, given a sequence of finite-dimensional subspaces $(\cH^*_N)_{N=1}^\infty$ converging to $\cH$, with $\cH^*_1\subset \cH^*_2 \subset ...$, either:

i) the Galerkin method is not convergent for the sequence $(\cH^*_N)_{N=1}^\infty$; or

ii) the Galerkin method for $(\cH^*_N)_{N=1}^\infty$ does converge but there exists a sequence $(\cH_N)_{N=1}^\infty$, sandwiched by the sequence $(\cH^*_N)_{N=1}^\infty$ (i.e.\ satisfying (b) in Theorem \ref{thm:GalGenNew}), such that the Galerkin method is not convergent for $(\cH_N)_{N=1}^\infty$.

Note that, in Theorem \ref{thm:GalGenNew}, $(\cH_N)_{N=1}^\infty$ {\em is} a sequence that converges to $\cH$. This follows from condition (b) and the convergence of the Galerkin method for $(\cH^*_N)_{N=1}^\infty$. For the proof of this theorem we need the following lemma.

\begin{lemma} \label{lem:cc}
If $(K_n)_{n=1}^\infty$ is a collectively compact sequence of bounded linear operators, meaning that \cite{An:71}
\begin{equation} \label{eq:cc}
\bigcup_{n\in \NN}\{K_n\phi:\|\phi\|_\cH \leq 1\}
\end{equation}
is relatively compact, and if each $K_n$ is self-adjoint, then $\phi_n\rightharpoonup 0$ as $n\to \infty$ implies that  $K_n\phi_n\to 0$, for every sequence $(\phi_n)_{n=1}^\infty \subset \cH$.
\end{lemma}
\begin{proof} Suppose that $\phi_n\rightharpoonup 0$, in which case $(\phi_n)$ is bounded. To show that $K_n\phi_n\to 0$ it is enough to show that each subsequence of $(K_n\phi_n)$ has a subsequence converging to zero. But if $(K_n)_{n=1}^\infty$ is collectively compact and each $K_n$ is self-adjoint then, by \cite[Corollary 5.7]{An:71}, $\{K_n:n\in \NN\}$ is relatively compact as a subset of the Banach space of bounded linear operators on $\cH$. So take a subsequence of $(K_n\phi_n)$, denoted again by $(K_n\phi_n)$. This has a subsequence, denoted again by $(K_n\phi_n)$, such that $\|K_n-\widehat K\|_\cH\to 0$, for some bounded linear operator $\widehat K$, and the relative compactness of \eqref{eq:cc} implies that $\widehat K$ is compact \cite[\S1.4]{An:71}, and so completely continuous. Thus $\|K_n\phi_n\|_\cH \leq \|K_n-\widehat K\|_\cH \|\phi_n\|_\cH + \|\widehat K\phi_n\|_\cH \to 0$.
\end{proof}

Surprisingly to us, if self-adjointness is dropped, then the above lemma does not hold, indeed it need not hold even that $K_n\phi_n\rightharpoonup 0$.  A counterexample is the operator sequence of \cite[Example 5.5]{An:71}.

\begin{theorem} \label{thm:GalGenNew} Suppose that $\cA$ is invertible but $\cA$ cannot be written in the form $\cA=\cA_0+\cK$, where $\cA_0$ is coercive and $\cK$ is compact, and that $(\cH^*_N)_{N=1}^\infty$ is a sequence of finite-dimensional subspaces of $\cH$, with $\cH^*_1\subset \cH^*_2 \subset ...$, for which the Galerkin method is convergent.
Then there exists a sequence $(\cH_N)_{N=1}^\infty$ of finite-dimensional subspaces of $\cH$, with $\cH_1\subset \cH_2 \subset ...$, such that:
\begin{enumerate}
\item[(a)] the Galerkin method is not convergent for the sequence $(\cH_N)_{N=1}^\infty$; and
\item[(b)] for each $N\in \NN$,
$\cH^*_N \subset \cH_N \subset \cH^*_{M_N}$, for some $M_N\in \NN$.
\end{enumerate}
\end{theorem}
\begin{proof}
We prove this result by constructing a sequence $(\cH_N)_{N=1}^\infty$ satisfying (b) for which
$$
\inf_{\phi\in \cH_N\setminus\{0\}}\frac{\|\cP_N\cA\phi\|_\cH}{\|\phi\|_\cH}\to 0 \quad \mbox{as} \quad N\to\infty,
$$
where $\cP_N$ denotes orthogonal projection onto $\cH_N$,
so that \eqref{eq:lower}, a necessary condition for convergence of the Galerkin method, fails.

Using that $0\in W_{\ess}(\cA)$ by Corollary \ref{cor:essnum}, \cite[Lemma 2]{Ma:74} (or \cite[Chapter II, Lemma 5.2]{GoFe:74}) shows that there exists an orthonormal sequence $(\phi_m)_{m=1}^\infty$ in $\cH$ such that $(\cA\phi_m,\phi_m)_\cH\to 0$ as $m\to\infty$. Further, arguing exactly as done to prove (5.4) in \cite[Chapter II, Lemma 5.2]{GoFe:74}, noting that every orthonormal sequence is weakly convergent to zero, so that $\phi_m\rightharpoonup 0$ and
$$
(\cA\phi_m,\phi)_\cH = (\phi_m,\cA'\phi)_\cH \to 0
$$
as $m\to\infty$, for every $\phi\in \cH$, where $\cA'$ is the adjoint of $\cA$,
we see that, by taking subsequences if necessary, we can choose the orthonormal sequence so that
\begin{equation} \label{eq:Rn}
\cR_n \cA\phi_n \to 0 \quad \mbox{as} \quad n\to\infty,
\end{equation}
where $\cR_n$ is orthogonal projection onto the subspace $\cV_n := \mathrm{span}\{\phi_1,...,\phi_n\}$.

Now choose a sequence $(\psi_m)_{m=1}^\infty$ such that, for every $m\in \NN$, $\psi_m\subset \cH^*_n$, for some $n\in \NN$, and
\begin{equation} \label{eq:close}
\|\phi_m-\psi_m\| \leq 2^{-m}, \quad m\in \NN.
\end{equation}
(This is possible since the Galerkin method is convergent for $(\cH^*_N)_{N=1}^\infty$ so that this sequence converges to $\cH$.) We note that $(\psi_m)_{m=1}^\infty$ is a Riesz sequence. Indeed, if, for some $m,n\in \NN$ with $n\geq n$ and some $\underline{a}=(a_m,...,a_{n})\in \C^{n+1-m}$,
$$
\psi := \sum_{j=m}^{n} a_{j} \psi_j \quad \mbox{ and } \quad \phi := \sum_{j=m}^{n} a_{j} \phi_j,
$$
then, by \eqref{eq:close} and Cauchy-Schwarz, where $\|\underline{a}\|_2:= (\sum_{j=m}^{n} |a_j|^2)^{1/2}$,
\begin{equation} \label{eq:psiphi}
\|\psi-\phi\|_\cH \leq \frac{\|\underline{a}\|_2}{2^{m-1}\sqrt{3}}\leq \frac{\|\underline{a}\|_2}{\sqrt{3}}, \quad \mbox{ so that} \quad \frac{\sqrt{3}-1}{\sqrt{3}} \, \|\underline{a}\|_2 \leq \|\psi\|_\cH \leq \frac{\sqrt{3}+1}{\sqrt{3}} \, \|\underline{a}\|_2
\end{equation}
since $\|\phi\|_\cH=\|\underline{a}\|_2$. Note also that, since \eqref{eq:close} holds and $\phi_m\rightharpoonup 0$ as $m\to\infty$, we have also that $\psi_m\rightharpoonup 0$.

For $m,n\in \NN$ with $m\leq n$, let $\cQ_{m,n}$ denote orthogonal projection onto the subspace $\cW_{m,n} := \mathrm{span}\{\psi_m,...,\psi_n\}$, and $\cQ_n$ orthogonal projection onto $\cW_n:= \cW_{1,n}$. We now show that \eqref{eq:Rn} holds with the orthonormal sequence $(\phi_n)$ replaced by its approximation $(\psi_n)$, precisely that
\begin{equation} \label{eq:Pn}
\cQ_n \cA\psi_n \to 0 \quad \mbox{as} \quad n\to\infty.
\end{equation}

For $m\in \NN$, let $\cW^m$ denote the subspace of those $\psi\in \cH$ that have the representation
$$
\psi = \sum_{j=m}^\infty a_j \psi_j \quad \mbox{with} \quad \sum_{j=m}^\infty|a_j|^2 < \infty;
$$
this is a closed subspace by the second set of inequalities in \eqref{eq:psiphi}. Let $\cQ^m$ denote orthogonal projection onto $\cW^m$. We note that $\cQ_{m,n}$ converges strongly to $\cQ^m$ as $n\to\infty$, i.e., for every $m\in \NN$,
\begin{equation} \label{eq:strongproj}
\cQ_{m,n} \phi \to \cQ^m\phi, \quad \mbox{as } n\to\infty, \tfa \phi\in \cH.
\end{equation}
As a consequence, for each fixed $m\in \NN$, the family of operators $\{\cQ_n-\cQ_{m,n}: n\geq m\}$ is collectively compact.
For if $(p_j)_{j=1}^\infty\subset \cH$ is a bounded sequence and, for each $j\in \NN$, $n_j\in \NN$ with $n_j\geq m$, then, for some coefficients $\underline{a}^j=(a_{1,j},...,a_{n_j,j})\in \C^{n_j}$, with $\|\underline{a}^j\|_2 \leq \sqrt{3}/(\sqrt{3}-1)\|p_j\|_\cH$,
$$
\cQ_{n_j}p_j = \sum_{n=1}^{n_j} a_{n,j}\psi_n, \mbox{ so that } (\cQ_{n_j}-\cQ_{m,n_j})p_j = \sum_{n=1}^{m-1} a_{n,j} (\psi_n-\cQ_{m,n_j}\psi_n).
$$
Since $\{|a_{n,j}|:j,n\in \NN, \, n\leq m-1\}$ is bounded, it follows by Bolzano-Weierstrass and \eqref{eq:strongproj} that $((\cQ_{n_j}-\cQ_{m,n_j})p_j)_{j=1}^\infty$ has a convergent subsequence. Thus $\{\cQ_n-\cQ_{m,n}: n\geq m\}$ is collectively compact; moreover, orthogonal projection operators are self-adjoint (e.g., \cite[Theorem 12.14]{Rud}).  Since $(\psi_n)_{n=1}^\infty$ is weakly convergent to zero, it follows from Lemma \ref{lem:cc} that, for every $m\in \NN$,
\begin{equation} \label{eq:small}
(\cQ_n-\cQ_{m,n})\cA\psi_n \to 0 \quad \mbox{as} \quad n\to\infty.
\end{equation}
To complete the proof of \eqref{eq:Pn} we show, using \eqref{eq:close} and \eqref{eq:psiphi}, that
\begin{equation} \label{eq:Qmn}
\cQ_{m,n}\cA\psi_n \to 0 \quad \mbox{as} \quad m\to\infty,
\end{equation}
uniformly in $n\geq m$. It is easy to see that this, together with \eqref{eq:small}, implies \eqref{eq:Pn}.

To see that \eqref{eq:Qmn} holds, note first that, by \eqref{eq:close}, it is enough to show that $\cQ_{m,n}\cA\phi_n \to 0$ as $m\to\infty$,
uniformly in $n\geq m$.  For each $m,n\in \NN$ with $n\geq m$ let
$$
\chi_{m,n} := \sum_{j=m}^n a_{j} \phi_j,
$$
where $(a_{j})_{j=m}^n$ is the unique set of coefficients such that
\begin{equation} \label{eq:Qmncoeff}
\cQ_{m,n}\cA\phi_n = \sum_{j=m}^n a_{j} \psi_j,
\end{equation}
and let $\cR_{m,n}$ denote orthogonal projection onto the subspace $\cV_{m,n} := \mathrm{span}\{\phi_m,...,\phi_n\}$. Then, for $m,n\in \NN$ with $n\geq m$,
$$
\|\cQ_{m,n}\cA\phi_n\|_\cH^2 + \|\cA\phi_n -\cQ_{m,n}\cA\phi_n\|_\cH^2 = \|\cA\phi_n\|_\cH^2 = \|\cR_{m,n}\cA\phi_n\|_\cH^2 + \|\cA\phi_n -\cR_{m,n}\cA\phi_n\|_\cH^2
$$
and, since $\|\cR_{m,n}\cA\phi_n\|_\cH\leq \|\cR_{n}\cA\phi_n\|_\cH$ and $\|\cA\phi_n -\cR_{m,n}\cA\phi_n\|_\cH\leq \|\cA\phi_n -\chi_{m,n}\|_\cH$ (since $\cR_{m,n}\cA\phi_n$ is the best approximation to $\cA\phi_n$ from $\cV_{m,n}$), it follows that
\begin{eqnarray*}
\|\cQ_{m,n}\cA\phi_n\|_\cH^2 &\leq&    \|\cR_{n}\cA\phi_n\|_\cH^2 + (\|\cA\phi_n -\chi_{m,n}\|_\cH+\|\cA\phi_n -\cQ_{m,n}\cA\phi_n\|_\cH) \times\\
& & \hspace{15ex}   (\|\cA\phi_n -\chi_{m,n}\|_\cH-\|\cA\phi_n -\cQ_{m,n}\cA\phi_n\|_\cH)\\
&\leq&    \|\cR_{n}\cA\phi_n\|_\cH^2 + (2\|\cA\phi_n\|_\cH+\|\chi_{m,n} -\cQ_{m,n}\cA\phi_n\|_\cH) \|\chi_{m,n}-\cQ_{m,n}\cA\phi_n\|_\cH.
\end{eqnarray*}
Now, where the coefficients $\underline{a}=(a_m,...,a_n)$ are defined by \eqref{eq:Qmncoeff},
$$
 \|\chi_{m,n}-\cQ_{m,n}\cA\phi_n\|_\cH= \left\|\sum_{j=m}^na_j(\phi_j -\psi_j)\right\|_\cH \leq \frac{\|a\|_2}{2^{m-1}\sqrt{3}} \leq \frac{\|\cQ_{m,n}\cA\phi_n\|_\cH}{2^{m-1}(\sqrt{3}-1)},
$$
using \eqref{eq:psiphi}. Since $\|\cQ_{m,n}\cA\phi_n\|_\cH\leq \|\cA\phi_n\|_\cH\leq \|\cA\|_\cH$, we see that we have shown that
\begin{eqnarray*}
\|\cQ_{m,n}\cA\phi_n\|_\cH^2
&\leq&    \|\cR_{n}\cA\phi_n\|_\cH^2 + \|\cA\|_\cH^2\left(2+1/(\sqrt{3}-1)\right) \frac{1}{2^{m-1}(\sqrt{3}-1)}.
\end{eqnarray*}
But this right hand side tends to zero as $m\to\infty$, uniformly in $n\geq m$, by \eqref{eq:Rn}.

We have shown that \eqref{eq:Pn} holds, where, for each $n\in\NN$, $\cQ_n$ is orthogonal projection onto $\cW_n$. Thus \eqref{eq:lower} fails for the sequence $(\cW_n)_{n=1}^\infty$, and note also that, by definition of $(\psi_n)$, it holds for every $n$ that $\cW_n\subset \cH^*_m$ for some $m\in\NN$. To achieve our initial aim and complete the proof we augment our subspaces $\cW_n$ in such a way that \eqref{eq:lower} still fails while, additionally, our modified sequence of subspaces satisfies (b).

For $n,N\in \NN$ let $\cP_n^N$ denote orthogonal projection from $\cH$ onto the subspace $\cW_n^N:=\cH_N^*+\cW_n$, and $L_n^N$ orthogonal projection onto the orthogonal complement in $\cW_n^N$ of $\cW_n$, a subspace of $\cH_N^*$. Then $\cP_n^N= \cQ_n+L_n^N$ and, clearly, for each $N$, $\{L_n^N: n\in \NN\}$ is collectively compact, for it is uniformly bounded and the range of $L_n^N$ is contained in the finite-dimensional space $\cH_N^*$, for every $n\in \NN$. Thus, and since $\psi_n\rightharpoonup 0$ and each $L^N_n$ is self-adjoint, $L^N_n\cA\psi_n\to 0$ as $n\to\infty$, for every $N\in \NN$, by Lemma \ref{lem:cc}. It follows, using \eqref{eq:Pn}, that, for every $N\in \NN$,
$$
\cP_n^N \cA \psi_n =\cQ_n \cA \psi_n + L^N_n \cA \psi_n \to 0 \quad \mbox{as} \quad n\to \infty.
$$
Choose an increasing sequence $(n_N)_{N=1}^\infty\subset \NN$ such that $\|\cP_{n_N}^N \cA \psi_{n_N}\|\leq N^{-1}$, and set $\cH_N := \cW_{n_N}^N=\cH_N^*+\cW_{n_N}$ and $\cP_N :=  \cP_{n_N}^N$, for $N\in \NN$, so that $\cP_N$ is orthogonal projection onto $\cH_N$. Then $\cH_1\subset \cH_2\subset ...$, the sequence $(\cH_N)_{N=1}^\infty$ satisfies (b), and, recalling \eqref{eq:close}, we see that
$$
\inf_{\phi\in \cH_N\setminus\{0\}}\frac{\|\cP_N\cA\phi\|_\cH}{\|\phi\|_\cH}\leq \frac{\|\cP_N\cA\psi_{n_N}\|_\cH}{\|\psi_{n_N}\|_\cH}\leq \frac{N^{-1}}{1-2^{-n_N}}\to 0 \quad \mbox{as} \quad N\to\infty,
$$
so that \eqref{eq:lower} fails and the proof is complete.
\end{proof}

\subsection{Calculating the essential numerical range and essential norm}

Our proofs of Theorems \ref{thm:Q1}--\ref{thm:Q2Poly} depend on localisation results for the essential norm and essential numerical range of $D$ (Theorems \ref{lem:local} and \ref{thm:loccon}). These in turn depend on the following general Hilbert space results. The first of these is a useful characterisation of the essential norm (cf.\ Weyl's criterion (e.g., \cite[Theorem 7.2]{HiSi96}, \cite[Lemma 4.3.15]{BDav}) for membership of the essential spectrum).

\begin{lemma} \label{lem:ess} Let $\cH$ be a separable Hilbert space and $\cA$ a bounded linear operator on $\cH$. Then
$$
\|\cA\|_{\cH,\mathrm{ess}} = \sup_{(\phi_n)\in \cS} \,\liminf_{n\to\infty} \|\cA\phi_n\|_\cH= \sup_{(\phi_n)\in \cS} \limsup_{n\to\infty} \|\cA\phi_n\|_\cH,
$$
where $\cS$ is the set of sequences $(\phi_n)_{n=1}^\infty$ such that $\|\phi_n\|_\cH=1$ and $\phi_n\rightharpoonup 0$ as $n\to\infty$.
\end{lemma}
\begin{proof}
If $(\phi_n)\in \cS$ and $T$ is compact, then
$$
\|\cA+T\|_\cH  \geq \limsup_{n\to\infty}\|(\cA+T)\phi_n\|_\cH = \limsup_{n\to\infty}\|\cA\phi_n\|_\cH,
$$
since $T\phi_n\to 0$ as $n\to\infty$. Thus
\beq\label{eq:ess1}
\|\cA\|_{\cH,\mathrm{ess}} \geq  \sup_{(\phi_n)\in \cS} \limsup_{n\to\infty} \|\cA\phi_n\|_\cH.
\eeq
We now prove that
\beq\label{eq:ess2}
\|\cA\|_{\cH,\mathrm{ess}} \leq \sup_{(\phi_n)\in \cS} \liminf_{n\to\infty} \|\cA\phi_n\|_\cH,
\eeq
and the result follows from \eqref{eq:ess1} and \eqref{eq:ess2}.

Let $\cP_n$ be a sequence of orthogonal projection operators converging strongly to the identity, each with finite-dimensional range, and let $\cQ_n:= \cI-\cP_n$. Then, since $\cA \cP_n$ is compact, the definition of the essential norm implies that
$$
\|\cA\|_{\cH,\mathrm{ess}} \leq \liminf_{n\to\infty} \|\cA \cQ_n\|_\cH.
$$
By the definition of the norm, and since $\cQ_n$ is a projection operator so that $\|\cA \cQ_n\|_\cH$ coincides with the norm of $\cA$ restricted to the range of $\cQ_n$, we can find, for each $n$, $\phi_n$ in the range of $\cQ_n$ such that  $\|\phi_n\|_\cH=\|\cQ_n\phi_n\|_\cH= 1$ and $\|\cA \cQ_n\phi_n\|_\cH \geq \|\cA \cQ_n\|_\cH -1/n$.
Therefore
$$
\|\cA\|_{\cH,\mathrm{ess}} \leq \liminf_{n\to\infty} \|\cA \cQ_n\|_\cH = \liminf_{n\to\infty} \|\cA \cQ_n\phi_n\|_\cH.
$$
The inequality \eqref{eq:ess2} then follows from the facts that $\|\cQ_n\phi_n\|_\cH=1$ for all $n$ and $\cQ_n\phi_n \rightharpoonup 0$ as $n\to\infty$; indeed, if $\psi\in \cH$ then $|(\psi,\cQ_n\phi_n)_\cH| = |(\cQ_n\psi,\cQ_n\phi_n)_\cH|\leq \|\cQ_n\psi\|_\cH\to 0$, since $\cQ_n$ converges strongly to zero.
\end{proof}

The following lemma has something of the flavour of limit operator arguments and results (e.g., \cite{Li:06}, \cite[\S5.3]{ChLi11}). We apply this as a component in our localisation results (e.g.\ in Lemma \ref{lem:loccon}). Moreover, application of this lemma in combination with \eqref{eq:subspaces} leads to a rather direct attack on Q2$^{\prime\prime}$: to show that a particular set $\Delta$ is contained in $W_\ess(D)$ it is enough, by \eqref{eq:subspaces}, to show that $\Delta$ is contained in $W_\ess(D^\dag)$, where $D^\dag$ is $D$ restricted to a subset $\Gamma^\dag$ of $\Gamma$. If the following lemma applies with $\widetilde \cA=D^\dag$ then it is, moreover, enough to show that $\Delta \subset W(D^\dag)$ to conclude. And we detail methods  in \S\ref{sec:CompNR}, based on \eqref{eq:subspaces}, to show that a set $\Delta \subset W(D^\dag)$.

\begin{lemma}\label{lem:Tlemma} Let $\cH$ be a Hilbert space and $\cA$ a bounded linear operator on $\cH$. If there exists an isometry $\cT$ on $\cH$ (so that $(\cT \phi,\cT\psi)_{\cH} = (\phi,\psi)_\cH$ for all $\phi,\psi\in \cH$) such that (i) $\cT$ commutes with $\cA$ and (ii) $\cT^n\to 0$ weakly as $n\to\infty$ (i.e.~ $\cT^n\phi\rightharpoonup 0$ for every $\phi\in \cH$), then $W_\ess(\cA) = \overline{W(\cA)}$ and $\|\cA\|_{\cH,\ess} = \|\cA\|_\cH$. If also $\cV$ is a closed subspace of $\cH$ and $\cQ\cT^n\to 0$ strongly as $n\to\infty$ (i.e.~ $\cQ\cT^n\phi\to 0$ for every $\phi\in \cH$), where $\cQ:=\cI-\cP$ and $\cP:\cH\to \cV$ is orthogonal projection, then $W_\ess(\widetilde \cA) = \overline{W(\widetilde \cA)}= \overline{W(\cA)}$ and $\|\widetilde \cA\|_{\cV,\ess} = \|\widetilde \cA\|_\cV=\|\cA\|_\cH$, where $\widetilde \cA := \cP\cA|_\cV$.
\end{lemma}

To get a concrete sense of situations where Lemma \ref{lem:Tlemma} can be applied (this is similar to the application we make in Lemma \ref{lem:loccon}), consider the case (e.g., \cite{BottSilb98,Li:06}) where $\cH=\ell^2(\Z)$, and $\cA$ and $\cT=\cA$ are both the right shift operator. Clearly, $\cT$ is an isometry and $\cA$ commutes with $\cT$. Further, $\cT^n\to 0$ weakly, so that, applying this lemma, $W_\ess(\cA) = \overline{W(\cA)}$. Since $\cA=\cT$ is a normal operator (its adjoint is the left shift operator) we have additionally \cite[\S1.4]{GuRa:97} that $\overline{W(\cA)}=\conv(\spec(\cA))$, which is the closed unit disk $\{z:|z|\leq 1\}$. Set $\cV= \ell^2(\NN)\subset \ell^2(\Z)$, $\cP:\cH\to \cV$ to be orthogonal projection, and $\widetilde \cA := \cP\cA|_\cV$. Then $\widetilde \cA$ (which is non-normal) is the restriction of $\cA$ to $\ell^2(\NN)$ that acts by multiplication by the infinite Toeplitz matrix which is zero except for $1$'s on the first subdiagonal. We see, by applying the second part of the lemma, that $W_\ess(\widetilde \cA)$ and $\overline{W(\widetilde \cA)}$ are also the closed unit disk.

\

\bpf[Proof of Lemma \ref{lem:Tlemma}]
We prove the statements about the numerical range and essential numerical range; the proofs of the statements about the norm and essential norm are analogous, using that
$$
\|\cA\|_\cH = \sup_{\phi, \,\psi\in \cH\setminus\{0\}} \frac{|(\cA\phi,\psi)_\cH|}{\|\phi\|_\cH\|\psi\|_\cH}.
$$

The definition of the essential numerical range implies immediately that $W_\ess(\cA)\subset \overline{W(\cA)}$ for any bounded linear operator $\cA$.
Conversely, suppose that $z\in W(\cA)$. Then, for some $\phi\in \cH$ with $\|\phi\|_\cH=1$, and all $n\in \mathbb{N}$,
\beqs
z = (\cA\phi,\phi)_\cH = (\cT^n\cA\phi,\cT^n\phi)_\cH = (\cA \cT^n\phi,\cT^n\phi)_\cH,
\eeqs
where we have used the facts that $\cT$ is an isometry and commutes with $\cA$.
Given a compact operator $\cK$, let
\beqs
z_n := \big((\cA+\cK)\cT^n\phi,\cT^n\phi\big)_\cH = z + (\cK \cT^n\phi,\cT^n\phi)_\cH.
\eeqs
Since $\cT^n \phi\rightharpoonup 0$, we have $\cK \cT^n \phi \tendo$, so $z_n\to z$. But $z_n\in W(\cA+\cK)$ for each $n$, so $z\in \overline{W(\cA+\cK)}$. Since this holds for all compact $\cK$, $z\in W_\ess(\cA)$. We have therefore shown that
$W(\cA)\subset W_\ess(\cA)$, so $\overline{W(\cA)} \subset W_\ess(\cA)$ since $W_\ess(\cA)$ is closed.

Now suppose that $\cV$ is a closed subspace of $\cH$ and that $\cQ\cT^n\to 0$ strongly as $n\to\infty$, where $\cQ:=\cI-\cP$ and $\cP:\cH\to \cV$ is orthogonal projection. We have that $W_\ess(\widetilde \cA) \subset \overline{W(\widetilde \cA)} \subset \overline{W(\cA)}$ by definition and \eqref{eq:subspaces}. To see that also $\overline{W(\cA)} \subset W_\ess(\widetilde \cA)$, suppose that $z\in W(\cA)$, in which case $z=(\cA\phi,\phi)_\cH$, for some $\phi\in \cH$ with $\|\phi\|_\cH=1$. Then $\cQ\cT^n\phi \to 0$ as $n\to \infty$ so that
\begin{equation} \label{eq:PTn}
\|\cP\cT^n\phi\|_\cH = \|\cT^n\phi\|_\cH + o(1) = \|\phi\|_\cH + o(1) =1 + o(1), \quad \mbox{as } n\to\infty.
\end{equation}
Thus, for some $N\in \NN$, $\cP\cT^n\phi \neq 0$ for $n\geq N$, so that, given any compact $\cK:\cV\to\cV$,
$$
z_n := \frac{((\widetilde \cA+\cK)\cP\cT^n\phi,\cP\cT^n\phi)_\cH}{\|\cP\cT^n\phi\|_\cH^2} \in W(\widetilde \cA+\cK), \quad n\geq N.
$$
Further, as $n\to \infty$, using \eqref{eq:PTn}, the definition of $\widetilde A$, that $\cQ\cT^n\phi \to 0$, and that $\cT^n\phi\rightharpoonup 0$ so that $\cK \cP\cT^n\phi \to 0$ (since $\cK\cP:\cH\to \cV$ is compact), it follows that
\begin{eqnarray*}
z_n = ((\widetilde \cA+\cK)\cP\cT^n\phi,\cP\cT^n\phi)_\cH + o(1)&=&(\cA\cP\cT^n\phi,\cP\cT^n\phi)_\cH + (\cK\cP\cT^n\phi,\cP\cT^n\phi)_\cH + o(1)\\
&=& (\cA\cT^n\phi,\cT^n\phi)_\cH + o(1)= z + o(1),
\end{eqnarray*}
since $\cA$ commutes with $\cT$ and $\cT$ is unitary.
Thus $z\in \overline{W(\widetilde \cA+\cK)}$. Since this holds for every compact $\cK$, $z\in W_\ess(\widetilde A)$. Thus $W(\cA)\subset W_\ess(\widetilde \cA)$; indeed $\overline{W(\cA)}\subset W_\ess(\widetilde \cA)$ since $W_\ess(\widetilde \cA)$ is closed.
\epf

\subsection{Calculating the numerical range and norm} \label{sec:CompNR}

Question Q1 requires estimation of the essential norm of $D$ and Q2$^{\prime\prime}$ requires determining membership of its  essential numerical range. Our localisation results in \S\ref{sec:localisation} below (and see Lemma \ref{lem:Tlemma} above) reduce these questions to determining norms and membership of the numerical range for local representatives of $D$. In this section we present  the results that we use for these calculations.

Expanding on \eqref{eq:subspaces}, suppose that $(\cH_N)_{N=1}^\infty$ is a sequence of finite-dimensional subspaces of the Hilbert space $\cH$, $\cP_N:\cH\to \cH_N$ is orthogonal projection and $\widetilde \cA_N := \cP_N \cA|_{\cH_N}$. Then
$$
W(\cA) \supset W(\widetilde \cA_N) = \{(\cA\psi,\psi)_\cH : \psi\in \cH_N, \; \|\psi\|_\cH=1\}.
$$
Further, suppose that $\dim(\cH_N)=N$ and  $\{\psi_1,...,\psi_N\}$ is an orthonormal basis for $\cH_N$, and construct the matrix $A_N\in \C^{N\times N}$ by
\begin{equation} \label{eq:ANGendef}
(A_N)_{jm} := (\cA\psi_m,\psi_j)_\cH, \quad 1\leq j,m\leq N.
\end{equation}
(Note that this is precisely the Galerkin matrix that we obtain on rewriting the Galerkin equations \eqref{eq:G} as a linear system using the basis $\{\psi_1,...,\psi_N\}$.)
Then, for every $\ba=(a_1,...,a_N)^T\in \C^{N\times 1}$, where $(\cdot,\cdot)_2$ denotes the standard inner product on $\C^N$ and $\|\cdot\|_2$ the associated norm (and also the induced operator norm),
$$
\ba^H A_N \ba = (A_N\ba,\ba)_2 = (\cA\phi,\phi)_{\cH}, \quad \|A_N\ba\|_2 = \|\cA\phi\|_{\cH}, \quad \mbox{and} \quad \|\ba\|_2 = \|\phi\|_{L^2(\Gamma)},
$$
where $\phi := \sum_{j=1}^N a_j \psi_j$, so that
\begin{equation} \label{eq:ANIncl}
\|A_N\|_2 = \|\widetilde \cA_N\|_\cV \leq \|\cA\|_\cH  \quad \mbox{and} \quad  W(A_N) = W(\widetilde \cA_N) \subset W(\cA).
\end{equation}
Further, if $\cH_1\subset \cH_2 \subset ...$ and $(\cH_N)_{N=1}^\infty$ is asymptotically dense in $\cH$, i.e.
$
\cH = \overline{\cup_{N=1}^\infty \cH_N},
$
then
\begin{equation} \label{eq:NRlimits}
\|A_1\|_2 \leq \|A_2\|_2 \leq ..., \;\;\; W(A_1)\subset W(A_2) \subset ..., \;\;\; \|\cA\|_\cH = \lim_{N\to\infty} \|A_N\|_2, \;\;\; \overline{W(\cA)} = \overline{\cup_{N=1}^\infty W(A_N)};
\end{equation}
(see, e.g., \cite[Theorem 9.3.4]{BDav} for this last result).

The equations \eqref{eq:NRlimits} reduce computing the norm and numerical range of an operator to computing the limits of norms and numerical ranges of finite matrices. In particular, it turns out that we need later to compute the norms and numerical ranges of real $N\times N$ matrices $C_N$ that have entries given by
\begin{equation} \label{eq:CNdef}
(C_N)_{jm} = (B_N)_{jm} e_{jm}, \quad 1\leq j,m \leq N,
\end{equation}
where, for some $e_0\geq 1$,
\begin{equation} \label{eq:ejm}
0< e_{jm} \leq e_0, \quad 1\leq j,m \leq N,
\end{equation}
and $B_N\in \R^{N\times N}$ has entries
\beq\label{eq:aBdef}
\big( B_N\big)_{jm}:= (-1)^{m+1} \sign(m-j), \quad 1\leq j,m\leq N,
\eeq
where $\mathrm{sign}(s) := 1$ for $s>0$, $:=-1$ for $s<0$, and $:=0$ for $s=0$. For a square matrix $E$, let $w_r(E)$ denote the numerical abscissa of $E$, i.e.\ $w_r(E) := \sup_{z\in W(E)} \Re(z)$. Since the numerical range is convex,
\begin{equation} \label{eq:WCN}
W(C_N) = \bigcap_{0\leq \theta \leq 2\pi} \left\{z: \Re(\re^{\ri\theta}z) \leq w_r\left( \re^{\ri \theta} C_N\right)\right\},
\end{equation}
and $w_r(\re^{\ri \theta} C_N)$, the numerical abscissa of $\re^{\ri \theta} C_N$, is also the largest eigenvalue of the Hermitian matrix $C_N^\theta := (\re^{\ri\theta}C_N + \re^{-\ri\theta}C_N^T)/2$ (e.g., \cite[Theorem 9.3.10]{BDav}).

\begin{lemma} \label{lem:CNprops} If $C_N\in \R^{N\times N}$ is defined by \eqref{eq:CNdef}, with $B_N$ defined by \eqref{eq:aBdef} and $e_{jm}$ satisfying \eqref{eq:ejm}, then:
\begin{itemize}
\item[(i)] $w(C_N)\leq \|C_N\|_2 \leq e_0(N-1)$.
\item[(ii)] $W(C_N)$ is symmetric about the real axis, so that $w_r(\re^{\ri \theta} C_N)=w_r(\re^{-\ri \theta} C_N)$, $\theta\in \R$.
\item[(iii)] If $e_{jm}=e_{mj}$ whenever $m-j$ is even, then
$$
W(C_N)\cap \R = [-w_r(C_N),w_r(C_N)]=[-\|C_N^0\|_2,\|C_N^0\|_2],
$$
so that $W(C_N)\cap\R \supset  [-a,a]$ for $N\geq 3$, with equality when $N=3$, where
$$
a := \half\sqrt{(e_{12}+e_{21})^2+(e_{23}+e_{32})^2}.
$$
\item[(iv)] If  $e_{jm}=e_{mj}$ for all $1\leq j,m\leq N$, then, for $\theta \in \R$,
\begin{equation} \label{eq:CNtheta}
(C^\theta_N)_{jm} = \frac{\sign(m-j)e_{jm}}{2}\left(\re^{-\ri \theta}(-1)^j-\re^{\ri \theta}(-1)^m\right), \quad 1\leq j,m \leq N,
\end{equation}
and $w_r(\re^{\ri \theta} C_N)=w_r(\re^{\ri (\theta+\pi)} C_N)= \|C_N^\theta\|_2$.
\item[(v)] If, for some $\epsilon >0$ and $N\geq 2$, $|e_{jm}-1|\leq \epsilon/(N-1)$ for $1\leq j,m\leq N$, then
\begin{equation} \label{eq:bigcapINCL}
\bigcap_{0\leq \theta \leq 2\pi} \left\{z: \Re(\re^{\ri\theta}z) \leq w_r\left( \re^{\ri \theta} B_N\right)-\epsilon\right\} \subset W(C_N)
\end{equation}
and
$$
W(C_N) \supset [\epsilon-b, b-\epsilon],
$$
where $b= \|B_N^0\|_2$.
\end{itemize}
\end{lemma}
\begin{proof}
Item (i) follows by \eqref{eq:NormEssNorm} and since the maximum row and column sums are both $\leq e_0(N-1)$. Item (ii) holds since the entries of $C_N$ are real, so that if $\lambda = \ba^H C_N  \ba$ then $\bar \lambda = \bar \ba^H C_N \bar \ba$.

\sloppy To see (iii), note that, by (ii) and since $W(C_N)$ is closed and convex, $W(C_N)\cap \R = [-w_r(-C_N), w_r(C_N)]$. Further, if $e_{jm}=e_{mj}$ when $j-m$ is even, and $\lambda$ is an eigenvalue of $C_N^0$ with eigenvector $\ba=(a_1,...,a_N)^T$, then $-\lambda$ is also an eigenvalue, with eigenvector $\bb=(b_1,...,b_N)^T$, where $b_j := (-1)^ja_j$, $j=1,...,N$. Thus, and since $w_r(C_N)$ and $-w_r(-C_N)$ are the largest and smallest eigenvalues of $C_N^0$, it follows that $\|C_N^0\|_2=w_r(-C_N)=w_r(C_N)$. The final claim of (iii) follows from \eqref{eq:subspaces} and since $a$ is the largest eigenvalue of $C_N^0$ when $N=3$.

The formula \eqref{eq:CNtheta} is immediate from the definition of $C_N^\theta$, and the rest of (iv) follows from the observation that, if  $e_{jm}=e_{mj}$,  $1\leq j,m\leq N$, and $\lambda$ is an eigenvalue of $C_N^\theta$ with eigenvector $\ba=(a_1,...,a_N)^T$, then $-\lambda$ is also an eigenvalue, with eigenvector $\bb=(b_1,...,b_N)^T$, where $b_j := (-1)^j\overline{a_j}$, $j=1,...,N$.

If, for some $\epsilon >0$ and $N\geq 2$, $|e_{jm}-1|\leq \epsilon/(N-1)$ for $1\leq j,m\leq N$, then $\|B_N-C_N\|_2 \leq \epsilon$ by part (i). We have \eqref{eq:CNtheta} and that $w_r(\re^{\ri \theta}C_N)$ and $w_r(\re^{\ri \theta}B_N)$ are the largest eigenvalues of the Hermitian matrices $C_N^\theta$ and $B_N^\theta$, respectively. Since $\|C_N^\theta-B_N^\theta\|_2 \leq \epsilon$, it follows that $w_r(\re^{\ri \theta}C_N)\geq w_r(\re^{\ri \theta}B_N)-\epsilon$, and \eqref{eq:bigcapINCL} follows by \eqref{eq:WCN}. Since also, by (iii), $W(B_N)\cap \R=[-b,b]$ where $b= \|B_N^0\|_2=w_r(\pm B_N)$ and, as a consequence of (ii), $W(C_N)\cap \R=[-w_r(-C_N),w_r(C_N)]$, it follows also that $[\epsilon-b, b-\epsilon]\subset W(C_N)$.
\end{proof}

Part (v) of the above lemma relates properties of $C_N$ to those of $B_N$ when $\|C_N-B_N\|_2$ is small. This is helpful as we can compute properties of $B_N$ explicitly, which we do in the following lemma. We see later that matrices of the form $C_N/2$  serve as approximations, in some sense, to the double-layer potential operator $D$ on $L^2(\Gamma)$, and that $C_N$ is approximated by $B_N$ for certain limiting geometries, so that $B_N/2$ is, in some sense, an approximation for $D$. This motivates the calculation of $\spec(B_N)$, though we do not use those calculations hereafter.

\ble[Properties of $B_N$]\label{lem:AN}

\

\begin{itemize}
\item[(i)] $W(B_1)\subset W(B_2) \subset \ldots$.
\item[(ii)] If $N\geq 3$ is odd then $w_r(B_N)\geq \sqrt{(N+1)/2}$, with equality when $N=3$, and $w_r(\re^{\ri \theta}B_N) \geq \sqrt{(N-1)/2}$, $0\leq \theta \leq 2\pi$, so that
\beqs
\conv\left(\left[-\sqrt{(N+1)/2},\sqrt{(N+1)/2}\right]\bigcup \left\{ z \in \Com: |z|< \sqrt{\frac{N-1}{2}}\right\}\right)
 \subset W(B_N).
\eeqs
\item[(iii)] If $N\geq 2$ then
$$
\mathrm{spec}(B_N) = \left\{\begin{array}{cc} \{-1,1\}, & \mbox{ if } N \mbox{ is even},\\ \{-1,0,1\}, & \mbox{ if } N \mbox{ is odd}.\end{array}\right.
$$
\end{itemize}
\ele
\bpf
Item (i) is an instance of \eqref{eq:subspaces}. From Lemma \ref{lem:CNprops}(iv), $w_r(\re^{\ri \theta}B_N) = \|B_N^\theta\|_2$, $0\leq \theta \leq 2 \pi$. The $2$-norm of $B_N^\theta$ is no smaller than the $2$-norm of any of its columns. When $N\geq 3$ is odd, using \eqref{eq:CNtheta} with $e_{jm}=1$, we see that the $2$-norm of the first column of $B_N^\theta$ is $\sqrt{(N-1)/2}$ and the $2$-norm of the second column of $B_N^0$ is $\sqrt{(N+1)/2}$, so that $\|B_N^0\|_2\geq \sqrt{(N+1)/2}$, with equality when $N=3$ by Lemma \ref{lem:CNprops}(iii). Thus (ii) follows since $W(B_N)\cap \R = [-\|B_N^0\|_2,\|B_N^0\|_2]$, by Lemma \ref{lem:CNprops}(iii), and using \eqref{eq:WCN} and that the numerical range is convex.

To see (iii), let $B_N(\lambda) = \lambda I_N - B_N$. Then, for $N=3,4,...$, $\det(B_N(\lambda)) = (\lambda^2-1)\det(B_{N-2}(\lambda))$, so that
$$
\det(\lambda I_N-B_N) = \left\{\begin{array}{cc} (\lambda^2-1)^\nu, & \mbox{ if } N=2\nu,\\ \lambda(\lambda^2-1)^\nu, & \mbox{ if } N=2\nu+1,\end{array}\right.
$$
for $\nu\in \mathbb{N}$, and the result follows.
\epf

Combining Lemma \ref{lem:CNprops}(v) and Lemma \ref{lem:AN}(ii), we obtain the following corollary.
\begin{corollary} \label{cor:CN} If $N\geq 3$ is odd and, for some $\epsilon >0$, $|e_{jm}-1|\leq \epsilon/(N-1)$ for $1\leq j,m\leq N$, then
\begin{equation} \label{eq:bigcapINCL2}
\left\{ z \in \Com: |z|\leq \sqrt{\frac{N-1}{2}}-\epsilon\right\} \subset W(C_N)
\end{equation}
and
$$
\left[\epsilon-\sqrt{(N+1)/2}, \sqrt{(N+1)/2}-\epsilon\right] \subset W(C_N).
$$
\end{corollary}

\section{Localisation of the essential norm and essential numerical range of the double-layer operator}

\subsection{Localisation arguments}\label{sec:localisation}

In this section we prove a localisation result (Theorem \ref{lem:local}) for the essential norm and essential numerical range, adapting and extending localisation arguments used by I.~Mitrea \cite[Lemma 1, Page 392]{Mi:99} who showed (using the notations below, including  \eqref{eqn:defDbxdelta} and that $B_\delta(\bx)$ is the open ball of radius $\delta$ with centre $\bx$) that  if, for some $\delta>0$ and $\lambda\in \C$, $\la P_\delta(\bx) - D_{\bx,\delta}$ is Fredholm of index zero on $L^2(B_\delta(\bx))$, for every $\bx\in \Gamma$, then $\la I- D$ is Fredholm of index zero on $\LtG$. Key tools in our arguments are the Weyl-type characterisation of the essential norm, Lemma \ref{lem:ess} above,
and that $D$ can be written as its local action plus a compact operator (see \eqref{eqn:Dphil2}), thanks to the compactness of the commutator $[w,D]$ when $w$ is a Lipschitz continuous function on $\Gamma$. We then specialise this localisation result, including to the case when $\Gamma$ is Lipschitz polyhedral, in Theorem \ref{thm:loccon} below.

\begin{definition}[Lipschitz constant and Lipschitz character]\label{def:Lipschitz}
A bounded Lipschitz domain $\Omega_-$ is {\em Lipschitz with constant $M$} if, for any $\bx\in \Gamma$, locally near $\bx$ the boundary $\Gamma$ is (in some rotated coordinate system) the graph of a Lipschitz function with Lipschitz constant $M$. The infimum of all $M\geq 0$ for which $\Omega_-$ is Lipschitz with constant $M$ is called the {\em Lipschitz character} of $\Omega_-$.
\end{definition}

For $\delta>0$ and $\bx\in \Gamma$, let $P_\delta(\bx)$ be the orthogonal projection operator on $L^2(\Gamma)$ that is multiplication by the characteristic function of $B_\delta(\bx)$ restricted to $\Gamma$.

\begin{theorem}[Localisation result for general Lipschitz case] \label{lem:local} Suppose that $\Omega_-$ is a bounded Lipschitz domain and let
\beq\label{eqn:defDbxdelta}
D_{\bx,\delta}:=P_\delta(\bx)D P_\delta(\bx).
\eeq
\noi (a) For some $\bx^*\in \Gamma$,
\beq\label{eqn:locall}
\|D\|_{L^2(\Gamma),\mathrm{ess}}= \lim_{\delta\to 0} \,\sup_{\bx\in \Gamma} \|D_{\bx,\delta}\|_{L^2(\Gamma)} = \sup_{\bx\in \Gamma} \,\lim_{\delta\to 0}  \|D_{\bx,\delta}\|_{L^2(\Gamma)} =  \lim_{\delta\to 0}  \|D_{\bx^*,\delta}\|_{L^2(\Gamma)}.
\eeq

\noi (b) For some $\bx^*\in \Gamma$,
\begin{align} \nonumber
W_{\mathrm{ess}}(D)= \bigcap_{\delta> 0} \mathrm{conv}\left(\overline{\bigcup_{\bx\in \Gamma} W(D_{\bx,\delta})}\right)
&= \mathrm{conv}\left(\bigcap_{\delta> 0} \overline{\bigcup_{\bx\in \Gamma}W(D_{\bx,\delta})}\right)
\\
&= \bigcup_{\bx\in \Gamma}\bigcap_{\delta> 0} \overline{W(D_{\bx,\delta})}=\bigcap_{\delta>0}\overline{ W(D_{\bx^*,\delta})}.
\label{eqn:erlocall}
\end{align}
\end{theorem}

\begin{proof}
(a)
The definition \eqref{eqn:defDbxdelta} implies that,
\beq\label{eq:inclusionpropD}
\text{if } \bx,\bx'\in \Gamma \tand B_{\delta}(\bx)\subset B_{\delta'}(\bx'), \text{ then }
\|D_{\bx,\delta}\|_{L^2(\Gamma)} \leq \|D_{\bx',\delta'}\|_{L^2(\Gamma)}.
\eeq
In particular
 $\|D_{\bx,\delta}\|_{L^2(\Gamma)}\leq \|D_{\bx,\delta^\prime}\|_{L^2(\Gamma)}$ for $0<\delta<\delta^\prime$ and $\bx\in \Gamma$, so that the limits in \eqref{eqn:locall} are well-defined. To see the last two equalities in \eqref{eqn:locall}, for $\delta>0$ let
\beq\label{eq:Cdelta}
C_\delta:=\sup_{\bx\in \Gamma} \|D_{\bx,\delta}\|_{L^2(\Gamma)},
\eeq
and let
$$
L_1 := \lim_{\delta\to 0}C_\delta = \lim_{\delta\to 0}\,\sup_{\bx\in \Gamma} \|D_{\bx,\delta}\|_{L^2(\Gamma)}, \quad L_2 := \sup_{\bx\in \Gamma} \,\lim_{\delta\to 0}  \|D_{\bx,\delta}\|_{L^2(\Gamma)}.
$$
Now choose a sequence $\delta_n>0$ such that $\delta_n\to 0$ as $n\to\infty$, and choose a sequence of points $\bx_n\in \Gamma$ such that
$$
L_1 = \lim_{n\to\infty}\|D_{\bx_n,\delta_n}\|_{L^2(\Gamma)}.
$$
By passing to a subsequence if necessary, we can assume also that, for some $\bx^*\in \Gamma$, $\bx_n\to \bx^*$ as $n\to\infty$. Then \eqref{eq:inclusionpropD} implies that, if $\epsilon_n := \delta_n + |\bx_n-\bx^*|$, then $\|D_{\bx_n,\delta_n}\|_{L^2(\Gamma)}\leq \|D_{\bx^*,\epsilon_n}\|_{L^2(\Gamma)}$ for all $n$ so that
\beq\label{eqn:intermediate}
L_1 \leq \lim_{n\to\infty}\|D_{\bx^*,\epsilon_n}\|_{L^2(\Gamma)} \leq  L_2.
\eeq
Similarly, we can choose a sequence of points $\bx_n\in \Gamma$ and an $\bx^\dag\in \Gamma$ such that $\epsilon_n:=|\bx^\dag-\bx_n|\to 0$ as $n\to\infty$ and
\begin{eqnarray*}
L_2 &=& \lim_{n\to\infty}\, \lim_{\delta\to 0}\|D_{\bx_n,\delta}\|_{L^2(\Gamma)} \leq \lim_{n\to\infty}\, \lim_{\delta\to 0}\|D_{\bx^\dag,\delta+\epsilon_n}\|_{L^2(\Gamma)} \leq
L_1.
\end{eqnarray*}
Thus $L_1=L_2$, i.e.~the second equality in \eqref{eqn:locall} holds, and then \eqref{eqn:intermediate} implies that the last equality in \eqref{eqn:locall} holds.

We now prove the first inequality in \eqref{eqn:locall}.
Suppose that $\delta>0$, choose $\varepsilon<\delta/3$ and, for $\bx\in \Gamma$, let $U_\bx:= B_\varepsilon(\bx)$. Let $G$ be a finite subset of $\Gamma$ such that $\mathcal{U}:= \{U_\bx:\bx \in G\}$ is a cover for $\Gamma$. For $\bx\in G$, let  $V_\bx$ denote the union of all those $U_\by\in \mathcal{U}$ that intersect $U_\bx$, note that $V_\bx \subset B_{\delta}(\bx)$, and let $\eta_\bx\in L^\infty(\Gamma)$ denote the characteristic function of $V_\bx \cap \Gamma$.  Let $(\chi_\bx)_{\bx\in G}$ be a partition of unity for $\Gamma$ subordinate to $\mathcal{U}$, so that $\chi_\bx\in C_0^\infty(\R^d)$, $0\leq \chi_\bx\leq 1$, $\chi_\bx$ is supported in $U_\bx$, and
\begin{equation} \label{eqn:pou2}
\sum_{\bx\in G}\chi_\bx = 1 \qquad \mbox{on } \Gamma.
\end{equation}
Note that $w_\bx:=\sqrt{\chi_\bx}$ is at least Lipschitz continuous\footnote{This claim follows easily from the following one-dimensional result: if $F\in C^2(\R)$ and, for some $m\geq 0$, $F(x)\geq 0$ and $|F^{\prime\prime}(x)|\leq m$ for $x\in \R$, then, arguing as in \cite[Lemma 1]{Gl:63}, $(F^\prime(x))^2\leq 2mF(x)$, $x\in \R$. This implies that $G(x):= \sqrt{F(x)}$, $x\in \R$, is Lipschitz continuous with Lipschitz constant $L:=\sqrt{m/2}$. For $G\in C(\R)$ and is  differentiable at every $x$ where $F(x)>0$, with $|G^\prime(x)|\leq L$. Thus, if $a<b$ and $F(x)>0$ for $a< x< b$, $|G(b)-G(a)|\leq L|b-a|$ by the mean value theorem (MVT). If $F(x)=0$ for some $x\in(a,b)$ then there exist $c,d\in [a,b]$ with $a\leq c\leq d\leq b$, $F(c)=F(d)=0$, and $F(x)>0$ for $a<x<c$ and $d<x<b$ so that, again by the MVT, $|G(b)-G(a)| \leq |G(c)-G(a)| + |G(b)-G(d)| \leq L|c-a|+L|b-d| \leq L|b-a|$.}. We observe further that if $w\in C^{0,1}(\Gamma)$ then the commutator $[w,D]$, defined by $[w,D]\phi:=wD\phi-D(w\phi)$ for $\phi\in L^2(\Gamma)$, is compact since it is an integral operator  on $L^2(\Gamma)$ with a weakly-singular kernel (e.g.,~\cite[Theorems 2.29 and 4.13]{Kr:14}).

For $\phi\in L^2(\Gamma)$,
\begin{equation} \label{eqn:Dphil2}
D\phi = \sum_{\bx\in G} D(\chi_\bx\phi)= \sum_{\bx\in G} w_\bx D (w_\bx\phi) -\sum_{\bx\in G } [w_\bx,D] (w_\bx\phi) = \sum_{\bx\in G } w_\bx D(w_\bx\phi) + T\phi,
\end{equation}
where $T$ is compact.
Now
\begin{eqnarray} \nonumber
\big((D-T)\phi,D\phi\big)_{L^2(\Gamma)} &=& \sum_{\bx\in G }\big(w_\bx D(w_\bx\phi),D\phi\big)\\ \nonumber
&=& \sum_{\bx\in G }\big(\eta_\bx D(w_\bx\phi), D(w_\bx\phi) + [w_\bx,D]\phi\big)\\
&= & \sum_{\bx\in G }\|\eta_\bx D(w_\bx\phi)\|^2_{L^2(\Gamma)} + \widetilde{T}(\phi),\label{eqn:bound1a}
\end{eqnarray}
where
$$
\widetilde{T}(\psi):= \sum_{\bx\in G }\big(\eta_\bx D(w_\bx\psi),[w_\bx,D]\psi\big) \quad \tfor \psi \in L^2(\Gamma).
$$
Since $V_\bx\cap \Gamma\subset B_\delta(\bx)$,
$$
\|\eta_\bx D(w_\bx\phi)\|_{L^2(\Gamma)} \leq \|D_{\bx, \delta}\|_{L^2(\Gamma)}\|w_\bx\phi\|_{L^2(\Gamma)}\leq C_\delta \|w_\bx\phi\|_{L^2(\Gamma)},
$$
where $C_\delta$ is defined by \eqref{eq:Cdelta}.
Using this last bound in \eqref{eqn:bound1a}, we have
\beq\label{eqn:bound2}
\left|((D-T)\phi,D\phi)_{L^2(\Gamma)}\right|\leq  C_\delta^2\sum_{\bx\in G } \|w_\bx\phi\|_{L^2(\Gamma)}^2 + |\widetilde{T}(\phi)| =  C_\delta^2\|\phi\|_{L^2(\Gamma)}^2 + |\widetilde{T}(\phi)|.
\eeq
Now suppose that $(\phi_n)\subset \cS$, where $\cS$ is defined in Lemma \ref{lem:ess}. Then, since $\|\phi_n\|=1$, $\phi_n \rightharpoonup 0$, and $T$ and each $[w_\bx,D]$ are compact, $T\phi_n\to 0$ and $\widetilde{T}(\phi_n)\to 0$. Thus, by \eqref{eqn:bound2},
$$
\limsup_{n\to\infty}\|D\phi_n\|_{L^2(\Gamma)}^2 \leq  C_\delta^2,
$$
and by Lemma \ref{lem:ess},
$\|D\|_{L^2(\Gamma), \,\mathrm{ess}} \leq C_\delta$. Since we have shown this for all sufficiently small $\delta$, and $L_1:=\lim_{\delta\tendo}C_\delta$, we have
$$
\|D\|_{L^2(\Gamma), \,\mathrm{ess}} \leq L_1.
$$
Now let $\bx^*\in \Gamma$ be such that the last equality in \eqref{eqn:locall} holds and choose a sequence $\delta_n>0$ such that $\delta_n\to 0$ as $n\to\infty$. Then
$$
L_1 = \lim_{n\to\infty}\|D_{\bx^*,\delta_n}\|_{L^2(\Gamma)} = \lim_{n\to\infty}\|D_{\bx^*,\delta_n}\phi_n\|_{L^2(\Gamma)},
$$
for some $\phi_n\in L^2(\Gamma)$, supported in $B_{\delta_n}(\bx^*)$, such that $\|\phi_n\|_{L^2(\Gamma)}=1$.
Since the support of $\phi_n$ is contained in $B_{\delta_n}(\bx^*)$, $\phi_n \rightharpoonup 0$ as $n\to \infty$. Thus, by Lemma \ref{lem:ess},
$$
\|D\|_{L^2(\Gamma), \,\mathrm{ess}} \geq \limsup_{n\to\infty}\|D\phi_n\|_{L^2(\Gamma)} \geq \lim_{n\to\infty}\|D_{\bx^*,\delta_n}\phi_n\|_{L^2(\Gamma)} = L_1,
$$
and we have proved the first inequality in \eqref{eqn:locall}.

(b) Similar to \eqref{eq:inclusionpropD}, the definition of $D_{\bx,\delta}$, \eqref{eqn:defDbxdelta}, implies that
\beq\label{eq:inclusionpropW}
\text{if } \bx,\bx'\in \Gamma \tand B_{\delta}(\bx)\subset B_{\delta'}(\bx'), \text{ then }
W(D_{\bx,\delta}) \subset W(D_{\bx',\delta'}),
\eeq
so that, in particular,  $W(D_{\bx,\delta}) \subset W(D_{\bx,\delta'})$ for $0<\delta<\delta^\prime$ and $\bx\in \Gamma$.

For $\delta >0$ and $\bx\in \Gamma$ let
$$
S_\delta := \overline{\bigcup_{\bx\in \Gamma}W(D_{\bx,\delta})}, \quad S_\infty := \bigcap_{\delta>0} S_\delta, \quad \mbox{and} \quad W_\bx := \bigcap_{\delta>0}\overline{W(D_{\bx,\delta})}.
$$
With these notations, we can abbreviate \eqref{eqn:erlocall} as
\beq \label{eqn:erlocall2}
W_{\mathrm{ess}}(D)= \bigcap_{\delta> 0} \mathrm{conv}\left(S_\delta\right) = \mathrm{conv}\left(S_\infty\right)= \bigcup_{\bx\in \Gamma}W_\bx=W_{\bx^*}.
\eeq
Observe that
$S_\delta \subset S_{\delta^\prime}$ for $0<\delta <\delta^\prime$,
and that $S_\infty$ and each $S_\delta$ are compact, from which it follows that
\beq\label{eq:intersection}
\conv(S_\infty) = \bigcap_{\delta>0} \conv (S_\delta).
\eeq
If $\phi\in L^2(\Gamma)$ and $\|\phi\|_{L^2(\Gamma)}=1$, in which case $\sum_{\bx\in G} \theta_\bx=1$, where
$\theta_\bx := \|w_\bx\phi\|_{L^2(\Gamma)}^2$, then, from \eqref{eqn:Dphil2},
\begin{equation} \label{eqn:conv}
((D-T)\phi,\phi) = \sum_{\bx\in G } (w_\bx D(w_\bx\phi),\phi) = \sum_{\bx\in G }\theta_\bx (D\psi_\bx,\psi_\bx),
\end{equation}
where $\psi_\bx := w_\bx\phi/\sqrt{\theta_\bx}$. Clearly $\|\psi_\bx\|_{L^2(\Gamma)}=1$ and $\mathrm{supp}(\psi_\bx)\subset V_\bx\subset B_\delta(\bx)\cap \Gamma$, so that $(D\psi_\bx,\psi_\bx)\in W(D_{\bx,\delta})$.
Therefore \eqref{eqn:conv} implies that
 \beqs
W(D-T) \subset
\conv \left( S_\delta\right) \,\, \text{ for every $\delta>0$},
\eeqs
and hence, since the convex hull of a compact set is compact,
 \begin{equation} \label{eqn:oneside}
W_{\mathrm{ess}}(D)
\subset
\bigcap_{\delta>0}
\mathrm{conv}\left(S_\delta\right).
\end{equation}

On the other hand, given $\bx\in \Gamma$, let $z\in W_\bx$, in which case $z\in \overline{W(D_{\bx,\delta_n})}$ for every positive null sequence $\delta_n$. So suppose that $\delta_n>0$ and $\delta_n\to 0$. Then, for each $n$ there exists $\psi_n\in L^2(\Gamma)$ with $\|\psi_n\|_{L^2(\Gamma)}=1$ and $\mathrm{supp}(\psi_n)\subset \Gamma\cap B_{\delta_n}(\bx)$ such that
\beqs
\big|(D_{\bx,\delta_n}\psi_n,\psi_n) - z\big|<\frac{1}{n}.
\eeqs
Since $(D_{\bx,\delta_n}\psi_n,\psi_n)=(D\psi_n,\psi_n)$ and
$\psi_n\rightharpoonup 0$, $z=\lim_{n\to\infty}((D+T)\psi_n,\psi_n)$ for every compact operator $T$, so that $z\in W_{\mathrm{ess}}(D)$.
Therefore
\beqs
W_\bx\subset W_{\mathrm{ess}}(D) \text{ for every $\bx \in \Gamma$},
\eeqs
so that
\beq\label{eqn:secondside}
\bigcup_{\bx\in \Gamma}W_\bx\subset W_{\ess}(D).
\eeq

Combining \eqref{eq:intersection}, \eqref{eqn:oneside} and \eqref{eqn:secondside}, we see that we have shown that
$$
\bigcup_{\bx\in \Gamma}W_\bx \subset W_{\mathrm{ess}}(D)\subset \bigcap_{\delta> 0} \mathrm{conv}\left(S_\delta\right) = \mathrm{conv}\left(S_\infty\right).
$$
Since $W_{\bx^*}$ is convex, to complete the proof of
\eqref{eqn:erlocall2} it is enough to show that
\beq \label{eqn:finalsub}
S_\infty \subset W_{\bx^*},
\eeq
for some $\bx^*\in \Gamma$.

So suppose that $z\in S_\infty$ and that the positive decreasing sequence $\delta_n\to 0$. Then $z\in S_{\delta_n}$ for each $n$, so that there exists $\bx_n\in \Gamma$ and $z_n\in W(D_{\bx_n,\delta_n})$ such that $|z_n-z|\leq 1/n$. Arguing using the compactness of $\Gamma$ as in the proof of Part (a), by passing to a subsequence if necessary we can assume that, for some $\bx^*\in \Gamma$, $\bx_n\to\bx^*$ as $n\to\infty$, with $\epsilon_n:= \delta_n+|\bx^*-\bx_n|$ a decreasing sequence. By \eqref{eq:inclusionpropW} we have that $z_n\in G_n:=\overline{W(D_{\bx^*,\epsilon_n})}$. Since $z_n\to z$, each $G_n$ is compact, and $G_{n+1}\subset G_n$ for $n\in \mathbb{N}$,
$$
z\in \bigcap_{n\in \mathbb{N}} G_n = W_{\bx^*}.
$$
\end{proof}

The following result is well-known, following, e.g., from \cite[Theorem 1.10]{Ho:94} (note that the space of Lipschitz continuous functions is continuously embedded in the BMO space $I_1(\mathrm{BMO})$ used in \cite{Ho:94}).

\begin{theorem} \label{thm:bounded} If $\Gamma$ is a Lipschitz graph with Lipschitz constant $M$, and $D$ is the double-layer potential operator on $\Gamma$, then there exists an absolute constant $\mu>0$, and a constant $C_d$ depending only on the dimension $d$, such that
$$
\|D\|_{L^2(\Gamma)} \leq C_d M(1+M)^\mu.
$$
\end{theorem}

Combining Theorems \ref{thm:bounded} and \ref{lem:local} we obtain.

\begin{theorem} \label{thm:essnorm} Suppose that $\Omega_-$ has Lipschitz character $M$ and $D$ is the double-layer potential on $\Oi$. Then, where $C_d$ and $\mu$ are as in Theorem \ref{thm:bounded},
\beq\label{eq:smallMbound}
\|D\|_{L^2(\Gamma), \,\mathrm{ess}} \leq  C_d M(1+M)^\mu.
\eeq
\end{theorem}

\bpf
By Theorem \ref{lem:local},
$\|D\|_{L^2(\Gamma),\mathrm{ess}}= \lim_{\delta\to 0}  \|D_{\bx^*,\delta}\|_{L^2(\Gamma)}$ for some $\bx^*\in \Gamma$. Since $\Omega_-$ has Lipschitz character $M$, for all $M'>M$ there exists $\delta>0$ such that, by \eqref{eq:subspaces} and Theorem \ref{thm:bounded}, $\|D_{\bx^*,\delta}\|_{L^2(\Gamma)}\leq C_d M'(1+M')^\mu$, and the result follows.
\epf

\begin{corollary}[$\|D\|_{\LtG, \ess}<1/2$ if the Lipschitz character is small enough.]\label{cor:smallM}
Denote the Lipschitz character of $\Omega_-$ by $M$.
There exists $M_0>0$ such that if $0\leq M<M_0$ then $\|D\|_{\LtG, \ess} <1/2$. Therefore, for this same range of $M$, $\lambda I+D$ is the sum of a coercive operator and a compact operator when $|\lambda|\geq 1/2$.
\end{corollary}

We make three remarks about Theorems \ref{lem:local}, \ref{thm:bounded}, and \ref{thm:essnorm} and Corollary \ref{cor:smallM}.
\bit
\item If $\Gamma$ is $C^1$, then $M=0$, and the bound \eqref{eq:smallMbound} reproduces the result \cite[Theorem 1.2(c)]{FaJoRi:78} that $D$ is compact.
\item Since the essential spectral radius of $D$ is less than its essential norm, Corollary \ref{cor:smallM} reproduces the result \cite[Theorem 1]{Mi:99} (stated for the elasto- and hydro-static analogues of the double-layer potential, but holding also for $D$) that
$\lambda I+D$ is Fredholm of index zero for $|\lambda|\geq 1/2$ if the Lipschitz character is small enough.
\item Theorems \ref{lem:local} and \ref{thm:bounded} make clear that the size of the essential norm depends only on the local Lipschitz character of $\Gamma$ when $\Gamma$ is the boundary of a bounded Lipschitz domain. In contrast, the norm of $D$ depends very much on the global geometry of $\Gamma$.
\eit

\subsection{Locally dilation invariant surfaces and polyhedra}\label{sec:locallyconical}

Our localisation result, Theorem \ref{lem:local}, relates the essential norm and essential numerical range of the double-layer potential on $\Gamma$ to the norm and numerical range of local representatives of $D$ at $\bx\in \Gamma$, specifically to the limits as $\delta\to 0$ of these properties of $D$ restricted to $B_\delta(\bx)$. In certain cases, the local geometry is sufficiently simple that we can characterise these limits rather explicitly. This is the case if $\Omega_-$ is a polyhedron as we note below, but it is also the case for a more general class of domains that we make use of later and introduce now.

\begin{definition}[Locally dilation invariant]\label{def:locdi}
When $\Gamma$ is the boundary of a bounded Lipschitz domain $\Omega_-$, $\Gamma$ is {\em locally dilation invariant at $\bx\in \Gamma$ with scale factor $\newalpha\in (0,1)$} if, for some $\delta>0$,
\begin{equation} \label{eqn:locdi}
\newalpha(B_\delta(\bx)\cap \Gamma-\bx) = B_{\newalpha\delta}(\bx)\cap \Gamma-\bx.
\end{equation}
$\Gamma$ is {\em locally dilation invariant at $\bx\in \Gamma$ if it is locally dilation invariant at $\bx\in \Gamma$ for some scale factor $\newalpha\in (0,1)$.}
$\Gamma$ is {\em locally dilation invariant} if $\Gamma$ is locally dilation invariant at each $\bx\in \Gamma$.
\end{definition}

The proof of the following lemma providing equivalent characterisations is straightforward and is omitted.

\begin{lemma} \label{lem:locdi} Suppose that $\Gamma$ is the boundary of a bounded Lipschitz domain $\Omega_-$ and that $\bx\in \Gamma$ and $\newalpha\in (0,1)$. Then the following are equivalent:
\begin{itemize}
\item[(i)]  $\Gamma$ is locally dilation invariant at $\bx$ with scale factor $\newalpha$.
\item[(ii)] $\newalpha(B_\delta(\bx)\cap \Gamma-\bx) \subset B_{\newalpha\delta}(\bx)\cap \Gamma-\bx$, for some $\delta>0$.
\item[(iii)] There exists some rotated coordinate system with origin at $\bx$ and a Lipschitz continuous function $f:\R^{d-1}\to\R^{d-1}$ with $f(\newalpha \by')=\newalpha f(\by')$, $\by'\in \R^{d-1}$, such that, in this rotated coordinate system, $B_\delta(\bx)\cap \Gamma= B_\delta(\bx)\cap \Gamma_\bx$, where
    $
    \Gamma_\bx := \{(\by',f(\by')):\by\in \R^{d-1}\}
    $ is the graph of $f$.
\item[(iv)] For some $\delta>0$,
\begin{equation} \label{eqn:locdi2}
B_\delta(\bx)\cap \Gamma =  \{\bx\}\cup\Big\{\bx+ \newalpha^n(\by-\bx):n\in \mathbb{N}_0, \, \by\in \big(B_\delta(\bx)\cap \Gamma\big)\setminus B_{\newalpha\delta}(\bx)\Big\}.
\end{equation}
\end{itemize}
\end{lemma}

It is helpful to state explicitly the following result about the invariance of  the normal vector on locally-dilation-invariant $\Gamma$.

\ble[Normal vector on locally-dilation-invariant $\Gamma$]\label{lem:normal}
If $\Gamma$ is locally dilation invariant at $\bx$ with scale factor $\newalpha\in (0,1)$, so that \eqref{eqn:locdi} holds for some $\delta>0$, then, for $\by \in \Gamma \cap B_\delta(\bx)$ such that the unit normal vector $\bn(\by)$ exists, $\bn(\by)=\bn(\bx + \newalpha (\by-\bx))$.
\ele

The following 2-d example of Definition \ref{def:locdi} (and see Figure \ref{fig:GammaEpsOmega}), which we return to later, may be instructive.
\begin{example} \label{ex:loccon}
Suppose that $\Gamma$ is the boundary of a 2-d Lipschitz domain $\Omega_-$, and that, locally near $\bze\in \Gamma$, $\Gamma$ is the graph of the Lipschitz continuous function $f:\R\to \R$ defined, for some $\newalpha \in (0,1)$, by
\beq\label{eq:graphf2}
f(s) =
\begin{cases}
0, & \quad s\leq 0,\\
\displaystyle{\frac{1+\newalpha}{1-\newalpha}\, (s-\newalpha^m)}, & \quad \newalpha^{m} \leq s \leq \newalpha^{m-1}(1+\newalpha)/2,\\
\displaystyle{ \frac{1+\newalpha}{1-\newalpha}\,(\newalpha^{m-1}-s),} & \quad \newalpha^{m-1}(1+\newalpha)/2 \leq s\leq \newalpha^{m-1},
\end{cases}
\eeq
for $m\in \Z$.
Then (by the equivalence of (i) and (iii) in Lemma \ref{lem:locdi}), $\Gamma$ is locally dilation invariant at $\bze$ with scale factor $\newalpha$.
\end{example}

The following is an important special case of Definition \ref{def:locdi}.

\begin{definition}[Locally conical]\label{def:loccon}
$\Gamma$ is {\em locally conical at $\bx\in \Gamma$} if, for some $\delta>0$,
\begin{equation} \label{eqn:loccon}
\newalpha(B_\delta(\bx)\cap \Gamma-\bx) = B_{\newalpha\delta}(\bx)\cap \Gamma-\bx, \quad \mbox{for all } \newalpha\in (0,1),
\end{equation}
equivalently if
\begin{equation} \label{eqn:loccon2}
B_\delta(\bx)\cap \Gamma = \Big\{\bx+t(\by-\bx):0\leq t < 1, \, \by\in \partial B_\delta(\bx)\cap \Gamma\Big\}.
\end{equation}
$\Gamma$ is {\em locally conical} if $\Gamma$ is locally conical at every $\bx\in \Gamma$.
\end{definition}

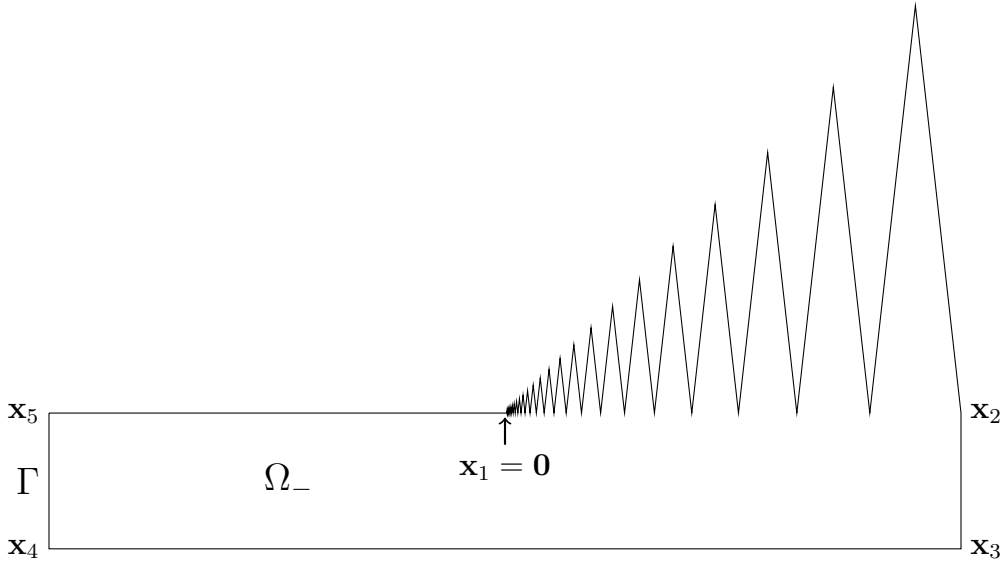
\begin{figure}
\begin{center}
\begin{tikzpicture}[scale=6]
\draw plot coordinates {(-1,0) (0.0037779,0) (0.0042501,0.0042501) (0.0047224,0) (0.0053127,0.0053127) (0.005903,0) (0.0066408,0.0066408) (0.0073787,0) (0.008301,0.008301) (0.0092234,0) (0.010376,0.010376) (0.011529,0) (0.01297,0.01297) (0.014412,0) (0.016213,0.016213) (0.018014,0) (0.020266,0.020266) (0.022518,0) (0.025333,0.025333) (0.028147,0) (0.031666,0.031666) (0.035184,0) (0.039582,0.039582) (0.04398,0) (0.049478,0.049478) (0.054976,0) (0.061848,0.061848) (0.068719,0) (0.077309,0.077309) (0.085899,0) (0.096637,0.096637) (0.10737,0) (0.1208,0.1208) (0.13422,0) (0.15099,0.15099) (0.16777,0) (0.18874,0.18874) (0.20972,0) (0.23593,0.23593) (0.26214,0) (0.29491,0.29491) (0.32768,0) (0.36864,0.36864) (0.4096,0) (0.4608,0.4608) (0.512,0) (0.576,0.576) (0.64,0) (0.72,0.72) (0.8,0) (0.9,0.9) (1,0)
(1,-0.3) (-1, -0.3) (-1,0)};
\draw (-1,-0.15) node[anchor=east] {\Large $\Gamma$};
\draw (-0.4,-0.15) node[anchor=east] {\Large $\Omega_-$};
\draw [thick,->] (0,-0.07) -- ++(90:0.06);
\draw (0,-0.12) node  {\large $\bx_1=\bze$};
\draw (1,0) node[anchor=west] {\large $\bx_2$};
\draw (1,-0.3) node[anchor=west] {\large $\bx_3$};
\draw (-1,-0.3) node[anchor=east] {\large $\bx_4$};
\draw (-1,0) node[anchor=east] {\large $\bx_5$};
\end{tikzpicture}
\end{center}
\caption{\label{fig:GammaEpsOmega} An example Jordan curve $\Gamma$ enclosing a bounded Lipschitz domain $\Omega_-$, and containing $\{(s,f(s)):-1\leq s\leq 1\}$, where $f$ is defined by \eqref{eq:graphf2} with $\newalpha=0.8$. $\Gamma$ is locally dilation invariant at $\bx_1=\bze$ with scale factor $\newalpha$. $\Gamma$ is locally conical (and so locally dilation invariant) at every other $\bx\in \Gamma$, so that $\Gamma$ is locally dilation invariant. The vertices $\{\bx_1,...,\bx_5\}$ are a set of generalised vertices of $\Gamma$ in the sense of Definition \ref{def:gv}. By the calculations at the end of \S\ref{sec:locdiExam}, $\half I \pm D$ and $\half I \pm D'$ cannot be written as sums of coercive plus compact operators for this particular $\Gamma$.}
\end{figure}

Thus, if $\Gamma$ is locally dilation invariant at $\bx\in \Gamma$, then, for some $\delta>0$ and $\newalpha\in (0,1)$, $\bx+\newalpha(\by-\bx)\in \Gamma$ for all $\by\in B_\delta(\bx)\cap \Gamma$, while if $\Gamma$ is locally conical at $\bx$ then this holds {\em for all} $\newalpha\in (0,1)$.

As Figure \ref{fig:GammaEpsOmega} illustrates, $\Gamma$ can have infinitely many vertices while at the same time being locally dilation invariant, and also locally conical at all but one $\bx\in \Gamma$. If $\Gamma$ is locally conical, i.e.\ is locally conical at {\em every} $\bx\in \Gamma$, $\Gamma$ cannot have infinitely many vertices as a consequence of a general result in metric-space geometry \cite{LePe:15}.

\begin{lemma}[Locally conical surfaces are polyhedra] \label{lem:locconPoly} Suppose that $\Gamma$ is the boundary of a bounded Lipschitz domain $\Omega_-$. Then $\Gamma$ is locally conical if and only if $\Gamma$ is a polygon ($d=2$) or a polyhedron ($d=3$).
\end{lemma}
\begin{proof} Equip $\Gamma$ with the metric $d(\cdot,\cdot)$ in which $d(\bx,\by)$ is the geodesic distance from $\bx$ to $\by$ on $\Gamma$. $\Gamma$ with this metric is a compact length space (in the sense of \cite[\S2]{LePe:15}). We note that $\Gamma$ is a polygon or a polyhedron if and only if it is a polyhedral space in the sense of \cite[\S2]{LePe:15}, i.e.\ it admits a finite triangulation (in the standard metric space sense, e.g. \cite[\S78]{Munkes:00}), with each simplex isometric to a simplex in $\R^{d-1}$ (i.e.\ to an interval for $d=2$, a triangle for $d=3$). Further, it follows from \cite[Theorem 1.1]{LePe:15} that $\Gamma$ is a polyhedral space if and only if it is locally conical.
\end{proof}

If $\Gamma$ is locally dilation invariant at $\bx$ with scale factor $\newalpha\in (0,1)$, so that \eqref{eqn:locdi} holds for some $\delta>0$, let
\begin{equation} \label{eqn:dcone}
\Gamma_\bx := \Big\{\bx + \newalpha^n(\by-\bx) :\by \in  B_\delta(\bx)\cap \Gamma, \, n\in \Z\Big\} \quad \mbox{and} \quad \Gamma_\bx^s := B_s(\bx)\cap \Gamma_\bx, \quad s>0.
\end{equation}
(Explicitly, $\Gamma_\bx$ is given as in (iii) of Lemma \ref{lem:locdi}, in terms of the local parameterisation of $\Gamma$ near $\bx$.) If $\Gamma$ is locally conical at $\bx$, then
\begin{equation} \label{eqn:cone}
\Gamma_\bx = \Big\{\bx + t(\by-\bx) :\by \in B_\delta(\bx)\cap \Gamma, \, t\geq 0\Big\},
\end{equation}
i.e.\ $\Gamma_\bx$ is a (Lipschitz) cone.  Let $D_\bx^s$, for $s>0$, denote the double-layer potential operator on $\Gamma_\bx^s$, and $D_\bx$ the double-layer potential operator on $\Gamma_\bx$.

When $\Gamma$ is locally dilation invariant at $\bx\in \Gamma$ with scale factor $\newalpha\in (0,1)$ (in which case $\bx+\newalpha^n(\by-\bx)\in \Gamma_\bx$ for all $\by\in \Gamma_\bx$ and $n\in \Z$),  define $V^\newdelta_\bx:L^2(\Gamma_\bx)\rightarrow L^2(\Gamma_\bx)$, with $\newdelta=\newalpha^n$ for some $n\in \Z$, by
\beq\label{eqn:Vnewdeltadef}
V^\newdelta_\bx \phi(\by):= \phi(\bx + \newdelta (\by-\bx)), \quad \mbox{for }\by\in \Gamma_\bx \;\;\mbox{ and }\;\; \phi\in L^2(\Gamma_\bx).
\eeq
This definition immediately implies that: (i) $V^\newdelta_\bx$ is invertible, with $(V^\newdelta_\bx)^{-1} = V_\bx^{\newdelta^{-1}}$; and
(ii) $V^\newdelta_\bx:L^2(\Gamma_\bx^\newdelta)\rightarrow L^2(\Gamma_\bx^1)$ is a bijection.

\begin{lemma}\label{lem:Visometry}
If $\Gamma$ is locally dilation invariant at $\bx\in \Gamma$ with scale factor $\newalpha\in (0,1)$, and $\newdelta=\newalpha^n$, for some $n\in \Z$, then
$\newdelta^{(d-1)/2} V^\newdelta_\bx$ is a bijective isometry from $L^2(\Gamma_\bx^\newdelta)$ to $L^2(\Gamma_\bx^1)$, and also from $L^2(\Gamma_\bx)$ to itself.
\end{lemma}

\bpf
We have already noted that these mappings are bijections. To see that they are isometries, observe that, for $\psi,\phi \in L^2(\Gamma_\bx)$,
\begin{align*}
\big(V^\newdelta_\bx \psi, V^\newdelta_\bx \phi\big)_{L^2(\Gamma_\bx)} &= \int_{\Gamma_\bx} \psi(\bx+\newdelta (\by-\bx))\, \overline{\phi(\bx+\newdelta(\by-\bx))} \, \rd s(\by)= \newdelta^{-(d-1)} \int_{\Gamma_\bx} \psi(\bz)\,\overline{\phi(\bz)} \, \rd s(\bz),
\end{align*}
by letting $\bz = \bx+\newdelta (\by-\bx)$ and recalling that $\rd s(\bz)= \newdelta^{(d-1)}\rd s(\by)$.
\epf

The following property of the double-layer operator is well known in the case that $\Gamma_\bx$ is a cone; see, e.g., \cite[Equation B.5]{El:92}.

\ble[Dilation property of the double-layer operator]\label{lem:Vcommute}
If $\Gamma$ is locally dilation invariant at $\bx\in \Gamma$ with scale factor $\newalpha\in (0,1)$ and $\newdelta=\newalpha^n$, for some $n\in \Z$, then
\beq\label{eqn:Vcommute}
D_\bx^1 V^\newdelta_\bx = V^\newdelta_\bx D_\bx^\newdelta,
\eeq
where both sides map $L^2(\Gamma_\bx^\rho)\rightarrow L^2(\Gamma_\bx^1)$.
Similarly, $D_\bx V^\newdelta_\bx = V^\newdelta_\bx D_\bx$ on $L^2(\Gamma_\bx)$.
\ele

\bpf
We prove the first of these claims; the second follows similarly.
From \eqref{eq:DD'2}, for $\phi\in L^2(\Gamma_\bx^\rho)$ and $\by \in \Gamma_\bx^\newdelta$,
\begin{eqnarray*}
D_\bx^\newdelta \phi(\by) &= &\frac{1}{c_d} \int_{\Gamma_\bx^\newdelta}\frac{(\by-\bz)\cdot \bn(\bz)}{|\by-\bz|^d}\phi(\bz) \,\rd s(\bz)\\
&=& \frac{1}{c_d} \int_{\Gamma_\bx^1} \frac{ \newdelta(\tby- \tbz)\cdot\bn (\tbz) }{\newdelta^d |\tby-\tbz|^d}\phi(\bx+\newdelta (\tbz-\bx)) \newdelta^{d-1} \rd s(\tbz)= D_\bx^1 (V^\newdelta_\bx \phi)(\tby),
\end{eqnarray*}
where we have used the change of variables $\by= \bx+\newdelta (\tby-\bx)$, $\bz= \bx+\newdelta (\tbz-\bx)$, with $\tby,\tbz\in \Gamma_\bx^1$, and Lemma \ref{lem:normal}. Therefore, for $\tby \in \Gamma_\bx^1,$
\beqs
D_\bx^1 (V^\newdelta_\bx \phi)(\tby) = D_\bx^\newdelta \phi(\by)= D_\bx^\newdelta \phi(\bx+\newdelta(\tby-\bx)) = V^\newdelta_\bx\big(D_\bx^\newdelta \phi\big)(\tby).
\eeqs
\epf

\begin{lemma} \label{lem:loccon} Suppose that $\Gamma$ is locally dilation invariant at $\bx$ with scale factor $\newalpha\in (0,1)$, and $\Gamma_\bx$ and $\Gamma_\bx^s$ are given by \eqref{eqn:dcone}. Let $D_\bx^s$, for $s>0$, denote the double-layer potential operator on $\Gamma_\bx^s$,
let $D_\bx$ denote the double-layer potential operator on $\Gamma_\bx$,
and use the other notations in Theorem \ref{lem:local}.

(i) For all $\delta>0$ sufficiently small so that $B_\delta(\bx)\cap \Gamma=\Gamma_\bx^\delta$,
\beq\label{eqn:DLPnormsequal1}
 \|D_{\bx,\delta}\|_{L^2(\Gamma)} =  \|D_\bx^\delta\|_{L^2(\Gamma^\delta_\bx)} \eeq
and
\beq\label{eqn:DLPnumrangeequal1}
 W(D_{\bx,\delta}) =  W(D_\bx^\delta).
\eeq

(ii) For all $\newdelta>0$,
\beq\label{eqn:DLPnormsequal2}
\|D_\bx^\newdelta\|_{L^2(\Gamma^\newdelta_\bx),\ess} = \|D_\bx^\newdelta\|_{L^2(\Gamma^\newdelta_\bx)} = \|D_\bx\|_{L^2(\Gamma_\bx)} = \|D_\bx\|_{L^2(\Gamma_\bx),\ess}
\eeq
and
\beq\label{eqn:DLPnumrangeequal3}
W_\ess(D_\bx^\newdelta) = \overline{W(D_\bx^\newdelta)}=\overline{W(D_\bx)} = W_\ess(D_\bx).
\eeq
\end{lemma}

\bpf
Equations \eqref{eqn:DLPnormsequal1} and \eqref{eqn:DLPnumrangeequal1} are immediate from the definitions since $\Gamma \cap B_\delta(\bx) = \Gamma^\delta_{\bx}$.
Equations \eqref{eqn:DLPnormsequal2} and \eqref{eqn:DLPnumrangeequal3} follow by applying Lemma \ref{lem:Tlemma} with $\cH:= L^2(\Gamma_\bx)$, $\cA:=D_\bx$, $\cT:= r^{(d-1)/2}V_r$, with $r=\newalpha^{-1}>1$, $\cV := L^2(\Gamma_\bx^\newdelta)\subset L^2(\Gamma_\bx)$, $\cP:\cH\to \cV$ orthogonal projection, and $\cQ:=\cI-\cP$. By Lemma \ref{lem:Visometry}, $\cT$ is an isometry on $L^2(\Gamma_\bx)$, and, by Lemma \ref{lem:Vcommute}, $\cT$ commutes with $D_\bx$.  If $\phi, \psi \in C_{0}(\Gamma_\bx)$, the set of compactly supported continuous functions on $\Gamma_\bx$ whose support is contained in $\Gamma_\bx\setminus\{\bx\}$, then $(\cT^n\phi,\psi)_{L^2(\Gamma_\bx)}=0$ for all sufficiently large $n$ such that the supports of $\cT^n\phi$ and $\psi$ are disjoint; similarly $\cQ\cT^n \phi=0$ for all sufficiently large $n$.
 Since $C_{0}(\Gamma_\bx)$ is dense in $L^2(\Gamma_\bx)$ it follows that $\cT^n\to 0$ weakly and $\cQ\cT^n\to 0$ strongly as $n\to\infty$. Thus the conditions of Lemma \ref{lem:Tlemma} are satisfied and \eqref{eqn:DLPnormsequal2} and \eqref{eqn:DLPnumrangeequal3} follow.
\epf

\begin{remark} \label{rem:neigh}
Part (ii) of the above lemma tells us that, if $\Gamma$ is locally dilation invariant at $\bx$, then, given $\rho>0$,
\begin{equation} \label{eq:neighgen}
\|\widetilde D\|_{L^2(\Gamma_\bx\cap \cN)}= \|\widetilde D\|_{L^2(\Gamma_\bx\cap \cN),\ess} \quad \mbox{and} \quad \overline{W(\widetilde D)} = W_\ess(\widetilde D),
\end{equation}
for the neighbourhood $\cN = B_\newdelta(\bx)$ of $\bx$; here $\widetilde D:= PD_\bx|_{\Gamma_\bx \cap \cN}$, where $P:L^2(\Gamma_\bx)\to L^2(\Gamma_\bx\cap \cN)$ is orthogonal projection, so that $\widetilde D$ is the double-layer potential operator on $\Gamma_\bx \cap \cN$.
In fact there is nothing special about these particular neighbourhoods. The relationships \eqref{eq:neighgen} hold for any open neighbourhood, $\cN$, of $\bx$. This  can be seen by inspecting the proof of the above lemma. Alternatively, if $\cN$ is such a neighbourhood, then $B_\newdelta(\bx) \subset \cN$, for some $\newdelta>0$, so that, from the definition of the essential norm, \eqref{eqn:DLPnormsequal2}, and \eqref{eq:subspaces},
$$
\|\widetilde D\|_{L^2(\Gamma_\bx\cap \cN),\ess}\leq \|\widetilde D\|_{L^2(\Gamma_\bx\cap \cN)}\leq \|D_\bx\|_{L^2(\Gamma_\bx)} =  \|D_\bx^\newdelta\|_{L^2(\Gamma^\newdelta_\bx),\ess} \leq \|\widetilde D\|_{L^2(\Gamma_\bx\cap \cN),\ess};
$$
the same argument holds for the numerical range.
\end{remark}

Combining the above lemma with Theorem \ref{lem:local} (our general localisation result for Lipschitz domains), we obtain a stronger localisation result (Theorem \ref{thm:loccon} below) for the case when $\Gamma$ is locally dilation invariant, in particular when $\Gamma$ is Lipschitz polyhedral. In the case of Lipschitz polyhedral $\Gamma$, similar localisation results have been obtained for the essential spectrum of $D$ on $\LtG$ and other spaces; see, in particular, \cite[\S4]{El:92} and \cite[Theorem 4.16]{LeCoPer:21}. To state our localisation result we need the following definition.

\begin{definition}[Generalised vertices] \label{def:gv} When $\Gamma$ is locally dilation invariant, equipping $\Gamma$ with the Euclidean topology on $\R^d$ restricted to $\Gamma$, the finite set $V_N =\{\bx_1,...,\bx_N\}$ $\subset \Gamma$ is a set of {\em generalised vertices} of $\Gamma$ if every $\bx\in \Gamma$ is an interior point of $\Gamma\cap \Gamma_{\bx_j}$, for some $j\in \{1,...,N\}$. $V_N$ is a {\em minimal set of generalised vertices} if $V_N$ is a set of generalised vertices but $V_N\setminus \{\bx_j\}$ is not a set of generalised vertices for $j=1,...,N$.
\end{definition}

We note that every locally-dilation-invariant $\Gamma$ has a set of generalised vertices (an example is Figure \ref{fig:GammaEpsOmega}). To see this, for each $\bx\in \Gamma$ choose $s(\bx)>0$ such that \eqref{eqn:locdi} holds with $\delta = s(\bx)$. Then $\{B_{s(\bx)}(\bx):\bx\in \Gamma\}$ is an open cover of $\Gamma$ which has a finite subcover $\{B_{s(\bx_1)}(\bx_1), ...,B_{s(\bx_N)}(\bx_N)\}$. Further, $\{\bx_1,...,\bx_N\}$ is a set of generalised vertices, since every $\bx\in \Gamma$ is in one of the balls $B_{s(\bx_j)}(\bx_j)$, and $B_{s(\bx_j)}(\bx_j)\cap \Gamma \subset \Gamma_{\bx_j}\cap \Gamma$, $j=1,...,N$.

It is easy to see that  any set of generalised vertices has a subset that is a minimal set of generalised vertices. It is unclear to us whether there is a unique minimal set of generalised vertices for every locally-dilation-invariant $\Gamma$. However, if $\Gamma$ is a polygon or a polyhedron, there is a unique minimal set of generalised vertices, namely the usual set of vertices; this set is a unique minimum since each vertex, $\bx$, is on the boundary or in the exterior of $\Gamma_\by\cap \Gamma$ for every other $\by\in \Gamma$.

\begin{theorem}[Localisation for locally-dilation-invariant surfaces and polyhedra] \label{thm:loccon} \hfill
Suppose that $\Gamma$, the boundary of the bounded Lipschitz domain $\Omega_-$, is locally dilation invariant and $V_N=\{\bx_1,...,\bx_N\}$ is a set of generalised vertices of $\Gamma$. (This is the case, in particular, if $\Omega_-$ is a polygon or a polyhedron and $V_N$ is the set of vertices in the normal sense.) Then
\beq \label{eq:Dlocal}
\|D\|_{L^2(\Gamma), \,\mathrm{ess}} = \sup_{\bx\in \Gamma} \|D_{\bx}\|_{L^2(\Gamma_{\bx})}=\max_{\bx\in V_N} \|D_{\bx}\|_{L^2(\Gamma_{\bx})}
\eeq
and, for some $\bx_*\in \Gamma$,
\beq \label{eqn:erlocall0}
W_{\mathrm{ess}}(D)= \bigcup_{\bx\in \Gamma} \overline{W(D_{\bx})}  =\bigcup_{\bx\in V_N}\overline{W(D_{\bx})}= \overline{W(D_{\bx^*})}.
\eeq
\end{theorem}
\begin{proof} By Theorem \ref{lem:local}, for some $\bx^*\in \Gamma$,
$$
\|D\|_{\LtG,\ess}  = \sup_{\bx\in \Gamma} \lim_{\delta\to 0} \|D_{\bx,\delta}\|_\LtG = \lim_{\delta\to 0} \|D_{\bx^*,\delta}\|_\LtG,
$$
from which it follows, by \eqref{eqn:DLPnormsequal1} and \eqref{eqn:DLPnormsequal2} in Lemma \ref{lem:loccon}, that
$$
\|D\|_{\LtG,\ess} = \sup_{\bx\in \Gamma} \|D_{\bx}\|_{L^2(\Gamma_{\bx})}=\|D_{\bx^*}\|_{L^2(\Gamma_{\bx^*})} = \|D_{\bx^*}^\newdelta\|_{L^2(\Gamma_{\bx^*}^\newdelta)},
$$
for all $\newdelta>0$. Since $V_N$ is a set of generalised vertices, $\bx_*$ is an interior point of $\Gamma \cap \Gamma_{\bx}$, for some $\bx\in V_N$, so that $\Gamma_{\bx^*}^\newdelta \subset \Gamma \cap \Gamma_{\bx}$, for some $\rho>0$ and $\bx\in V_N$, so that
$\|D_{\bx^*}^\newdelta\|_{L^2(\Gamma_{\bx^*}^\newdelta)}\leq \|D_{\bx}\|_{L^2(\Gamma_{\bx})}$, by \eqref{eq:subspaces}. Thus
$$
\|D\|_{\LtG,\ess} =\|D_{\bx^*}\|_{L^2(\Gamma_{\bx^*})}= \|D_{\bx^*}^\newdelta\|_{L^2(\Gamma_{\bx^*}^\newdelta)}\leq \max_{\bx\in V_N} \|D_{\bx}\|_{L^2(\Gamma_{\bx})} \leq \sup_{\bx\in \Gamma} \|D_{\bx}\|_{L^2(\Gamma_{\bx})} = \|D\|_{\LtG,\ess}.
$$

Similarly, by Theorem \ref{lem:local}, for some $\bx^*\in \Gamma$,
$$
W_\ess(D) = \bigcup_{\bx\in \Gamma} \bigcap_{\delta>0} \overline{W(D_{\bx,\delta})} = \bigcap_{\delta>0} \overline{W(D_{\bx^*,\delta})},
$$
from which it follows, by \eqref{eqn:DLPnumrangeequal1} and \eqref{eqn:DLPnumrangeequal3} in Lemma \ref{lem:loccon}, that
$$
W_\ess(D) = \bigcup_{\bx\in \Gamma} \overline{W(D_{\bx})}  = \overline{W(D_{\bx^*})} = \overline{W(D^\newdelta_{\bx^*})},
$$
for all $\newdelta>0$. Again, $\Gamma_{\bx^*}^\newdelta \subset \Gamma \cap \Gamma_{\bx}$, for some $\rho>0$ and $\bx\in V_N$, so that, by \eqref{eq:subspaces},
$$
W_\ess(D)   = \overline{W(D_{\bx^*})} = \overline{W(D^\newdelta_{\bx^*})} \subset \bigcup_{\bx\in V_N}\overline{W(D_{\bx})} \subset \bigcup_{\bx\in \Gamma} \overline{W(D_{\bx})}=W_\ess(D).
$$
\end{proof}

We suspect, but have no proof of this conjecture, that the point $\bx_*$ in Theorem \ref{thm:loccon} is one of the generalised vertices, $\bx_1,...,\bx_N$.

\section{The essential norm and essential numerical range of the double-layer operator on Lipschitz domains (proofs of Theorems \ref{thm:Q1} and \ref{thm:Q2})}\label{sec:3}

In this section we prove two of our main theorems, Theorems \ref{thm:Q1} and \ref{thm:Q2}, constructing particular 2-d and 3-d Lipschitz domains $\Omega^M$, with Lipschitz constant $M$, for which we can prove rather sharp quantitative estimates for the essential norm and essential numerical range and their dependence on $M$. These constructions and the associated arguments have some subtleties. To get the main ideas across, we present in \S\ref{sec:locdiExam} calculations for the slightly-simpler 2-d example that we have met already in Figure \ref{fig:GammaEpsOmega} and Example \ref{ex:loccon}, this an example that illustrates our localisation result, Theorem \ref{thm:loccon}, for locally-dilation-invariant surfaces.

Nevertheless, throughout this section (and \S\ref{sec:4} to follow) we use, in the main, the same tools, namely:
\begin{itemize}
\item[i)] our localisation results, Theorems \ref{lem:local} and \ref{thm:loccon}, and the associated Lemma \ref{lem:loccon}, that replace calculation of the essential norm and essential range of the double-layer potential operator $D$ on $\Gamma$, with calculations of the ordinary norm and numerical range;
    \item[ii)] equations \eqref{eq:ANIncl} and \eqref{eq:NRlimits}, that reduce  calculation of the norm and numerical range of a bounded linear operator to calculation of the same quantities for matrices; and
        \item[iii)] Lemmas \ref{lem:CNprops} and \ref{lem:AN} and Corollary \ref{cor:CN} that provide estimates of the numerical ranges of the matrices that arise in our calculations.
\end{itemize}
Additionally, in Sections \ref{sec:Q1Q2_2d} and \ref{sec:Q1Q2_3d} we require calculations of the norms of infinite Toeplitz matrices that are associated with our operator $D$.

Related to ii), we use throughout this section and \S\ref{sec:4} the following construction, a particular instance of the general Hilbert space construction below \eqref{eq:ANGendef}. Suppose that $\widetilde \Gamma$ is some subset of $\Gamma$, measurable with respect to surface measure. For some $N\in \NN$,
let $\cH_N$ denote an $N$-dimensional subspace of $L^2(\widetilde \Gamma)\subset L^2(\Gamma)$ and let $\{\psi_1,...,\psi_N\}$ be an orthonormal basis for $\cH_N$, and define the Galerkin matrix $D_N \in \R^{N\times N}$ by
\begin{equation} \label{eq:DNdef}
(D_N)_{jm} := (D\psi_m,\psi_j), \quad 1\leq j,m\leq N.
\end{equation}
Then (see \eqref{eq:ANIncl})
\begin{equation}
\label{eq:DNIncl}
\|D_N\|_2 \leq \|\widetilde D\|_{L^2(\widetilde \Gamma)} \leq \|D\|_{L^2(\Gamma)} \quad \mbox{and} \quad  W(D_N) \subset W(\widetilde D) \subset W(D),
\end{equation}
where $\widetilde D$ and $D$ are the double-layer potential operators on $\widetilde \Gamma$ and $\Gamma$, respectively.
Suppose further that $(\cH_N)_{N=1}^\infty$ is a sequence of finite-dimensional subspaces, with $\cH_N\subset L^2(\widetilde D)$, $\cH_N$ of dimension $N$, and $\cH_1\subset \cH_2\subset ...$, and let
$$
V := \overline{\bigcup_{N=1}^\infty \cH_N} \subset L^2(\widetilde \Gamma).
$$
($V=\LtG$ if $(\cH_N)_{N=1}^\infty$ is asymptotically dense in $\LtG$.)
Then, where $D_V := PD|_V$ and $P:\LtG\to V$ is orthogonal projection, it follows  from \eqref{eq:NRlimits} and \eqref{eq:subspaces} that
\begin{eqnarray} \label{eq:DNsubsets}
&\|D_1\|_2 \leq \|D_2\|_2 \leq ... \leq \|D_V\|_V,\quad  W(D_1)\subset W(D_2) \subset ...\subset W(D_V),\\
 \label{eq:limits}
&\displaystyle{\lim_{N\to \infty}} \|D_N\|_2 = \|D_V\|_V \leq \|\widetilde D\|_{L^2(\widetilde \Gamma)}, \;\; \mbox{and} \;\; \overline{W(\widetilde D)} \supset \overline{W(D_V)} = \overline{\bigcup_{N=1}^\infty W(D_N)}.
\end{eqnarray}

\subsection{The essential numerical range of a domain with a locally-dilation-invariant boundary} \label{sec:locdiExam}

\begin{definition}[2-d graph $\Gamma_\newalpha$]\label{def:Gamma2d}
Choose $\newalpha \in (0,1)$. For $m\in \mathbb{N}$, let $\Gamma_{2m-1}$ denote the open straight line segment connecting $(\newalpha^{m-1},0)$ and $(\newalpha^{m-1}(1+\newalpha)/2,\newalpha^{m-1}(1+\newalpha)/2)$, and let $\Gamma_{2m}$ denote the open line segment connecting $(\newalpha^{m-1}(1+\newalpha)/2,\newalpha^{m-1}(1+\newalpha)/2)$ and $(\newalpha^{m},0)$.
 Let
$$
\Gamma_\newalpha := \{(s,0):-1\leq s\leq 0\} \cup \bigcup_{m\in \mathbb{N}} \overline{\Gamma_m};
$$
equivalently $\Gamma_\newalpha:=\{(s,f(s)):-1\leq s\leq 1\}$, where $f$ is defined by \eqref{eq:graphf2}.
\end{definition}

We consider the case where $\Gamma_\newalpha\subset \Gamma$ and $\Gamma$ is the boundary of a Lipschitz domain $\Omega_-$. (Such a case is shown, for $\newalpha=0.8$,  in Figure \ref{fig:GammaEpsOmega}, where $\Gamma_\newalpha$ is the part of $\Gamma$ between the points $\bx_5$ and $\bx_2$.)
From \eqref{eq:graphf2},
\begin{equation} \label{eq:fderiv}
|f^\prime(s)| = \frac{1+\newalpha}{1-\newalpha}, \quad \mbox{for almost all } s\in [0,1],
\end{equation}
so that the Lipschitz constant of $\Omega_-$ tends to infinity as $\newalpha\to 1^-$. We note also that, as discussed in Example \ref{ex:loccon}, $\Gamma$ is locally dilation invariant at $\bx=\bze$, with scale factor $\newalpha$. Thus, by Remark \ref{rem:neigh} and \eqref{eq:subspaces},
\begin{equation} \label{eq:GammaAl}
\overline{W(D_\newalpha)} = W_\ess(D_\newalpha) \subset W_\ess(D),
\end{equation}
where $D_\newalpha$ and $D$ are the double-layer potential operators on $\Gamma_\newalpha$ and $\Gamma$, respectively. We note further that if, as in Figure \ref{fig:GammaEpsOmega}, $\Gamma$ is locally dilation invariant, then $W_\ess(D)$ can be characterised precisely by Theorem \ref{thm:loccon}.

To establish a lower bound for $W(D_\newalpha)$ when $\newalpha\approx 1$ our method is to relate, in the sense of the following lemma,  $D_\newalpha$ to the matrix $B_N$ defined by \eqref{eq:aBdef}.

\ble[The relationship of $B_N$ to $D_\newalpha$]\label{lem:ANDLP}
Let $\Gamma_\newalpha$ be as in Definition \ref{def:Gamma2d}.
Given $N\in \NN$, define the orthonormal set $\{\psi_1,...,\psi_N\}\subset L^2(\Gamma_\newalpha)$ by
\beq\label{eq:defphi2d}
\psi_m(\bx):=
\begin{cases}
|\Gamma_m|^{-1/2}, & \bx \in \Gamma_m,\\
0, & \mbox{otherwise},
\end{cases}
\eeq
for $m=1,...,N$.
Suppose that the unit normal $\bn$ on $\Gamma_\newalpha$ is such that $n_2(\bx)>0$ for almost every $\bx \in \Gamma_\newalpha$, and define the Galerkin matrix $D_N$ by \eqref{eq:DNdef}, where $D=D_\newalpha$ is the double-layer potential operator on $\Gamma_\newalpha$, and define $B_N\in \R^{N\times N}$ by \eqref{eq:aBdef}.
Then
\beq\label{eq:DNform}
\big( D_N\big)_{jm}= \big( B_N\big)_{jm}d_{jm}, \quad 1\leq j,m\leq N,
\eeq
where $d_{jm}:= 0$ for $j=m$,
\beq \label{eq:djmDef}
d_{jm} := \frac{1}{2\pi |\Gamma_j|^{1/2}|\Gamma_m|^{1/2}}\int_{\Gamma_j} \alpha_m \, \rd s, \quad \mbox{for } j\neq m,
\eeq
 and
\beq\label{eq:DLPsolidangle}
 \alpha_m(\bx) := 2\pi \left|\int_{\Gamma_m} \frac{\partial \Phi(\bx,\by)}{\partial n(\by)} \rd s(\by)\right|= \int_{\Gamma_m} \frac{|(\bx-\by)\cdot \bn(\by)|}{|\bx-\by|^2}\rd s(\by), \quad \bx \in \R^2\setminus \overline{\Gamma_m},
\eeq
is the angle subtended at $\bx$ by $\Gamma_m$. Further, $d_{jm}=d_{mj}$ when $m-j$ is even, and
\beq \label{eq:WDNGal}
W(D_N)\cap \R \supset [-a,a],
\eeq
for $N\geq 3$, with equality when $N=3$,  where
\begin{equation} \label{eq:aForm}
a := \frac{1}{2}\sqrt{(d_{12}+d_{21})^2 + (d_{23}+d_{32})^2}.
\end{equation}
Moreoever, for $1\leq j,m \leq N$, with $j\neq m$, $d_{jm}\to 1/2$ as $\newalpha\to 1^-$, so that
\beq\label{eq:limitDA}
D_N \rightarrow \frac{1}{2}\, B_N.
\eeq
\ele
\bpf
For $m\in \mathbb{N}$ and $\bx\in \R^2\setminus \overline{\Gamma_m}$,
\beq\label{eq:DLPsolidangle2a}
 \int_{\Gamma_m} \frac{\partial \Phi(\bx,\by)}{\partial n(\by)} \rd s(\by) = \pm \frac{\alpha_m(\bx)}{2\pi},
\eeq
where the $+$ sign is taken if $\bx$ is on the side of $\Gamma_m$ to which the normal points, i.e.\ if, for some $\by\in \Gamma_m$, $(\bx-\by)\cdot \bn(\by)>0$, otherwise the $-$ sign is taken.
Thus, for $1\leq j,m\leq N$,
\begin{align}
D_\newalpha\psi_m(\bx) = \frac{1}{2\pi|\Gamma_m|^{1/2}} (-1)^{m+1}\mathrm{sign}(m-j)\alpha_m(\bx),  \quad \bx\in \Gamma_j,\label{Dphi}
\end{align}
and \eqref{eq:DNform} follows. When $m-j$ is even, $\Gamma_j$ and $\Gamma_m$ are parallel, so that $d_{jm}=d_{mj}$  by
the combination of \eqref{eq:djmDef} and \eqref{eq:DLPsolidangle} and the fact that
$|(\bx-\by)\cdot\bn(\by)|= |(\bx-\by)\cdot\bn(\bx)|$ for $\bx \in \Gamma_j$, $\by\in\Gamma_m$.
Then \eqref{eq:WDNGal} holds for $N\geq 3$, with equality when $N=3$, by Lemma \ref{lem:CNprops}(iii). As $\newalpha\to 1^-$, $d_{jm}\to 1/2$ (for $j\neq m$)
implying \eqref{eq:limitDA}, since
$$
|\Gamma_j|\sim |\Gamma_m| \to 1, \quad \mbox{and} \quad \int_{\Gamma_j} \alpha_m \, \rd s\to \pi
$$
by the dominated convergence theorem.
\epf

\begin{theorem}\label{thm:main2-d}
Given $R>0$, there exists $\newalpha_0\in (0,1)$ such that, if $\Omega_-$ is a bounded Lipschitz domain with boundary $\Gamma$ that contains $\Gamma_\newalpha$ and $\newalpha_0\leq \newalpha<1$, then
\beq \label{eq:WessIncl}
W_{\ess}(D) \supset \big\{ z \in \Com : |z| <R\big\},
\eeq
so that also $\|D\|_{\LtG,\ess} \geq R$, and $\lambda I +D$ and $\lambda I + D'$ cannot be written as the sum of coercive and compact operators for any $\lambda\in \C$ with $|\lambda|\leq R$.
\end{theorem}
\begin{proof}
Given $R>0$, choose the smallest odd $N\geq 3$ such that $\sqrt{(N-1)/2}\geq 2R+1$, and then  choose $\newalpha_0\in (0,1)$ such that $|2d_{jm}-1| \leq 1/(N-1)$ for $1\leq j,m \leq N$ and  $\newalpha_0\leq \newalpha<1$, with this possible by Lemma \ref{lem:ANDLP}. Then, by Corollary \ref{cor:CN} applied with $\epsilon =1$,
$$
W(2D_N) \supset \big\{ z \in \Com : |z| \leq\sqrt{(N-1)/2}-\epsilon\big\} \supset \big\{ z \in \Com : |z| <2R\big\},
$$
for $\newalpha_0\leq \newalpha<1$, so that $W(D_N) \supset \{ z \in \Com : |z| <R\}$. When the outward-pointing normal, $\bn$, on $\Gamma$, has $n_2>0$ almost everywhere on $\Gamma_\newalpha$, this inclusion, together with  \eqref{eq:DNIncl}, applied with $\widetilde \Gamma=\Gamma_\newalpha$ and $\widetilde D=D_\newalpha$, and \eqref{eq:GammaAl}, implies \eqref{eq:WessIncl}. When $n_2<0$, we have instead that $-W(D_N)=W(-D_N) \subset W_\ess(D)$, and the same conclusion holds.
It follows that $\|D\|_{\LtG,\ess} \geq R$ by \eqref{eq:NormEssNorm}. Further, \eqref{eq:WessIncl} implies that also $W_\ess(D') \supset \{ z \in \Com : |z| <R\}$, so that $0$ is in both $W_\ess(\lambda I + D)$ and $W_\ess(\lambda I + D')$, for $|\lambda|\leq R$, so that, by Corollary \ref{cor:essnum}, $\lambda I +D$ and $\lambda I + D'$ cannot be written as sums of coercive and compact operators.
\end{proof}

 The above theorem guarantees that \eqref{eq:WessIncl} holds, in particular that $\pm 1/2\in W_\ess(D)$, if $\Gamma\supset \Gamma_\newalpha$ and $\newalpha$ is close enough to 1, but gives no idea of how close to one $\newalpha$ needs to be. We can estimate this using \eqref{eq:aForm}. Arguing as in the proof of the above theorem and using \eqref{eq:WDNGal}, we have, under the same assumptions as the theorem, that
 $$
 [-a,a] \subset W_\ess(D)=W_\ess(D'),
 $$
where $a$ is given by \eqref{eq:aForm} with the coefficients $d_{jm}$ given by \eqref{eq:djmDef}. In Figure \ref{fig:aValpha} we plot $a$ against $\newalpha$, computing the integrals \eqref{eq:djmDef} accurately by numerical quadrature. As expected, since $d_{jm}\to 1/2$ as $\newalpha\to 1^-$ by Lemma \ref{lem:ANDLP} provided $j\neq m$, $a$ approaches the limit $\sqrt{2}/2 \approx 0.7071$ as $\newalpha\to 1^-$. The data we plot suggest that $[-a,a]\supset \{-\half,\half\}$ provided $\newalpha > 0.7903$, so that, by Corollary \ref{cor:essnum}, $\half I \pm D$ and $\half I \pm D'$ cannot be written as sums of coercive and compact operators for $\newalpha$ above this value. (This is the case, in particular, for the $\Gamma$ in Figure \ref{fig:GammaEpsOmega}.) The value $\newalpha = 0.7903$ corresponds, by \eqref{eq:fderiv}, to a Lipschitz constant for the surface $\Gamma_\newalpha\subset \Gamma$ of $(1+\newalpha)/(1-\newalpha)\approx 8.54$.

\begin{figure}
\begin{center}
\includegraphics[width=.5\textwidth]{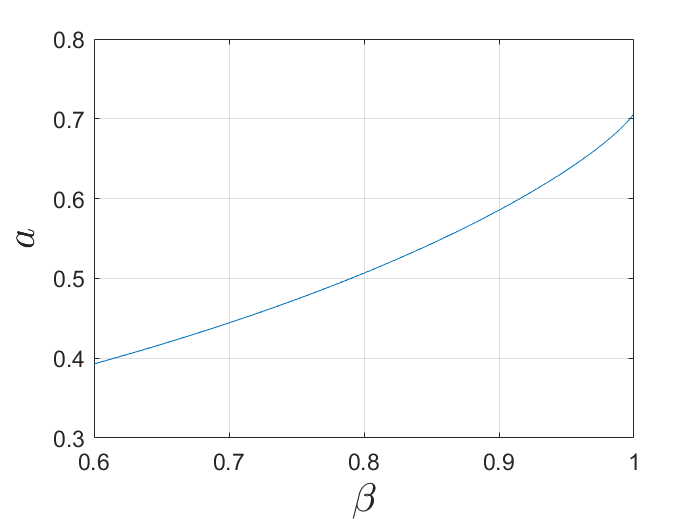}
\end{center}
\caption{\label{fig:aValpha} Plot of $a$, given by \eqref{eq:aForm}, against $\newalpha$. Note that $[-a,a]\subset \R\cap W(D_N)$, for $N\geq 3$, with equality when $N=3$, so that also $[-a,a]\subset W_\ess(D_\newalpha)\subset W_\ess(D)$. }
\end{figure}

\subsection{Proof of Theorems \ref{thm:Q1} and \ref{thm:Q2} in the 2-d case} \label{sec:Q1Q2_2d}

\subsubsection{The norm and numerical range of the double-layer potential operator  on a periodic Lipschitz graph} \label{sec:pergraph}

As the main step in the proofs of  Theorems \ref{thm:Q1} and \ref{thm:Q2}, we calculate in this section lower bounds on the norm and numerical range (and their essential variants)  for the double-layer potential operator $D$ on a particular ``sawtooth'' periodic Lipschitz graph.

\begin{definition}[The ``sawtooth'' Lipschitz graph $\Gamma^M$] \label{def:GammaM}
Given $M> 0$ (the Lipschitz constant) define $f_M:\R\to\R$ by
\begin{equation}\label{eq:Gammaeps}
f_M(s) := \left\{\begin{array}{ll}
                          Ms-2m, & 2m\leq Ms \leq 2m+1, \\
                          2m+2-Ms, & 2m+1\leq Ms \leq 2m+2,
                        \end{array}
\right.
\end{equation}
for $m\in \Z$, so that $f_M$ is periodic with period $2/M$ and
$$
|f^\prime_M(s)| = M,
$$
for almost all $s\in \R$. Let $\Gamma^M:= \{(s,f_M(s)):s\in \R\}$ be the graph of $f_M$, so that
$$
\Gamma^M = \bigcup_{m\in \Z} \overline{\Gamma_m},
$$
where $\Gamma_m$ is the open line segment connecting $((1-m)/M,0)$ and $(-m/M,1)$, for $m$ odd, connecting $((1-m)/M,1)$ and $(-m/M,0)$, for $m$ even. (See Figure \ref{fig:Saw} for $\Gamma^M$ when $M=1$.)
\end{definition}

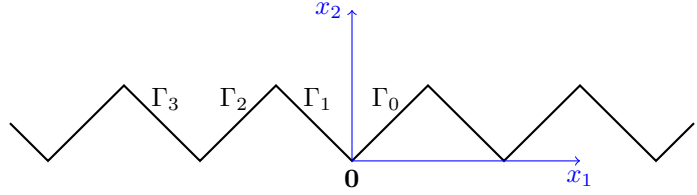
\begin{figure}
\begin{center}
\begin{tikzpicture}[scale=1]
\draw[thick] plot coordinates {(4.5,0.5) (4,0) (3,1) (2,0) (1,1) (0,0) (-1,1) (-2,0) (-3,1) (-4,0) (-4.5,0.5)};
\draw [blue,->] (0,0) -- ++(0:3);
\draw [blue,->] (0,0) -- ++(90:2);
\draw (0,0) node[anchor=north] {$\bze$};
\draw (0.45,0.55) node[anchor=south] {$\Gamma_0$};
\draw (-0.45,0.55) node[anchor=south] {$\Gamma_1$};
\draw (-1.55,0.55) node[anchor=south] {$\Gamma_2$};
\draw (-2.45,0.55) node[anchor=south] {$\Gamma_3$};
\draw[blue] (3,0) node[anchor=north] {$x_1$};
\draw[blue] (0,2) node[anchor=east] {$x_2$};
\end{tikzpicture}
\end{center}
\caption{\label{fig:Saw} The ``sawtooth'' curve $\Gamma^M$, as specified in Definition \ref{def:GammaM}, when $M=1$.}
\end{figure}

Let $D^M:L^2(\Gamma^M)\to L^2(\Gamma^M)$ denote the double-layer potential operator on $\Gamma^M$, defined with the unit normal $\bn$ on $\Gamma^M$ such that $n_2(\bx)>0$ for almost all $\bx\in \Gamma^M$. Define $V^M_+\subset V^M\subset L^2(\Gamma^M)$ by
\begin{equation} \label{eq:Vdef}
V^M := \{\phi\in L^2(\Gamma^M): \phi \mbox{ is constant on } \Gamma_m \mbox{ for } m\in \Z\}, \quad V^M_+ := \{\phi\in V^M: \phi|_{\Gamma_m} = 0, \mbox{ for } m\leq 0\}.
\end{equation}
Let $P^M:L^2(\Gamma^M)\to V^M$ and $P^M_+:L^2(\Gamma^M)\to V_+^M$ be orthogonal projections, and let $D_{V^M}:= P^MD^M|_{V^M}$ and $D_{V_+^M}:= P_+^MD^M|_{V_+^M}$.

The next two lemmas do most of the quantitative work for us in proving Theorems \ref{thm:Q1} and \ref{thm:Q2}. In particular, Lemma \ref{lem:ANDLPrime} provides precise characterisations for  $\|D_{V^M}\|_{V^M} \leq \|D^M\|_{L^2(\Gamma^M)}$ and  $\overline{W(D_{V^M})}\subset \overline{W(D^M)}$ in terms of symbols of associated infinite Toeplitz matrices.

\begin{lemma} \label{lem:DM} For $M>0$,
\begin{eqnarray*}
\|D^M\|_{L^2(\Gamma^M),\ess} &=& \|D^M\|_{L^2(\Gamma^M)} \geq \|D_{V^M}\|_{V^M} = \|D_{V_+^M}\|_{V_+^M} \quad \mbox{and}\\
W_\ess(D^M) &=& \overline{W(D^M)} \supset \overline{W(D_{V^M})} = \overline{W(D_{V^M_+})}.
\end{eqnarray*}
\end{lemma}
\begin{proof}
The first equalities in each line follow by applying Lemma \ref{lem:Tlemma} with $\cH=L^2(\Gamma^M)$, $\cA=D^M$, and $\cT= T$, where $T:L^2(\Gamma^M)\to L^2(\Gamma^M)$ is a right-shift operator defined by
$$
T \phi(\bx) = \phi\left((x_1-2/M,x_2)\right), \quad \mbox{for } \bx=(x_1,x_2)\in \Gamma^M, \quad \phi\in L^2(\Gamma^M).
$$
Similarly, the last equalities in each line follow by applying Lemma \ref{lem:Tlemma} with $\cH=V^M$, $\cA=D_{V^M}$, $\cT= T|_{V^M}$, $\cV= V^M_+$, and $\cP = P^M_+|_{V^M}$. The other statements of the lemma (the ``$\geq$'' and the ``$\supset$'') follow from \eqref{eq:subspaces}.
\end{proof}

In the following lemma we use the notations of \S\ref{sec:CompNR}. In particular, $w_r(C_N)$ denotes the numerical abscissa of a square matrix $C_N$, and $B_N$ is the matrix defined by \eqref{eq:aBdef}. We denote the norm on $L^\infty(-\pi,\pi)$ by $\|\cdot\|_\infty$.

\ble[The double-layer operator on $\Gamma^M$]\label{lem:ANDLPrime}
Given $N\in \NN$, define the orthonormal set $\{\psi_1,...,\psi_N\}\subset V_+^M\subset L^2(\Gamma^M)$ by \eqref{eq:defphi2d}, but with $\Gamma_m$ as in Definition \ref{def:GammaM}, and define the Galerkin matrix $D_N$ by \eqref{eq:DNdef}, where $D=D^M$ is the double-layer potential operator on $\Gamma^M$.
Then:
\begin{itemize}
\item[(a)]
\beq\label{eq:DNformP}
\big( D_N\big)_{jm}= \big( B_N\big)_{jm}d^\prime_{j-m}, \quad 1\leq j,m\leq N,
\eeq
where $d^\prime_0:= 0$,
\beq \label{eq:djmDef2}
d^\prime_{\ell} := \frac{1}{2\pi \left|\Gamma_{|\ell|}\right|}\int_{\Gamma_{|\ell|}} \alpha \, \rd s, \quad \ell\in \Z\setminus\{0\},
\eeq
and
\beq\label{eq:DLPsolidangle2}
 \alpha(\bx) := 2\pi \left|\int_{\Gamma_0} \frac{\partial \Phi(\bx,\by)}{\partial n(\by)} \rd s(\by)\right|, \quad \bx \in \R^2\setminus \overline{\Gamma_0},
\eeq
is the angle subtended at $\bx$ by $\Gamma_0$. Further, for each $\ell\in \Z\setminus\{0\}$,
\beq\label{eq:limitDAP1}
d_\ell^\prime \to \half \quad \mbox{so that} \quad D_N \rightarrow \frac{1}{2}\, B_N \quad \mbox{as} \quad M \to\infty,
\eeq
and
\begin{equation} \label{eq:dl_lim}
d_{\ell}^\prime = \frac{M}{2\pi |\ell|} + O(\ell^{-2}), \quad \mbox{as}\quad|\ell|\to \infty.
\end{equation}
\item[(b)]
\begin{equation} \label{eq:DNormlim}
\|D_N\|_2 \to \|D_{V^M}\|_{V^M} = \|E\|_{\ell^2(\NN)} = \|e\|_\infty, \quad \mbox{as} \quad N\to \infty,
\end{equation}
where $E:\ell^2(\NN)\to\ell^2(\NN)$ is (multiplication by) the infinite Toeplitz matrix defined by
$$
\left(E\right)_{jm} := \sign(m-j) d^\prime_{m-j}, \quad j,m\in \NN,
$$
and $e\in L^\infty(-\pi,\pi)$ is its symbol, given by
\begin{equation} \label{eq:symb}
e(t) = -2\ri \sum_{m=1}^\infty d_m^\prime \sin(m t) = -\frac{\ri M\sign(t)(\pi-|t|)}{2\pi} + e_r(t), \quad -\pi \leq t \leq \pi,
\end{equation}
where $e_r\in C(\R)$ is $2\pi$-periodic with $e_r(0)=0$, so that
\begin{equation} \label{eq:enorm}
\|D_{V^M}\|_{V^M} = \|E\|_{\ell^2(\NN)} =\|e\|_\infty \geq \frac{M}{2}.
\end{equation}
\item[(c)] For $N\geq 3$,
\begin{equation} \label{eq:WDNsup}
W(D_N)\cap\R\supset [-a,a], \quad \mbox{where} \quad a:= \sqrt{2}\, d_1^\prime,
\end{equation}
with equality when $N=3$. Further, where $D_N^\theta$ denotes the real part of $\re^{\ri \theta} D_N$, $w_r(\re^{\ri\theta}D_N)=w_r(\re^{-\ri\theta}D_N)=w_r(-\re^{\ri\theta}D_N)=\|D_N^\theta\|_2$, $\theta\in \R$, and
\begin{equation} \label{eq:WDN}
w_r(\re^{\ri\theta}D_N) =\|D_N^\theta\|_2 \to \|H^\theta\|_2 = \|h^\theta\|_\infty, \quad \mbox{as} \quad N\to\infty,
\end{equation}
where $H^\theta:\ell^2(\NN)\to\ell^2(\NN)$ is (multiplication by) the infinite Toeplitz matrix defined by
$$
\left(H^\theta\right)_{jm} := \frac{1-\re^{-2\ri\theta}(-1)^{j-m}}{2}\,\sign(m-j) d^\prime_{m-j}, \quad j,m\in \NN,
$$
and $h^\theta\in L^\infty(-\pi,\pi)$ is its symbol, given by
\begin{equation} \label{eq:hsymb}
h^\theta(t) = -\ri \sum_{m=1}^\infty (1-\re^{-2\ri\theta}(-1)^m)d_m^\prime \sin(mt) = -\ri M\,\frac{\sign(t)(\pi-|t|) + \re^{-2\ri\theta}t}{4\pi} + h^\theta_r(t),
\end{equation}
for $-\pi < t <\pi$, where $h^\theta_r\in C(\R)$ is $2\pi$-periodic with $h^\theta_r(0)=0$, so that
\begin{equation} \label{eq:hnorm}
\lim_{N\to\infty} w_r(\re^{\ri\theta}D_N)=\|h^\theta\|_\infty \geq \frac{M}{4}, \quad \theta \in \R.
\end{equation}
Moreover,
\begin{equation} \label{eq:finalW}
\overline{W(D_{V^M})} =\overline{\bigcup_{N=1}^\infty W(D_N)}= \bigcap_{0\leq \theta \leq 2\pi} \left\{\lambda\in \C: \Re(\re^{\ri\theta}\lambda) \leq \|h^\theta\|_\infty\right\}\supset \{\lambda\in \C:|\lambda| \leq M/4\}.
\end{equation}
\end{itemize}
\ele
\bpf
Part (a). Arguing as in the proof of Lemma \ref{lem:ANDLP}, $(D_N)_{jm}$ is given by the right hand side of \eqref{eq:DNform}, but with $d_{jm}$ replaced by $d^\prime_{jm}$, defined by
\eqref{eq:djmDef} and \eqref{eq:DLPsolidangle}
 with $\Gamma_m$ as in Definition \ref{def:GammaM}. By symmetry of $\Gamma^M$ and since $f_M$ is periodic with period $2/M$, $d^\prime_{jm} = d^\prime_{|j-m|,0}=d^\prime_{|j-m|}$, for $1\leq j,m\leq N$. That $d_{-\ell}^\prime = d_\ell^\prime = d_{\ell,0}^\prime\to 1/2$ as $M\to\infty$, for each $\ell\in \NN$, so that \eqref{eq:limitDAP1} holds, follows exactly as in the proof of Lemma \ref{lem:ANDLP}, and the asymptotics \eqref{eq:dl_lim} follow easily from the definitions \eqref{eq:djmDef2} and \eqref{eq:DLPsolidangle2}.

Part (b). $\|D_{V^M}\|_{V^M}=\|D_{V_+^M}\|_{V^M}$ by Lemma \ref{lem:DM}, and $\|D_N\|_2\to \|D_{V_+^M}\|_{V^M}$ as $N\to\infty$ by \eqref{eq:limits}, since $\cH_1\subset \cH_2\subset ...$, where $\cH_N$ is the space spanned by $\{\psi_1,...,\psi_N\}$, and
$V_+^M=\overline{\cup_{N=1}^\infty \cH_N}$.
Moreover, it is easy to see that, for each $N\in \NN$, $\|D_N\|_2=\|E_N\|_2$, where $E_N$ is the order $N$ finite section of $E$. That $\|E_N\|_2\to \|E\|_2 = \|e\|_\infty$, with $e$ given by the first of the equalities in \eqref{eq:symb}, are standard properties of infinite Toeplitz matrices with bounded symbols \cite{BottSilb98}. The last equation in \eqref{eq:symb}, with
\begin{equation} \label{eq:erdef}
e_r(t) := -2\ri \sum_{m=1}^\infty (d_m^\prime-M/(2\pi m)) \sin(mt), \quad t\in \R,
\end{equation}
follows since
\begin{equation} \label{eq:FS1}
\sign(t)(\pi-|t|) = 2\sum_{m=1}^\infty \frac{\sin(mt)}{m}, \quad -\pi\leq t \leq \pi,
\end{equation}
and that $e_r$ is continuous and $e_r(0)=0$ follows since the series \eqref{eq:erdef} is absolutely and uniformly convergent by \eqref{eq:dl_lim}. The bound \eqref{eq:enorm} follows since $\|e\|_\infty \geq \lim_{t\to 0}|e(t)|=M/2$.

Part (c). Where $C_N$ is the matrix in Lemma \ref{lem:CNprops}, $D_N=C_N$ if we set $e_{jm}=d^\prime_{|j-m|}$, for $1\leq j,m\leq N$. Thus \eqref{eq:WDNsup} follows from Lemma \ref{lem:CNprops}(iii), and that $w_r(\re^{\ri\theta}D_N)=w_r(\re^{-\ri\theta}D_N)=w_r(-\re^{\ri\theta}D_N)=\|D_N^\theta\|_2$, for $\theta\in \R$, from Lemma \ref{lem:CNprops}(ii) and  (iv). Using \eqref{eq:CNtheta}, we see also that $\left(H^\theta\right)_{jm} = \re^{-\ri\theta}(-1)^{m+1}(D^\theta_N)_{jm}$, $1\leq j,m\leq N$, so that $\|D_N^\theta\|_2=\|H^\theta_N\|_2$, where $H^\theta_N$ is the order $N$ finite section of $H^\theta$. The remaining results up to and including \eqref{eq:hnorm} follow in the same way as we proved \eqref{eq:DNormlim}-\eqref{eq:enorm}, using \eqref{eq:FS1} and that
$$
-t = 2\sum_{m=1}^\infty (-1)^m\frac{\sin(mt)}{m}, \quad -\pi<t < \pi.
$$
The final statement \eqref{eq:finalW} follows from \eqref{eq:limits}, \eqref{eq:WCN}, and \eqref{eq:hnorm}.
\epf

Combining Lemma \ref{lem:DM} and \ref{lem:ANDLPrime} we see that, for $M>0$,
\begin{eqnarray} \label{eq:DMbounds}
\|D^M\|_{L^2(\Gamma^M)} &\geq & \lim_{N\to\infty} \|D_N\|_2 = \|D_{V^M}\|_{V^M} \geq \frac{M}{2} \quad \mbox{and}\\ \label{eq:WDMbounds}
\overline{W(D^M)} &\supset& \overline{\bigcup_{N=1}^\infty W(D_N)} =\overline{W(D_{V^M})} \supset  \left\{\lambda\in \C: |\lambda|\leq M/4\right\},
\end{eqnarray}
so that $w(D^M)\geq M/4$. On the other hand, by \eqref{eq:NormEssNorm} and Theorem \ref{thm:bounded}, for some $C, \mu>0$,
$$
\half \|D^M\|_{L^2(\Gamma^M)}\leq w(D^M)\leq \|D^M\|_{L^2(\Gamma^M)} \leq CM(1+M)^\mu, \quad M>0.
$$
For sufficiently large $M>0$ it appears that the last of the bounds in \eqref{eq:DMbounds} and the last inclusion in \eqref{eq:WDMbounds} are in fact equalities.
Indeed, it follows from \eqref{eq:enorm} and \eqref{eq:hnorm} that \eqref{eq:DMbounds} holds with the last ``$\geq$'' replaced by ``$=$'' if $\|e\|_\infty = M/2$, and \eqref{eq:WDMbounds} holds with the last ``$\supset$'' replaced by ``$=$'' if $\|h_\infty\|_\infty= M/4$, for $\theta\in \R$, and the plots in Figure \ref{fig:symbols} suggest the conjecture that
$$
\|e\|_\infty = M/2, \quad \mbox{at least for $M\geq 1$}, \quad \mbox{and} \quad \|h^\theta\|_\infty = M/4, \quad \mbox{for $\theta\in \R$, at least for $M\geq 2$}.
$$

\begin{figure}
\includegraphics[width=.5\textwidth]{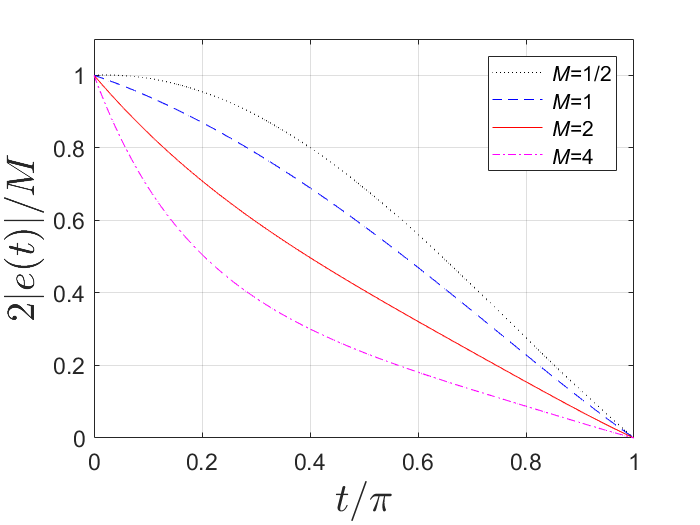}\includegraphics[width=.5\textwidth]{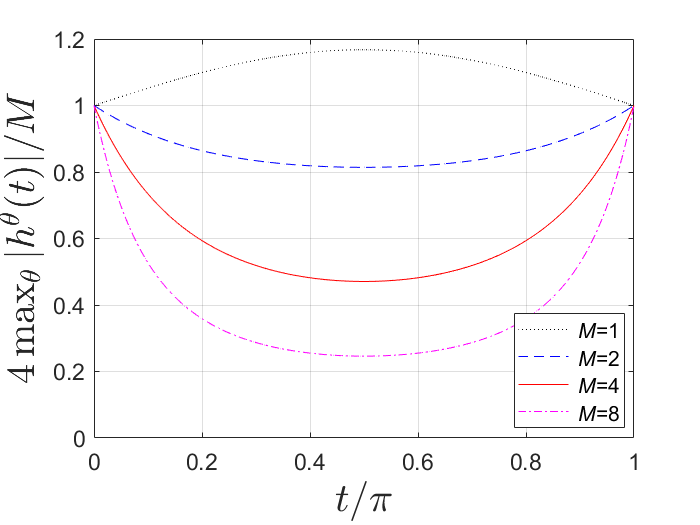}
\caption{\label{fig:symbols} Graphs of $2|e(t)|/M$ and $4\max_{0\leq \theta\leq 2\pi}|h^\theta(t)|/M$ against $t/\pi$, where $e$ and $h^\theta$ are the symbols of the infinite Toeplitz matrices $E$ and $H^\theta$, given by \eqref{eq:symb} and \eqref{eq:hsymb}, respectively.}
\end{figure}

\subsubsection{The proof of Theorems \ref{thm:Q1} and \ref{thm:Q2}}

We now use the above results to prove Theorems \ref{thm:Q1} and \ref{thm:Q2} in the 2-d case, as Theorem \ref{thm:Q1Q22d} below. Let
\begin{equation} \label{eq:GammaMm}
\Gamma^{M,m} := \overline{\bigcup_{j=1}^{2m} \Gamma_j}, \quad m\in \NN,
\end{equation}
where $\Gamma_j$ is as defined in Definition \ref{def:GammaM} (and see Figure \ref{fig:Saw}).

\begin{definition}[$\Gamma^{M}_d$ and $\Omega_d^M$ for $d=2$] \label{def:OmegaM2d} Choose $\beta\in (0,1)$  and, given $M>0$, define $(\gamma_j)_{j=1}^\infty\subset (0,\infty)$ by
$$
\gamma_j := \frac{M(\beta^{j-1}-\beta^j)}{2j}, \quad j\in \NN.
$$
Define $f^M_\beta:\R\to [0,\infty)$ by
\begin{equation} \label{eq:deffmam}
f^M_{\beta}(s) := \left\{\begin{array}{cc}
                             0, & \mbox{if }s\leq 0 \mbox{ or } s\geq 1, \\
                             \gamma_j f_M((s-\beta^j)/\gamma_j),& \mbox{if } \beta^j\leq s < \beta^{j-1},
                           \end{array}
\right.
\end{equation}
for $j\in \NN$, where $f_M$ is defined by \eqref{eq:Gammaeps}, and note that the definition we make for $\gamma_j$ ensures that $f^M_{\beta}\in C^{0,1}(\R)$, with $|(f^M_{\beta})^\prime(s)| = M$ for almost all $s\in (0,1)$, and that
\begin{equation} \label{eq:fMint}
\{f^M_{\beta}(s):\beta^j\leq s\leq \beta^{j-1}\} = (\beta^{j-1},0) + \gamma_j \Gamma^{M,j}, \quad j\in \NN,
\end{equation}
where $\Gamma^{M,j}$ is defined by \eqref{eq:GammaMm}. Let $\varepsilon:=0.1$ and let
\begin{eqnarray} \nonumber
\Gamma^M_{2} &:= &\left\{\left(s,f^M_{\beta}(s)\right):-\varepsilon \leq  s\leq 1+\varepsilon\right\}\\ \label{eq:GammaMbabm}
& = &\{(s,0): -\varepsilon \leq s \leq 0 \mbox{ or } 1\leq s\leq 1+\varepsilon\} \,\cup\, \bigcup_{j\in \NN}\left((\beta^{j-1},0) + \gamma_j \Gamma^{M,j}\right);
\end{eqnarray}
see Figure \ref{fig:GammaM*} for $\Gamma^1_2$ when  $\beta = 0.6$.
Set $\bx^\prime:= (-\varepsilon,-\varepsilon)$ and $\bx^{\prime\prime}:= (1+\varepsilon,-\varepsilon)$, and let
$$
\Omega_2^M := \left\{(x_1,x_2):-\varepsilon< x_1< 1+\varepsilon \mbox{ and } -2\varepsilon < x_2 < f^M_{\beta}(x_1)\right\} \cup B_\varepsilon(\bx^\prime) \cup B_\varepsilon(\bx^{\prime\prime});
$$
see Figure \ref{fig:OmegaM} for $\Omega^1_2$ when $\beta = 0.6$. Note that $\Omega_2^M\subset \R^2$ is a simply-connected Lipschitz domain with Lipschitz constant $M$, and the boundary $\Gamma$ of $\Omega_2^M$ contains $\Gamma^M_2$ and is $C^1$ except at a countable set of points on $\big\{\big(s,f^M_{\beta}(s)\big):0 \leq  s\leq 1\big\}\subset \Gamma^M_2$.
\end{definition}

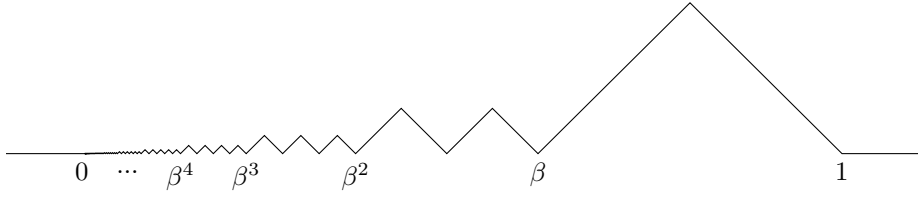
\begin{figure}
\begin{center}
\begin{tikzpicture}[scale=10]
\draw plot coordinates {(-0.1,0)
(0,0) (0.00047018,0) (0.00048063,1.0449e-05) (0.00049108,0) (0.00050153,1.0449e-05) (0.00051198,0) (0.00052243,1.0449e-05) (0.00053288,0) (0.00054332,1.0449e-05) (0.00055377,0) (0.00056422,1.0449e-05) (0.00057467,0) (0.00058512,1.0449e-05) (0.00059557,0) (0.00060602,1.0449e-05) (0.00061646,0) (0.00062691,1.0449e-05) (0.00063736,0) (0.00064781,1.0449e-05) (0.00065826,0) (0.00066871,1.0449e-05) (0.00067916,0) (0.0006896,1.0449e-05) (0.00070005,0) (0.0007105,1.0449e-05) (0.00072095,0) (0.0007314,1.0449e-05) (0.00074185,0) (0.0007523,1.0449e-05) (0.00076274,0) (0.00077319,1.0449e-05) (0.00078364,0) (0.0008023,1.8658e-05) (0.00082096,0) (0.00083962,1.8658e-05) (0.00085827,0) (0.00087693,1.8658e-05) (0.00089559,0) (0.00091425,1.8658e-05) (0.00093291,0) (0.00095156,1.8658e-05) (0.00097022,0) (0.00098888,1.8658e-05) (0.0010075,0) (0.0010262,1.8658e-05) (0.0010449,0) (0.0010635,1.8658e-05) (0.0010822,0) (0.0011008,1.8658e-05) (0.0011195,0) (0.0011381,1.8658e-05) (0.0011568,0) (0.0011755,1.8658e-05) (0.0011941,0) (0.0012128,1.8658e-05) (0.0012314,0) (0.0012501,1.8658e-05) (0.0012688,0) (0.0012874,1.8658e-05) (0.0013061,0) (0.0013396,3.3489e-05) (0.001373,0) (0.0014065,3.3489e-05) (0.00144,0) (0.0014735,3.3489e-05) (0.001507,0) (0.0015405,3.3489e-05) (0.001574,0) (0.0016075,3.3489e-05) (0.001641,0) (0.0016744,3.3489e-05) (0.0017079,0) (0.0017414,3.3489e-05) (0.0017749,0) (0.0018084,3.3489e-05) (0.0018419,0) (0.0018754,3.3489e-05) (0.0019089,0) (0.0019424,3.3489e-05) (0.0019758,0) (0.0020093,3.3489e-05) (0.0020428,0) (0.0020763,3.3489e-05) (0.0021098,0) (0.0021433,3.3489e-05) (0.0021768,0) (0.0022372,6.0466e-05) (0.0022977,0) (0.0023582,6.0466e-05) (0.0024186,0) (0.0024791,6.0466e-05) (0.0025396,0) (0.0026,6.0466e-05) (0.0026605,0) (0.002721,6.0466e-05) (0.0027814,0) (0.0028419,6.0466e-05) (0.0029024,0) (0.0029628,6.0466e-05) (0.0030233,0) (0.0030838,6.0466e-05) (0.0031442,0) (0.0032047,6.0466e-05) (0.0032652,0) (0.0033256,6.0466e-05) (0.0033861,0) (0.0034466,6.0466e-05) (0.003507,0) (0.0035675,6.0466e-05) (0.003628,0) (0.0037379,0.00010994) (0.0038478,0) (0.0039578,0.00010994) (0.0040677,0) (0.0041777,0.00010994) (0.0042876,0) (0.0043975,0.00010994) (0.0045075,0) (0.0046174,0.00010994) (0.0047274,0) (0.0048373,0.00010994) (0.0049472,0) (0.0050572,0.00010994) (0.0051671,0) (0.005277,0.00010994) (0.005387,0) (0.0054969,0.00010994) (0.0056069,0) (0.0057168,0.00010994) (0.0058267,0) (0.0059367,0.00010994) (0.0060466,0) (0.0062482,0.00020155) (0.0064497,0) (0.0066513,0.00020155) (0.0068528,0) (0.0070544,0.00020155) (0.0072559,0) (0.0074575,0.00020155) (0.007659,0) (0.0078606,0.00020155) (0.0080622,0) (0.0082637,0.00020155) (0.0084653,0) (0.0086668,0.00020155) (0.0088684,0) (0.0090699,0.00020155) (0.0092715,0) (0.009473,0.00020155) (0.0096746,0) (0.0098761,0.00020155) (0.010078,0) (0.010451,0.00037325) (0.010824,0) (0.011197,0.00037325) (0.011571,0) (0.011944,0.00037325) (0.012317,0) (0.01269,0.00037325) (0.013064,0) (0.013437,0.00037325) (0.01381,0) (0.014183,0.00037325) (0.014557,0) (0.01493,0.00037325) (0.015303,0) (0.015676,0.00037325) (0.01605,0) (0.016423,0.00037325) (0.016796,0) (0.017496,0.00069984) (0.018196,0) (0.018896,0.00069984) (0.019596,0) (0.020295,0.00069984) (0.020995,0) (0.021695,0.00069984) (0.022395,0) (0.023095,0.00069984) (0.023795,0) (0.024494,0.00069984) (0.025194,0) (0.025894,0.00069984) (0.026594,0) (0.027294,0.00069984) (0.027994,0) (0.029327,0.001333) (0.03066,0) (0.031993,0.001333) (0.033326,0) (0.034659,0.001333) (0.035992,0) (0.037325,0.001333) (0.038658,0) (0.039991,0.001333) (0.041324,0) (0.042657,0.001333) (0.04399,0) (0.045323,0.001333) (0.046656,0) (0.049248,0.002592) (0.05184,0) (0.054432,0.002592) (0.057024,0) (0.059616,0.002592) (0.062208,0) (0.0648,0.002592) (0.067392,0) (0.069984,0.002592) (0.072576,0) (0.075168,0.002592) (0.07776,0) (0.082944,0.005184) (0.088128,0) (0.093312,0.005184) (0.098496,0) (0.10368,0.005184) (0.10886,0) (0.11405,0.005184) (0.11923,0) (0.12442,0.005184) (0.1296,0) (0.1404,0.0108) (0.1512,0) (0.162,0.0108) (0.1728,0) (0.1836,0.0108) (0.1944,0) (0.2052,0.0108) (0.216,0) (0.24,0.024) (0.264,0) (0.288,0.024) (0.312,0) (0.336,0.024) (0.36,0) (0.42,0.06) (0.48,0) (0.54,0.06) (0.6,0) (0.8,0.2) (1,0)
(1.1,0)};
\draw (0,0) node[anchor=north] {$0$};
\draw (1,0) node[anchor=north] {$1$};
\draw (0.6,0) node[anchor=north] {$\beta$};
\draw (0.36,0) node[anchor=north] {$\beta^2$};
\draw (0.216,0) node[anchor=north] {$\beta^3$};
\draw (0.1296,0) node[anchor=north] {$\beta^4$};
\draw (0.06,-0.01) node[anchor=north] {$...$};
\end{tikzpicture}
\end{center}
\caption{\label{fig:GammaM*} The curve $\Gamma^M_2$, as specified in Definition \ref{def:OmegaM2d}, in the case $M=1$ and $\beta=0.6$. The labels are the $x_1$-coordinates of the point $\bze=(0,0)$ and of the first 5 of the points $(\beta^j,0)$, $j=0,1,...$. All these points lie on $\Gamma^M_2$.}
\end{figure}

The proofs of Theorems \ref{thm:Q1} and \ref{thm:Q2}, in both the 2-d and 3-d cases, depend on \eqref{eq:DMbounds} and \eqref{eq:WDMbounds}, and on the localisation result Theorem \ref{lem:local}. They also depend on the simple observation that the norm and numerical range of the double-layer potential on a curve or surface $\Gamma^\prime$ are the same as those of the double-layer potential operator on its translate, $\Gamma^\prime + \bx$, for $\bx\in \R^d$. These quantities are also invariant under scaling in the sense of the following lemma.

\begin{lemma} \label{lem:dil2} Suppose that $\Gamma^\prime := \{(\by^\prime,f(\by^\prime)):\by^\prime=(y_1,...,y_{d-1}) \in \overline{\cN}\}$, for some open $\cN\subset \R^{d-1}$, $d=2$ or $3$, and some $f\in C^{0,1}(\R^{d-1})$, and let $D_\kappa$ denote the double-layer potential operator on $\kappa \Gamma^\prime$ for $\kappa>0$. Then
$$
\|D_\kappa\|_{L^2(\kappa \Gamma^\prime)} = \|D_1\|_{L^2(\Gamma^\prime)} \;\mbox{ and } \; W(D_\kappa) = W(D_1), \quad \kappa>0.
$$
\end{lemma}
\begin{proof} Define $V:L^2(\kappa\Gamma^\prime)\to L^2(\Gamma^\prime)$ by $V\phi(\by)=\kappa^{(d-1)/2}\phi(\kappa\by)$, for $\phi\in L^2(\Gamma^\prime)$, $\by\in \Gamma^\prime$. Then, arguing exactly as in the proofs of Lemmas \ref{lem:Visometry} and \ref{lem:Vcommute}, we see that $V$ is an isometric isomorphism and $D_1V=VD_\kappa$, so that $D_1$ and $D_\kappa$ are unitarily equivalent, and the result follows.
\end{proof}

\begin{theorem}(Theorems \ref{thm:Q1} and \ref{thm:Q2} in the 2-d case) \label{thm:Q1Q22d}  Suppose that $\Gamma$ is the boundary of $\Omega_2^M$, defined as in Definition \ref{def:OmegaM2d}, for some $M>0$ and $\beta\in (0,1)$. Then
$$
\|D\|_{L^2(\Gamma),\ess} \geq \frac{M}{2} \quad \mbox{and} \quad W_\ess(D) \supset \{\lambda\in \C: |\lambda| \leq M/4\}.
$$
\end{theorem}
\begin{proof}
Let $\bx^*:= \bze\in \Gamma^M_2\subset \Gamma$. By Theorem \ref{lem:local},
$$
\|D\|_{L^2(\Gamma),\ess} \geq \lim_{\delta\to 0}\|D_{\bx^*,\delta}\|_\LtG \quad \mbox{and} \quad W_\ess(D) \supset \bigcap_{\delta>0} \overline{W(D_{\bx^*,\delta})}.
$$
Thus the result follows if we show that $\|D_{\bx^*,\delta}\|_\LtG\geq M/2$ and $\overline{W(D_{\bx^*,\delta})}\supset \{\lambda\in \C: |\lambda| \leq M/4\}$. By \eqref{eq:DMbounds} and \eqref{eq:WDMbounds}, this in turn follows if we can show that $\|D_{\bx^*,\delta}\|_\LtG\geq \|D_N\|_2$ and $\overline{W(D_{\bx^*,\delta})}\supset W(D_N)$, for every $N\in \NN$ and $\delta>0$, where $D_N$ is as defined in Lemma \ref{lem:ANDLPrime}. Given $m\in \NN$, set $\widetilde \Gamma:= \Gamma^{M,m}$, defined by \eqref{eq:GammaMm}. Then, by \eqref{eq:DNsubsets} and \eqref{eq:DNIncl}, $\|D_N\|_2 \leq \|\widetilde D\|_{L^2(\widetilde \Gamma)}$ and $W(D_N)\subset W(\widetilde D)$, for $N=1,...,2m$, where $\widetilde D$ denotes the double-layer potential operator on $\widetilde \Gamma$. But, by construction of $\Gamma^M_2$ in Definition \ref{def:OmegaM2d} (see \eqref{eq:fMint}), for every $\delta>0$ there exists $\kappa>0$ and $s\in \R$ such that $\widehat \Gamma := (s,0)+\kappa \widetilde \Gamma \subset B_\delta(\bx^*)\cap \Gamma$, so that
$$
\|\widehat D\|_{L^2(\widehat \Gamma)}  \leq \|D_{\bx^*,\delta}\|_\LtG \quad \mbox{and} \quad W(\widehat D)\subset W(D_{\bx^*,\delta}),
$$
by \eqref{eq:subspaces}, where $\widehat D$ denotes the double-layer potential operator on $\widehat \Gamma$.
Furthermore, by Lemma \ref{lem:dil2},
$$
\|\widehat D\|_{L^2(\widehat \Gamma)} = \|\widetilde D\|_{L^2(\widetilde \Gamma)} \quad \mbox{and} \quad W(\widehat D) = W(\widetilde D),
$$
so that
$$
\|D_{\bx^*,\delta}\|_\LtG\geq \|D_N\|_2 \quad \mbox{and} \quad W(D_{\bx^*,\delta})\supset W(D_N),
$$
for $N=1,...,2m$. Since this holds for every $m\in \NN$ and $\delta>0$, the proof is complete.
\end{proof}

\subsection{Proof of Theorems \ref{thm:Q1} and \ref{thm:Q2} in the 3-d case} \label{sec:Q1Q2_3d}
We now prove Theorems \ref{thm:Q1} and \ref{thm:Q2} in the 3-d case as Theorem \ref{thm:Q1Q23d} below, which proves the bounds \eqref{eq:Q1Q2bounds} for the domain $\Omega_3^M$ specified in the following definition. The proof builds on the 2-d case. Indeed the bounds \eqref{eq:Q1Q2bounds} hold if we define $\Omega_3^d$ alternatively as
\begin{equation} \label{eq:simpledef}
[-1,1]\times \Omega_2^M,
\end{equation}
where $\Omega_2^M$ is as in Definition \ref{def:OmegaM2d}; all that is needed for \eqref{eq:Q1Q2bounds} to hold is that $\Gamma$, the boundary of $\Omega_3^M$, contains $[-\epsilon,\epsilon]\times \Gamma_2^M$, for some $\epsilon>0$. But $\Omega_3^M$ given by the simple definition \eqref{eq:simpledef} has Lipschitz constant larger than $M$. The following more elaborate definition, which makes  a smoother cut-off in the $x_1$-direction, constructs $\Omega_3^M$ so that $[-1,1]\times \Gamma_2^M\subset \Gamma$ while ensuring that the Lipschitz constant of $\Omega_3^M$ does not exceed $M$.

\begin{figure}
\begin{center}
\includegraphics[width=0.7\textwidth]{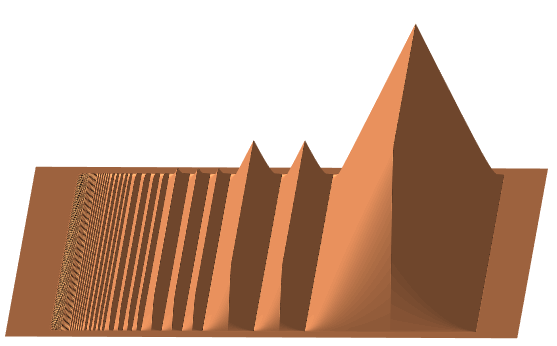}
\end{center}
\caption{\label{fig:GM3} View from above of $\Gamma_3^M$ when $M=2$ and $\beta=0.6$.}.
\end{figure}

\begin{definition}[$\Gamma_{d}^M$ and $\Omega_d^M$ for $d=3$] \label{def:OmegaM3d} For $M>0$, define $F_M:\R\to [0,\infty)$ by
\begin{equation} \label{eq:FlowM}
F_M(s) = \left\{\begin{array}{ll}
                  1, & |s|\leq 1, \\
                  \cos(2(|s|-1)/(1-\beta)), & 1<|s|<1+\vartheta, \\
                  0, & |s|\geq 1+\vartheta,
                \end{array}
\right.
\end{equation}
where $\vartheta:= \pi(1-\beta)/4$.
Fix $\beta\in (0,1)$ and, for $M>0$, define $f^M_{\text{3-d}}:\R^2\to [0,\infty)$ by
\begin{equation} \label{eq:fM3d}
f^M_{\text{3-d}}(x_1,x_2) = F_M(x_1)f^M_{\beta}(x_2), \quad x_1,x_2\in \R,
\end{equation}
where $f^M_{\beta}$ is defined by \eqref{eq:deffmam}. Then $F_M\in C^{0,1}(\R)$ and $f^M_{\text{3-d}}\in C^{0,1}(\R^2)$. Further, since $|(f^M_{\beta})^\prime(s)|=M$ and $0\leq f^M_{\beta}(s)\leq \gamma_1=M(1-\beta)/2$, for almost all $s\in [0,1]$, $|\nabla f^M_{\text{3-d}}|\leq M$, for almost all $(x_1,x_2)\in \R^2$, with equality for almost all $(x_1,x_2)\in (-1,1)\times(0,1)$. Let $\varepsilon := 0.1$ and let
\begin{equation} \label{eq:GM*3d}
\Gamma_{3}^M := \left\{\left(x_1,x_2,f^M_{\text{3-d}}(x_1,x_2)\right):  |x_1| \leq 1+\varepsilon+\vartheta, -\varepsilon \leq x_2\leq 1+ \varepsilon \right\}.
\end{equation}
$\Gamma_3^M$ is shown in Figure \ref{fig:GM3};
note also that, where $\Gamma_2^{M}$ is as in Definition \ref{def:OmegaM2d},
\begin{equation} \label{eq:keyG3M}
\left\{\bx=(x_1,x_2,x_3)\in\Gamma_{3}^M:|x_1|\leq 1\right\} = [-1,1] \times \Gamma_2^{M},
\end{equation}
so that $\Gamma_2^{M}$, shown in Figure \ref{fig:GammaM*}, is a cross-section of $\Gamma_{3}^M$ through the plane $x_1=c$, for any $c\in [-1,1]$. Let $C:=2+2\varepsilon + \vartheta$, so that
$$
[-1-\varepsilon-\vartheta,1+\varepsilon+\vartheta]\times [-\varepsilon,1+ \varepsilon] \subset
E_C:=
 \left\{(x_1,x_2)\in \R^2: x_1^2+x_2^2< C^2\right\}.
$$
Let $G \subset \{ (x_1,x_2,x_3) \in \Rea^3 : x_3<0\}$ be any $C^1$ domain such that $E_C\times\{0\}\subset \partial G$,
 e.g.,
\begin{equation} \label{eq:GC1def}
G := \left\{\bx=(x_1,x_2,x_3)\in \R^3:r:=(x_1^2+x_2^2)^{1/2}  <  C+\varepsilon, -g(r)-\varepsilon < x_3< g(r)-\varepsilon\right\},
\end{equation}
where
$$
g(s) := \left\{\begin{array}{ll}
                  \varepsilon, & 0\leq s\leq C, \\
                  \left(\varepsilon^2-(s-C)^2\right)^{1/2}, & C< s\leq C+\varepsilon.
                \end{array}
\right.
$$
Set
$$
\Omega_3^M := G \, \cup \, \left\{(x_1,x_2,x_3)\in \R^3: |x_1| < 1+\varepsilon+\vartheta, -\varepsilon < x_2< 1+ \varepsilon, \,
0\leq   x_3 < f^M_{\text{3-d}}(x_1,x_2)\right\}.
$$
Then $\Gamma_{3}^M\subset \Gamma$, where $\Gamma$ is the boundary of $\Omega_3^M$, $\Omega_3^M$ is a simply-connected Lipschitz domain with Lipschitz constant $M$, and $\Gamma$ is locally $C^1$ at every point on $\overline{\Gamma\setminus\Gamma_{3}^M}$.
\end{definition}

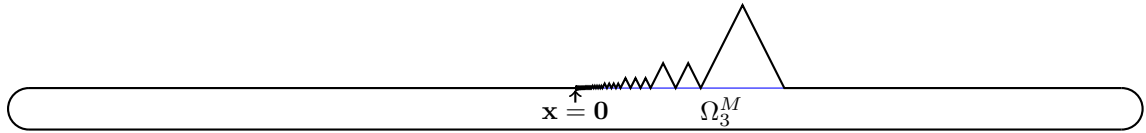
\begin{figure}[h!]
\begin{tikzpicture}[scale=2.75]
\draw[blue,thin] plot coordinates {(0,0) (1,0)};
\draw[thick] (-2.6142,0) arc(90:270:0.1);
\draw[thick] (2.6142,-0.2) arc(-90:90:0.1);
\draw[thick] plot coordinates {(-2.6142,-0.2) (2.6142,-0.2)};
\draw[thick] plot coordinates {(-2.6142,0)
 (0,0) (0.00047018,0) (0.00048063,2.0897e-05) (0.00049108,0) (0.00050153,2.0897e-05) (0.00051198,0) (0.00052243,2.0897e-05) (0.00053288,0) (0.00054332,2.0897e-05) (0.00055377,0) (0.00056422,2.0897e-05) (0.00057467,0) (0.00058512,2.0897e-05) (0.00059557,0) (0.00060602,2.0897e-05) (0.00061646,0) (0.00062691,2.0897e-05) (0.00063736,0) (0.00064781,2.0897e-05) (0.00065826,0) (0.00066871,2.0897e-05) (0.00067916,0) (0.0006896,2.0897e-05) (0.00070005,0) (0.0007105,2.0897e-05) (0.00072095,0) (0.0007314,2.0897e-05) (0.00074185,0) (0.0007523,2.0897e-05) (0.00076274,0) (0.00077319,2.0897e-05) (0.00078364,0) (0.0008023,3.7316e-05) (0.00082096,0) (0.00083962,3.7316e-05) (0.00085827,0) (0.00087693,3.7316e-05) (0.00089559,0) (0.00091425,3.7316e-05) (0.00093291,0) (0.00095156,3.7316e-05) (0.00097022,0) (0.00098888,3.7316e-05) (0.0010075,0) (0.0010262,3.7316e-05) (0.0010449,0) (0.0010635,3.7316e-05) (0.0010822,0) (0.0011008,3.7316e-05) (0.0011195,0) (0.0011381,3.7316e-05) (0.0011568,0) (0.0011755,3.7316e-05) (0.0011941,0) (0.0012128,3.7316e-05) (0.0012314,0) (0.0012501,3.7316e-05) (0.0012688,0) (0.0012874,3.7316e-05) (0.0013061,0) (0.0013396,6.6978e-05) (0.001373,0) (0.0014065,6.6978e-05) (0.00144,0) (0.0014735,6.6978e-05) (0.001507,0) (0.0015405,6.6978e-05) (0.001574,0) (0.0016075,6.6978e-05) (0.001641,0) (0.0016744,6.6978e-05) (0.0017079,0) (0.0017414,6.6978e-05) (0.0017749,0) (0.0018084,6.6978e-05) (0.0018419,0) (0.0018754,6.6978e-05) (0.0019089,0) (0.0019424,6.6978e-05) (0.0019758,0) (0.0020093,6.6978e-05) (0.0020428,0) (0.0020763,6.6978e-05) (0.0021098,0) (0.0021433,6.6978e-05) (0.0021768,0) (0.0022372,0.00012093) (0.0022977,0) (0.0023582,0.00012093) (0.0024186,0) (0.0024791,0.00012093) (0.0025396,0) (0.0026,0.00012093) (0.0026605,0) (0.002721,0.00012093) (0.0027814,0) (0.0028419,0.00012093) (0.0029024,0) (0.0029628,0.00012093) (0.0030233,0) (0.0030838,0.00012093) (0.0031442,0) (0.0032047,0.00012093) (0.0032652,0) (0.0033256,0.00012093) (0.0033861,0) (0.0034466,0.00012093) (0.003507,0) (0.0035675,0.00012093) (0.003628,0) (0.0037379,0.00021988) (0.0038478,0) (0.0039578,0.00021988) (0.0040677,0) (0.0041777,0.00021988) (0.0042876,0) (0.0043975,0.00021988) (0.0045075,0) (0.0046174,0.00021988) (0.0047274,0) (0.0048373,0.00021988) (0.0049472,0) (0.0050572,0.00021988) (0.0051671,0) (0.005277,0.00021988) (0.005387,0) (0.0054969,0.00021988) (0.0056069,0) (0.0057168,0.00021988) (0.0058267,0) (0.0059367,0.00021988) (0.0060466,0) (0.0062482,0.00040311) (0.0064497,0) (0.0066513,0.00040311) (0.0068528,0) (0.0070544,0.00040311) (0.0072559,0) (0.0074575,0.00040311) (0.007659,0) (0.0078606,0.00040311) (0.0080622,0) (0.0082637,0.00040311) (0.0084653,0) (0.0086668,0.00040311) (0.0088684,0) (0.0090699,0.00040311) (0.0092715,0) (0.009473,0.00040311) (0.0096746,0) (0.0098761,0.00040311) (0.010078,0) (0.010451,0.0007465) (0.010824,0) (0.011197,0.0007465) (0.011571,0) (0.011944,0.0007465) (0.012317,0) (0.01269,0.0007465) (0.013064,0) (0.013437,0.0007465) (0.01381,0) (0.014183,0.0007465) (0.014557,0) (0.01493,0.0007465) (0.015303,0) (0.015676,0.0007465) (0.01605,0) (0.016423,0.0007465) (0.016796,0) (0.017496,0.0013997) (0.018196,0) (0.018896,0.0013997) (0.019596,0) (0.020295,0.0013997) (0.020995,0) (0.021695,0.0013997) (0.022395,0) (0.023095,0.0013997) (0.023795,0) (0.024494,0.0013997) (0.025194,0) (0.025894,0.0013997) (0.026594,0) (0.027294,0.0013997) (0.027994,0) (0.029327,0.0026661) (0.03066,0) (0.031993,0.0026661) (0.033326,0) (0.034659,0.0026661) (0.035992,0) (0.037325,0.0026661) (0.038658,0) (0.039991,0.0026661) (0.041324,0) (0.042657,0.0026661) (0.04399,0) (0.045323,0.0026661) (0.046656,0) (0.049248,0.005184) (0.05184,0) (0.054432,0.005184) (0.057024,0) (0.059616,0.005184) (0.062208,0) (0.0648,0.005184) (0.067392,0) (0.069984,0.005184) (0.072576,0) (0.075168,0.005184) (0.07776,0) (0.082944,0.010368) (0.088128,0) (0.093312,0.010368) (0.098496,0) (0.10368,0.010368) (0.10886,0) (0.11405,0.010368) (0.11923,0) (0.12442,0.010368) (0.1296,0) (0.1404,0.0216) (0.1512,0) (0.162,0.0216) (0.1728,0) (0.1836,0.0216) (0.1944,0) (0.2052,0.0216) (0.216,0) (0.24,0.048) (0.264,0) (0.288,0.048) (0.312,0) (0.336,0.048) (0.36,0) (0.42,0.12) (0.48,0) (0.54,0.12) (0.6,0) (0.8,0.4) (1,0)
(2.6142,0)};

\draw (0.7,0) node[anchor=north] { $\Omega_3^M$};
\draw [thick,->] (0,-0.07) -- ++(90:0.06);
\draw (0,-0.02) node[anchor=north] { $\bx = \bze$};
\end{tikzpicture}
\caption{\label{fig:OmegaM3d} Cross-section in the plane $x_1=0$ through the domain $\Omega_3^M$, as specified in Definition \ref{def:OmegaM3d}, for Lipschitz constant $M=2$ (and $\beta=0.6$). The thinner blue line shows the part of the boundary of the $C^1$ domain $G$ given by \eqref{eq:GC1def} in Definition \ref{def:OmegaM3d} that is not shared with $\Omega_3^M$.
See also Figure \ref{fig:GM3} for the key, non-$C^1$, part $\Gamma_3^M$ of the upper surface of $\Omega_3^M$.}
\end{figure}

For $a>0$ and $m\in \NN$ let
\begin{equation}
\Gamma_{m,a} := (-a,a)\times \Gamma_m \quad \mbox{and} \quad \Gamma^{M,m}_a := [-a,a] \times \Gamma^{M,m},
\end{equation}
where $\Gamma_m$ and $\Gamma^{M,m}$ are as  in Definition \ref{def:GammaM} and \eqref{eq:GammaMm}, respectively.

\begin{lemma}[3-d version of Lemma \ref{lem:ANDLPrime}] \label{lem:3dvANDLPrime} Given $N\in \NN$ and $a>0$, define the orthonormal set $\{\psi_1,...,\psi_N\}\subset [-a,a]\times \Gamma^M$ by \eqref{eq:defphi2d}, but with $\Gamma_m$ replaced by $\Gamma_{m,a}$, and define the Galerkin matrix $D_{N,a}$ by \eqref{eq:DNdef} (with $D_N$ replaced by $D_{N,a}$), where $D$ is the double-layer potential operator on $[-a,a]\times \Gamma^M$. Then
\beq\label{eq:DNformP3d}
\big( D_{N,a}\big)_{jm}= \big( B_N\big)_{jm}d^\prime_{j-m,a}, \quad 1\leq j,m\leq N,
\eeq
where $d^\prime_{0,a}:= 0$,  and, for $\ell\in \NN$,
\begin{eqnarray} \label{eq:djmDef3d}
d^\prime_{-\ell,a}:=d^\prime_{\ell,a} &:=& \frac{1}{\left|\Gamma_{\ell,a}\right|}\int_{\Gamma_{\ell,a}}\int_{\Gamma_{0,a}} \left|\frac{\partial \Phi(\bx,\by)}{\partial n(\by)}\right| \rd s(\by)\rd s(\bx)\\ \label{eq:3dto2d}
 &=& \frac{1}{2\pi \left|\Gamma_{\ell}\right|}\int_{\Gamma_{\ell}}\int_{\Gamma_{0}}\frac{|(\by-\bx)\cdot \bn(\by)|}{|\bx-\by|^2}\left(1+\frac{|\bx-\by|^2}{4a^2}\right)^{1/2} \rd s(\by)\rd s(\bx).
\end{eqnarray}
Furthermore,
\begin{equation} \label{eq:dlbounds}
d^\prime_\ell \leq d^\prime_{\ell,a} \leq
\left(1+\frac{1}{4a^2}\left(1+ \frac{(1+\ell)^2}{M^2}\right)\right)^{1/2}
d^\prime_\ell, \quad \ell\in \Z, \;\; a>0,
\end{equation}
where $d^\prime_\ell$ is as defined in Lemma \ref{lem:ANDLPrime}, and thus
\begin{equation} \label{eq:DNalim}
D_{N,a} \to D_N \quad \tas \quad a\to\infty,
\end{equation}
where $D_N$ is defined by \eqref{eq:DNformP}.
\end{lemma}
\begin{proof}
That \eqref{eq:DNformP3d} and \eqref{eq:djmDef3d} hold follows exactly as in the proof of Lemma \ref{lem:ANDLPrime}. To see \eqref{eq:3dto2d}, note that \eqref{eq:djmDef3d} implies, since $|\Gamma_{\ell,a}|=2a|\Gamma_\ell|$, that
\begin{eqnarray*}
d^\prime_{\ell,a} &=& \frac{1}{8\pi a\left|\Gamma_{\ell}\right|}\int_{\Gamma_{\ell,a}}\int_{\Gamma_{0,a}} \frac{|(\widehat\by-\widehat \bx)\cdot \bn(\widehat \by)|}{|\widehat\bx-\widehat\by|^3}\rd s(\widehat \by)\rd s(\widehat\bx), \quad \ell\in \NN.
\end{eqnarray*}
Writing $\widehat \bx = x_1 \be_1 + \bx$ and $\widehat \by = y_1 \be_1 + \by$, where  $\be_1$ is the unit vector in the $x_1$-direction and $x_1:=\be_1\cdot \widehat \bx$, $y_1 := \be_1\cdot \widehat \by$, it follows, since $\be_1\cdot \bn(\widehat \by)=0$, that, for $\ell\in\NN$ and $a>0$,
\begin{eqnarray*}
d^\prime_{\ell,a} &=& \frac{1}{8\pi a\left|\Gamma_{\ell}\right|}\int_{\Gamma_{\ell}}\int_{\Gamma_{0}}\int_{-a}^a\int_{-a}^a \frac{|(\by-\bx)\cdot \bn(\by)|}{\left(|\bx-\by|^2+(x_1-y_1)^2\right)^{3/2}}\rd x_1\rd y_1\rd s(\by)\rd s(\bx).
\end{eqnarray*}
Now
\begin{eqnarray*}
\int_{-a}^a \frac{\rd x_1}{\left(|\bx-\by|^2+(x_1-y_1)^2\right)^{3/2}} = \frac{1}{|\bx-\by|^2}\left(\frac{a-y_1}{\sqrt{|\bx-\by|^2+(a-y_1)^2}}+\frac{a+y_1}{\sqrt{|\bx-\by|^2+(a+y_1)^2}}\right)
\end{eqnarray*}
and
\begin{eqnarray*}
\int_{-a}^a \frac{a\pm y_1}{\sqrt{|\bx-\by|^2+(a\pm y_1)^2}}\rd y_1 = \sqrt{|\bx-\by|^2+4a^2},
\end{eqnarray*}
so \eqref{eq:3dto2d} follows. Using \eqref{eq:djmDef2} and \eqref{eq:DLPsolidangle2}, we see that
\begin{eqnarray}\label{eq:dell}
d^\prime_{\ell} &=& \frac{1}{2\pi \left|\Gamma_{\ell}\right|}\int_{\Gamma_{\ell}}\int_{\Gamma_{0}}\frac{|(\by-\bx)\cdot \bn(\by)|}{|\bx-\by|^2} \rd s(\by)\rd s(\bx), \quad \ell\in \NN.
\end{eqnarray}
For $\bx\in \Gamma_0$, $\by\in \Gamma_\ell$, $|\bx-\by|^2\leq 1+(1+\ell)^2/M^2$.
Combining this inequality with \eqref{eq:3dto2d} and \eqref{eq:dell}, we obtain the bounds \eqref{eq:dlbounds} for $\ell\in \NN$. These bounds also hold for $\ell\in \Z$, since $d^\prime_{0,a}=d^\prime_0=0$ and $d^\prime_{-\ell,a}=d^\prime_{\ell,a}$ and $d^\prime_{-\ell}=d^\prime_{\ell}$, for $\ell\in \NN$ and $a>0$. Clearly, \eqref{eq:dlbounds} implies \eqref{eq:DNalim}.
\end{proof}

The proof of the following theorem follows that of Theorem \ref{thm:Q1Q22d}, but also uses Lemma \ref{lem:3dvANDLPrime}.
\begin{theorem}[Theorems \ref{thm:Q1} and \ref{thm:Q2} in the 3-d case] \label{thm:Q1Q23d}  Suppose that $\Gamma$ is the boundary of $\Omega_3^M$, defined as in Definition \ref{def:OmegaM2d}, for some $M>0$ and $\beta\in (0,1)$. Then
\begin{equation} \label{eq:Q1Q2bounds}
\|D\|_{L^2(\Gamma),\ess} \geq \frac{M}{2} \quad \mbox{and} \quad W_\ess(D) \supset \{\lambda\in \C: |\lambda| \leq M/4\}.
\end{equation}
\end{theorem}
\begin{proof}
Let $\bx^*:= \bze\in \Gamma^M_3\subset \Gamma$. Arguing as in the proof of Theorem \ref{thm:Q1Q22d}, we obtain that
$$
\|D\|_{L^2(\Gamma),\ess} \geq \lim_{\delta\to 0}\|D_{\bx^*,\delta}\|_\LtG \quad \mbox{and} \quad W_\ess(D) \supset \bigcap_{\delta>0} \overline{W(D_{\bx^*,\delta})},
$$
by Theorem \ref{lem:local}.
Hence the result follows if we show that $\|D_{\bx^*,\delta}\|_\LtG\geq M/2-\epsilon$ and $\overline{W(D_{\bx^*,\delta})}\supset \{\lambda\in \C: |\lambda| \leq M/4-\epsilon\}$, for every $\epsilon>0$ and $\delta>0$.

So suppose that $\epsilon>0$ and $\delta>0$. With $D_N$ as defined in Lemma \ref{lem:DM}, $\|D_1\|_2\leq \|D_2\|_2\leq ...$ and $W(D_1)\subset W(D_2)\subset ...$ by \eqref{eq:DNsubsets}. It follows from \eqref{eq:DMbounds} and  \eqref{eq:WDMbounds} and the convexity of each $W(D_N)$ that, for some $N\in \NN$, $\|D_N\|_2\geq M/2-\epsilon/2$ and $W(D_N)\supset \{\lambda\in \C: |\lambda| \leq M/4-\epsilon/2\}$ and then, by Lemma \ref{lem:3dvANDLPrime}, that there exists $a>0$ such that  $\|D_{N,a}\|_2\geq M/2-\epsilon$ and $W(D_{N,a})\supset \{\lambda\in \C: |\lambda| \leq M/4-\epsilon\}$. So the result is proved if we can show that $\|D_{\bx^*,\delta}\|_\LtG\geq \|D_{N,a}\|_2$ and $W(D_{\bx^*,\delta})\supset W(D_{N,a})$.

Setting $\widetilde \Gamma := \Gamma_a^{M,m}$ for some $m\geq N/2$, it follows from \eqref{eq:subspaces} that $\|D_{N,a}\|_2 \leq \|\widetilde D\|_{L^2(\widetilde \Gamma)}$ and $W(D_{N,a})\subset W(\widetilde D)$, where $\widetilde D$ denotes the double-layer potential operator on $\widetilde \Gamma$.
 But, by construction of $\Gamma^M_3$ in Definition \ref{def:OmegaM3d} (see \eqref{eq:keyG3M} and \eqref{eq:fMint}), there exists $\kappa>0$ and $s\in \R$ such that $\widehat \Gamma := (0,s,0)+\kappa \widetilde \Gamma \subset B_\delta(\bx^*)\cap \Gamma_3^M\subset B_\delta(\bx^*)\cap \Gamma$, so that
$$
\|\widehat D\|_{L^2(\widehat \Gamma)}  \leq \|D_{\bx^*,\delta}\|_\LtG \quad \mbox{and} \quad W(\widehat D)\subset W(D_{\bx^*,\delta}),
$$
by \eqref{eq:subspaces}, where $\widehat D$ denotes the double-layer potential operator on $\widehat \Gamma$.
But also, by Lemma \ref{lem:dil2},
$$
\|\widehat D\|_{L^2(\widehat \Gamma)} = \|\widetilde D\|_{L^2(\widetilde \Gamma)} \quad \mbox{and} \quad W(\widehat D) = W(\widetilde D),
$$
so that
$$
\|D_{\bx^*,\delta}\|_\LtG\geq \|D_{N,a}\|_2 \quad \mbox{and} \quad W(D_{\bx^*,\delta})\supset W(D_{N,a})
$$
and the proof is complete.
\end{proof}

\section{The essential numerical range of the double-layer operator on polyhedra (proof of Theorem \ref{thm:Q2Poly})}\label{sec:4}

We now consider Lipschitz polyhedral $\Gamma$, proving Theorem \ref{thm:Q2Poly}, showing that the essential numerical range of the open-book Lipschitz polyhedron $\Omega_{\theta,n}$ contains an arbitrarily large disc centred on zero in the complex plane if $n\geq 2$, the number of pages, is large enough, and the opening angle $\theta\in (0,\pi]$ is small enough. In \S\ref{sec:3-dconstruction} we define $\Omega_{\theta,n}$. In \S\ref{sec:Thm13Proof} we give the proof of Theorem \ref{thm:Q2Poly}.

\subsection{Definition of the family of ``open-book" polyhedra}\label{sec:3-dconstruction}

In this section we define $\Omega_{\theta,n}$, the open-book polyhedron with $n\geq 2$ pages and opening angle $\theta\in (0,\pi]$ (see Definition \ref{def:Gamma3-d} below and Figures \ref{fig:book} and \ref{fig:book2}). This star-shaped, Lipschitz polyhedron lies between the planes $x_3=0$ and $x_3=-1$. Indeed its $5n+1$ vertices comprise $2n+1$ vertices (the ``top'' vertices) that lie in the plane $x_3=0$ and $3n$ vertices (the ``bottom'' vertices) that lie in the plane $x_3=-1$. We first specify, in Definitions \ref{def:top} and \ref{def:bottom}, these top and bottom vertices, and we specify underneath these definitions (and see Figures \ref{fig:toppoints} and \ref{fig:bottompoints}) the polygons $\Gamma_0$ and $\Gamma_{-1}$ that form the top and bottom faces of $\Omega_{\theta,n}$. We then show, in Lemma \ref{lem:planes}, that certain groups of four vertices, each group comprising two top vertices and two bottom vertices, lie in planes (under certain constraints on the parameters $r_1$ and $r_2$ in Definition \ref{def:bottom}). The boundary $\Gamma_{\theta,n}$ of the open-book polyhedron $\Omega_{\theta,n}$ has $3n+2$ faces; $3n$ of these are the convex hulls of these groups of four vertices, each of these a convex quadrilateral; the remaining two are the top and bottom faces $\Gamma_0$ and $\Gamma_{-1}$.

\begin{definition}[The vertices $\newbx^m$, $m=1,\ldots,3n$ (the ``top" vertices)] \label{def:top}
Given $\newangle\in (0,\pi]$ and $n\in \mathbb{N}$ with $n\geq 2$, let
\begin{equation} \label{eq:thetan}
\theta_n := \frac{\theta}{2n-1},
\end{equation}
\begin{align*}
&\newbx^{3j-2}:=\bze, \qquad
\newbx^{3j-1}:=\left(\cos\left((2j-2)\newangle_n\right), \sin\left((2j-2)\newangle_n\right),0\right), \\
&\text{ and }\quad  \newbx^{3j}:=
\left(\cos\left((2j-1)\newangle_n\right), \sin\left((2j-1)\newangle_n\right),0\right), \quad\tfor j=1,\ldots,n.
\end{align*}
\end{definition}

Let $\Gamma_0$ be the polygon formed by the union of $n$ (congruent) triangles, where the $j$th triangle has vertices $\by^{3j-2}$, $\by^{3j-1}$, and $\by^{3j}$, for $j=1,...,n$; $\Gamma_0$ forms the ``top" face of our polyhedron (see Figure \ref{fig:toppoints}). Observe that these triangles meet at the common vertex $\bzero$ where each triangle has angle $\newangle_n$.

\begin{figure}
\begin{center}
\begin{tikzpicture}[scale=6]
\fill[lightgray] plot coordinates {(0,0) (1,0) (0.97493,0.22252) (0,0)};
\fill[lightgray] plot coordinates {(0,0) (0.90097,0.43388) (0.78183,0.62349) (0,0)};'
\fill[lightgray] plot coordinates {(0,0) (0.62349,0.78183) (0.43388,0.90097) (0,0)};'
\fill[lightgray] plot coordinates {(0,0) (0.22252,0.97493) (6.1232e-17,1) (0,0)};'
\draw[thick] plot coordinates {(0,0) (1,0) (0.97493,0.22252) (0,0)};
\draw[thick] plot coordinates {(0,0) (0.90097,0.43388) (0.78183,0.62349) (0,0)};'
\draw[thick] plot coordinates {(0,0) (0.62349,0.78183) (0.43388,0.90097) (0,0)};'
\draw[thick] plot coordinates {(0,0) (0.22252,0.97493) (6.1232e-17,1) (0,0)};'
\draw[red] (1,0) arc(0:90:1);
\draw (0,0) node[anchor=north] {$\by^{3j-2}=\bze$};
\draw (0,-0.08) node[anchor=north] {$j=1,...,n$};
\draw (1,0) node[anchor=west] {$\by^2=(1,0,0)$};
\draw (0,1) node[anchor=east] {$\by^{3n}=(\cos(\theta),\sin(\theta),0)$};
\draw (0.2225,0.9749) node[anchor=south west] {$\by^{3n-1}=(\cos((2n-2)\theta_n),\sin((2n-2)\theta_n),0)$};
\draw (0.9749,0.2225) node[anchor=west] {$\by^{3}=(\cos(\theta_n),\sin(\theta_n),0)$};
\draw (0.90097,0.43388) node[anchor=west] {$\by^{5}=(\cos(2\theta_n),\sin(2\theta_n),0)$};
\draw (0.57,0) arc(0:25.71:0.57);
\draw (0.497,0.056) node {$\theta_n$};
\draw (0.472,0.165) node {$\theta_n$};

\end{tikzpicture}
\end{center}
\caption{\label{fig:toppoints} The polygon $\Gamma_0$, consisting of $n$ congruent triangles, that forms the ``top'' face of the open-book polyhedron $\Omega_{\theta,n}$, i.e.\ the part of its boundary, $\Gamma_{\theta,n}$, that lies in the plane $x_3=0$. Shown is the case $n=4$ and $\theta=\pi/2$. The vertices of the $n$ triangles that are not $\bze$ lie on the circle of radius one centred at $\bze$ (shown in red).}
\end{figure}
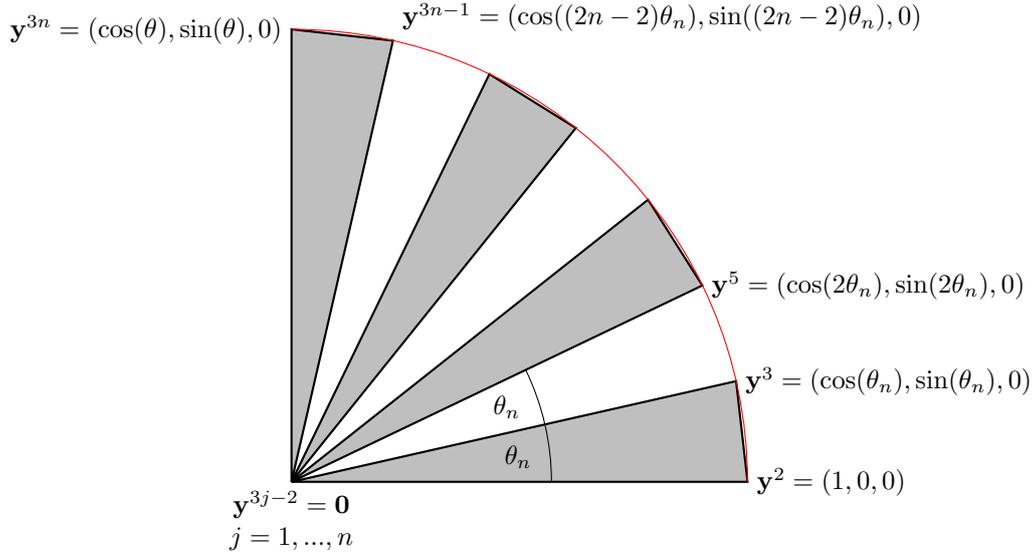

\begin{figure}
\begin{center}
\begin{tikzpicture}[scale=6]

\fill[lightgray] plot coordinates {(0,0) (1,0) (0.97493,0.22252) (0,0)};
\fill[lightgray] plot coordinates {(0,0) (0.90097,0.43388) (0.78183,0.62349) (0,0)};'
\fill[lightgray] plot coordinates {(0,0) (0.62349,0.78183) (0.43388,0.90097) (0,0)};'
\fill[lightgray] plot coordinates {(0,0) (0.22252,0.97493) (6.1232e-17,1) (0,0)};'
\fill[pink] plot coordinates { (0,0) (0.99645,0) (0.9683,0.24978) (0.23644,0.082732) (0.91278,0.40844) (0.76404,0.64517) (0.17712,0.17712) (0.64517,0.76404) (0.40844,0.91278) (0.082732,0.23644) (0.24978,0.9683) (6.1015e-17,0.99645) (0,0)};
\fill[lightgray] plot coordinates {(0,0) (1,0) (0.97493,0.22252) (0,0)};
\fill[lightgray] plot coordinates {(0,0) (0.90097,0.43388) (0.78183,0.62349) (0,0)};'
\fill[lightgray] plot coordinates {(0,0) (0.62349,0.78183) (0.43388,0.90097) (0,0)};'
\fill[lightgray] plot coordinates {(0,0) (0.22252,0.97493) (6.1232e-17,1) (0,0)};'
\draw[thick] plot coordinates { (0,0) (0.99645,0) (0.9683,0.24978) (0.23644,0.082732) (0.91278,0.40844) (0.76404,0.64517) (0.17712,0.17712) (0.64517,0.76404) (0.40844,0.91278) (0.082732,0.23644) (0.24978,0.9683) (6.1015e-17,0.99645) (0,0)};
\draw (0,0) node[anchor=north] {$\bz^{1}=(0,0,-1)$};
\draw (1,0) node[anchor=west] {$\bz^2=(r_1,0,-1)$};
\draw (0,1) node[anchor=south east] {$\bz^{3n}=(r_1\cos(\theta),r_1\sin(\theta),-1)$};
\draw (0.2225,0.9749) node[anchor=south west] {$\bz^{3n-1}$};
\draw (0.9683,0.24978) node[anchor=west] {$\bz^{3}$};
\draw (0.23644,0.082732) node[anchor=east] {$\bz^{4}$};
\draw (0.91278,0.40844) node[anchor=west] {$\bz^{5}$};

\end{tikzpicture}
\end{center}
\caption{\label{fig:bottompoints} The $3n$-sided polygon $\Gamma_{-1}$, that forms the ``bottom'' face of the open-book polyhedron $\Omega_{\theta,n}$, i.e.\ the part of its boundary, $\Gamma_{\theta,n}$, that lies in the plane $x_3=-1$. Shown is the case $n=4$ and $\theta=\pi/2$. Also shown is the polygon $\Gamma_0$ (in light gray). The vertices $\bz^1,\ldots,\bz^{3n}$ of $\Gamma_{-1}$ are given as in Definition \ref{def:bottom}, with the parameters $r_1$, $r_2$, and $\eta$ given by \eqref{eq:r_1}, \eqref{eq:r_2}, and \eqref{eq:etadef}, respectively. Note that $\Omega_{\theta,n}$ is shown in 3-d with the same parameter values in Figure \ref{fig:book}.}
\end{figure}
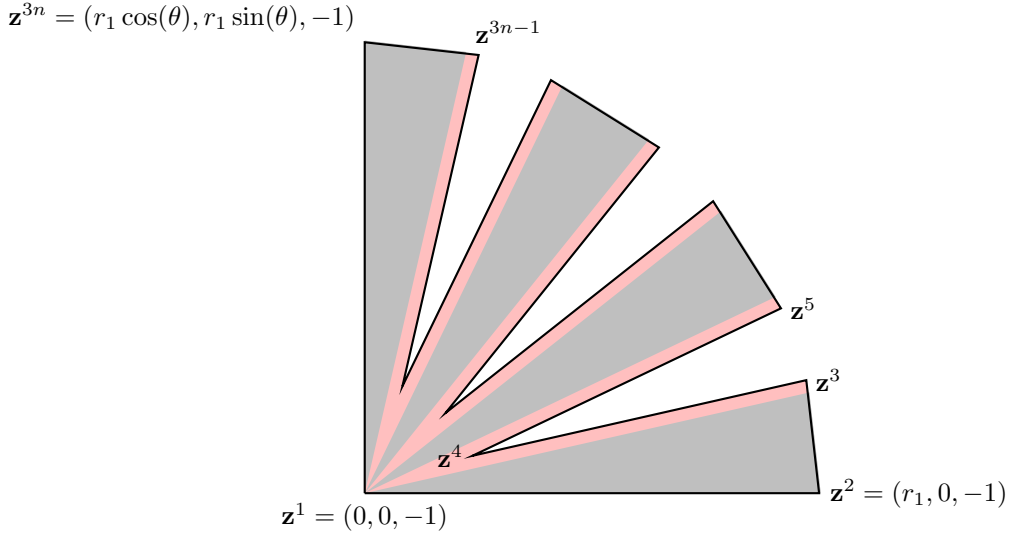

We now define $\newbz^m$, $m=1,\ldots,3n$, depending on three parameters $r_1, r_2,$ and $\eta$. The idea is that each $\newbz^m$ is defined as a perturbation of $\newbx^m - (0,0,1)$, so that, for certain parameter values, the $\newbx^m$ and $\newbz^m$ are the vertices of a polyhedron, where the perturbations are such that the polyhedron is star-shaped and Lipschitz.

\begin{definition}[The vertices $\newbz^m$, $m=1,\ldots,3n$ (the ``bottom" vertices)] \label{def:bottom}
Given $\newangle\in (0,\pi]$, $n\in \mathbb{N}$ with $n\geq 2$, $r_j >0$, $j=1,2,$ and $\eta$ in the range
\beq\label{eq:eta}
0< \eta <\frac{\theta_n}{2},
\eeq
where $\theta_n$ is given by \eqref{eq:thetan}, let
\begin{align*}
&\newbz^1 := (0,0,-1), \quad \newbz^2 := (r_1, 0, -1), \\
&\newbz^{3j}:=\left(\cos\left((2j-1)\theta_n+\eta\right), \sin\left((2j-1)\theta_n+\eta\right),-1\right), \quad j=1,\ldots n-1, \\
&\newbz^{3j-2}:=\left(r_2\cos\left((2j-5/2)\theta_n\right), r_2\sin\left((2j-5/2)\theta_n\right),-1\right), \quad j=2,\ldots n, \\
&\newbz^{3j-1}:=\left(\cos\left((2j-2)\theta_n-\eta\right), \sin\left((2j-2)\theta_n-\eta\right),-1\right), \quad j=2,\ldots n, \quad\tand\\
&\newbz^{3n}:= (r_1 \cos\newangle, r_1 \sin\newangle, -1).
\end{align*}
\end{definition}

Observe that the points $\newbz^{3j}, j=1,\ldots n-1$, are rotations of $\newbx^{3j}- (0,0,1)$ by the angle $\eta$ in the  $x_3=-1$ plane, and similarly for $\newbz^{3j-1}$ and $\newbx^{3j}-(0,0,1)$. If $r_2$ is small, then $\newbz^{3j-2}, j=2, \ldots, n$, can be considered as perturbations of $\newbx^{3j-2}- (0,0,-1)= (0,0,-1)$ (see Figures \ref{fig:toppoints} and \ref{fig:bottompoints}).
Let $\Gamma_{-1}$ be the polygon in the plane $x_3=-1$ formed by connecting the vertices $\bz^1$, $\bz^2$,...,$\bz^{3n}$, $\bz^1$, in that order; $\Gamma_{-1}$ forms the ``bottom" face of our polyhedron (see Figure \ref{fig:bottompoints}).

We join pairs of vertices of $\Gamma_0$ with pairs of vertices in $\Gamma_{-1}$ to form the open-book polyhedron $\Omega_{\theta,n}$ (so-called because it resembles -- see Figures \ref{fig:book} and \ref{fig:book2} -- an open book with $n$ pages). Before proceeding with the definition, we check that the pairs of vertices that we propose to join to create quadrilateral faces of $\Omega_{\theta,n}$ do indeed lie in the same planes.

\ble[Particular groups of four vertices lie in planes]\label{lem:planes}

\

(i) (``Front of Page $1$ and back of Page $n$.") For all $r_1>0$, the points $\newbx^1, \newbx^2, \newbz^1,$ and $\newbz^2$ lie in the $x_2=0$ plane, and the points $\newbx^{3n}, \newbx^1, \newbz^{3n},$ and $\newbz^1$ lie in the
$x_2=0$ plane rotated by angle $\newangle$ clockwise about the $x_3$ axis.

(ii) (``The ends of Pages $1$ and $n$.") Given $\eta$ satisfying \eqref{eq:eta}, if
\beq\label{eq:r_1}
r_1 := \frac{\cos(\theta_n/2+\eta)}{\cos(\theta_n/2)},
\eeq
then $\newbx^2, \newbx^3, \newbz^2,$ and $\newbz^3$ lie in a plane, and $\newbx^{3n-1}, \newbx^{3n}, \newbz^{3n-1},$ and $\newbz^{3n}$ also lie in a plane.

(iii) (``The backs of Pages $1, \ldots, n-1$ and the fronts of Pages $2, \ldots, n$.") Given $\eta$ satisfying \eqref{eq:eta}, if
\beq\label{eq:r_2}
r_2 := \frac{
\sin(\eta)
}{
\sin\left(\theta_n/2\right)
},
\eeq
then $\newbx^{3j}, \newbx^{3j+1}, \newbz^{3j},$ and $\newbz^{3j+1}$, $j=1,\ldots, n-1$, lie in a plane, and
$\newbx^{3j-2}, \newbx^{3j-1}, \newbz^{3j-2},$ and $\newbz^{3j-1}$, $j=2,\ldots, n$, also lie in a plane.

(iv) (``The ends of Pages $2, \ldots, n-1$.") For all $\eta$ satisfying \eqref{eq:eta}, the points $\newbx^{3j-1}, \newbx^{3j}, \newbz^{3j-1},$ and $\newbz^{3j}$, $j=2,\ldots, n-1$, lie in a plane.
\ele

The proof of Lemma \ref{lem:planes} is straightforward, and so is omitted.

\bre \label{rem:r_asymp} The constraints on $\theta$ and $n$ in Definitions \ref{def:top} and \ref{def:bottom}, together with \eqref{eq:eta}, ensure that $r_j\in (0,1)$, for $j=1,2$, where $r_1$ and $r_2$ are defined by \eqref{eq:r_1} and \eqref{eq:r_2}, respectively, and imply that
\beqs
r_1\rightarrow 1 \quad\tand \quad r_2  \sim \frac{2\eta}{\theta_n}=\frac{2(2n-1)\eta}{\theta} \quad \mbox{as } \newangle\tendo.
\eeqs
\ere

\

We now define the polygonal faces (other than $\Gamma_0$ and $\Gamma_{-1}$) of our polyhedron.

\begin{definition}[Faces $\Gamma_j$, $j=1,\ldots, 3n$]

\
(i) (``Front of pages.") Let $\Gamma_{3j-2}, j=1,\ldots, n$, be the convex hull of $\newbx^{3j-2}, \newbx^{3j-1}, \newbz^{3j-2},$ and $\newbz^{3j-1}$.

(ii) (``Ends of pages.") Let $\Gamma_{3j-1}, j=1,\ldots, n$, be the convex hull of $\newbx^{3j-1}, \newbx^{3j}, \newbz^{3j-1},$ and $\newbz^{3j}$.

(iii) (``Back of pages.") Let $\Gamma_{3j}, j=1,\ldots, n$, be the convex hull of $\newbx^{3j}, \newbx^{3j+1}, \newbz^{3j},$ and $\newbz^{3j+1}$, for $j<n$, the convex hull of $\newbx^{3n}, \newbx^{1}, \newbz^{3n},$ and $\newbz^{1}$ for $j=n$.
\end{definition}

\begin{corollary}
Provided the parameters $r_1$, $r_2$ and $\eta$ satisfy \eqref{eq:r_1}, \eqref{eq:r_2}, and \eqref{eq:eta}, $\Gamma_j$, $j=1,\ldots,3n$, are polygons.
\end{corollary}

\bpf
For $\Gamma_{3j-2}, j=1,\ldots, n$, this follows from Part (i) (for $j=1$) and Part (iii) (for $j=2,\ldots, n$) of Lemma \ref{lem:planes}.
For $\Gamma_{3j-1}, j=1,\ldots, n$, this follows from Part (ii) (for $j=1$, $n$) and Part (iv) (for $j=2,\ldots, n-1$) of Lemma \ref{lem:planes}.
For $\Gamma_{3j}, j=1,\ldots, n$, this follows from Part (i) (for $j=n$) and Part (iii) (for $j=1,\ldots, n-1$) of Lemma \ref{lem:planes}.
\epf

\begin{definition}[Open-book polyhedron $\Omega_{\theta,n}$]\label{def:Gamma3-d}
With $\Gamma_j$, $j=-1,0, 1,\dots, 3n$, as above, $r_1$ given by \eqref{eq:r_1}, $r_2$ given by \eqref{eq:r_2}, and $\eta$ given by
\begin{equation} \label{eq:etadef}
\eta := \frac{\theta\theta_n}{4\pi} = \frac{\theta^2}{4\pi(2n-1)},
\end{equation}
so that $\eta$ satisfies \eqref{eq:eta}, let $\Gamma_{\theta,n}$ denote the closed surface
\beq\label{eq:Gamma3-d}
\Gamma_{\newangle, n}:= \bigcup_{j=-1}^{3n}\Gamma_j,
\eeq
and let $\Omega_{\theta,n}$ (the {\em open-book polyhedron with $n$ pages and opening angle $\theta$}) denote the interior of $\Gamma_{\theta,n}$ (see Figures \ref{fig:book} and \ref{fig:book2}).
\end{definition}

\begin{remark}[Closing the book] \label{rem:closing} Key to the proof below of Theorem \ref{thm:Q2Poly} is the limit $\theta\to 0$ (``closing the book''). In this limit (see \eqref{eq:etadef} and Remark \ref{rem:r_asymp}) $r_1\to 1$ and $r_2\to 0$ and the ``fronts'' and ``backs'' of the pages of the book (i.e.\ the faces $\Gamma_{3j-2}$ and $\Gamma_{3j}$, $j=1,\ldots,n$) collapse onto the unit square $[0,1]\times\{0\}\times [0,1]$. Precisely, for $j=1,...,n$, $\by^{3j-2}=\bze$, $\by^{3j-1}\to (1,0,0)$, $\bz^{3j-2}\to (0,0,-1)$, $\bz^{3j-1}\to (1,0,-1)$, $\by^{3j}\to (1,0,0)$, and $\bz^{3j}\to (1,0,-1)$.
\end{remark}

\begin{figure}
\begin{center}
\includegraphics[width=.35\textwidth]{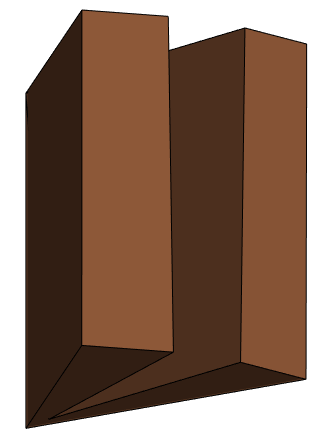}\includegraphics[width=.4\textwidth]{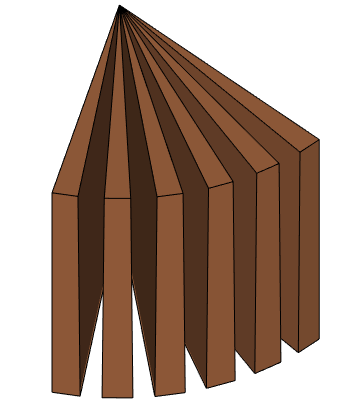}
\caption{\label{fig:book2} The open-book polyhedron $\Omega_{\theta,n}$ with: $n=2$ pages, opening angle $\newangle = \pi/4$ (left); $n=6$ pages, opening angle $\theta=\pi/3$ (right).}
\end{center}
\end{figure}

The open-book polyhedron is star-shaped with respect to a ball, in the following standard sense.

\begin{definition}[Star-shaped and star-shaped with respect to a ball] \label{def:ss}
A bounded domain $\Omega\subset \R^d$ is {\em star-shaped with respect to $\bx_0\in \Omega$} if the line $[\bx,\bx_0]\subset \Omega$, for every $\bx\in \Omega$, where $[\bx,\bx_0]:= \mathrm{conv}\left(\{\bx,\bx_0\}\right)$. $\Omega$ is {\em star-shaped} if it is star-shaped with respect to some $\bx_0\in \Omega$.  $\Omega$ is {\em star-shaped with respect to the ball $B_\epsilon(\by)$}, for some $\epsilon>0$ and $\by\in \Omega$, if $[\bx,\bx_0]\subset \Omega$ for every $\bx\in \Omega$ and $\bx_0\in B_\epsilon(\by)$.  $\Omega$ is {\em star-shaped with respect to a ball} if it is star-shaped with respect to $B_\epsilon(\by)$, for some $\epsilon>0$ and $\by\in \Omega$.
\end{definition}

\ble\label{lem:polyhedron}
For $\theta\in (0,\pi]$ and $n\in \NN$ with $n\geq 2$, $\Omega_{\theta,n}$ is a bounded Lipschitz polyhedron  that is star-shaped with respect to a ball, and $\Gamma_{\theta,n}$ is its boundary.
\ele

\bpf
By construction and Lemma \ref{lem:planes}, $\Omega_{\theta,n}$ is a bounded polyhedron with boundary $\Gamma_{\theta,n}$. To see that $\Omega_{\theta,n}$ is star-shaped with a respect to a ball, which implies that $\Omega_{\theta,n}$ is Lipschitz, it is enough, by \cite[Lemma 3.1]{HiMo:11}, to show that $(\bx-\bx_0)\cdot \bn(\bx)\geq \gamma$, for some $\bx_0\in \Omega$ and $\gamma>0$ and all $\bx\in \Gamma_{\theta,n}$ for which the outward-pointing unit normal $\bn(\bx)$ is defined. If $\bx_0=(0,0,-1)$, it is easy to see that this condition holds for all $\bx\in \Gamma_j$, with $j=0$ or $j=2,\ldots,3n-1$; indeed, by the symmetry of $\Gamma_{\theta,n}$, it is enough to check this holds for $j=0$ and $j=2,\ldots,6$. But this implies that this condition also holds, on the same subset of $\Gamma_{\theta,n}$, if we take $\bx_0=(\epsilon,\epsilon,\epsilon-1)\in \Omega_{\theta,n}$ for any sufficiently small $\epsilon$. Moreover, with this choice of $\bx_0$ it holds also, for some $\gamma>0$, that $(\bx-\bx_0)\cdot \bn(\bx)\geq \gamma$ for $\bx\in \Gamma_j$ with $j\in \{-1,1,3n\}$.
\epf

\subsection{Proof of Theorem \ref{thm:Q2Poly}} \label{sec:Thm13Proof}

We first state and prove an analogue of Lemma \ref{lem:ANDLP} with $\Gamma_{\newangle,n}$ instead of $\Gamma_\beta$.
For this purpose it is convenient to introduce a notation for the $2n$ faces of $\Gamma_{\newangle,n}$ that are ``fronts" and ``backs" of pages (and not ``ends", or the top, $\Gamma_0,$ or the bottom, $\Gamma_{-1}$).

\begin{definition}[Relabelling of ``front" and ``back" pages of $\Gamma_{\newangle,n}$]\label{def:widetildeGamma}
Let
\beqs
\widetilde{\Gamma}_m := \Gamma_{3j-2} \quad \tfor m=2j-1,\,\, j=1,\ldots,n,
\eeqs
and
\beqs
\widetilde{\Gamma}_m := \Gamma_{3j} \quad \tfor m=2j,\,\, j=1,\ldots,n.
\eeqs
\end{definition}

\ble[The relationship of $B_N$ to the double-layer potential on $\Gamma_{\theta,n}$]\label{lem:ANDLP3-d}
Let $\Gamma_{\theta,n}$ be as in Definition \ref{def:Gamma3-d}, with $\widetilde{\Gamma}_{m}, m=1,\ldots, 2n,$ as in Definition \ref{def:widetildeGamma}.
Setting $N=2n-1$ (which makes $N$ odd),
define the orthonormal set $\{\psi_1,...,\psi_N\}\subset L^2(\Gamma_{\theta,n})$ by
$$
\psi_m(\bx):=
\begin{cases}
|\widetilde\Gamma_m|^{-1/2}, & \bx \in \widetilde\Gamma_m,\\
0, & \mbox{otherwise},
\end{cases}
$$
for $m=1,...,N$.
Define the Galerkin matrix $D_N$ by \eqref{eq:DNdef}, where $D$ is the double-layer potential operator on $\Gamma_{\theta,n}$, and define $B_N\in \R^{N\times N}$ by \eqref{eq:aBdef}.
Then
\beq\label{eq:DNform3d}
\big( D_N\big)_{jm}= \big( B_N\big)_{jm}d_{jm}, \quad 1\leq j,m\leq N,
\eeq
where $d_{jm}:= 0$ for $j=m$,
$$
d_{jm} := \frac{1}{4\pi |\widetilde \Gamma_j|^{1/2}|\widetilde\Gamma_m|^{1/2}}\int_{\widetilde \Gamma_j} \alpha_m \, \rd s, \quad \mbox{for } j\neq m,
$$
 and
$$
 \alpha_m(\bx) := 4\pi \left|\int_{\widetilde \Gamma_m} \frac{\partial \Phi(\bx,\by)}{\partial n(\by)} \rd s(\by)\right|= \int_{\widetilde \Gamma_m} \frac{|(\bx-\by)\cdot \bn(\by)|}{|\bx-\by|^3}\rd s(\by), \quad \bx \in \R^3\setminus \widetilde \Gamma_m,
$$
is the solid angle subtended at $\bx$ by $\widetilde \Gamma_m$.
Further, for $1\leq j,m \leq N$, with $j\neq m$, $d_{jm}\to 1/2$ as $\theta\to 0$, so that
\beq\label{eq:limitDA3d}
D_N \rightarrow \frac{1}{2}\, B_N.
\eeq
\ele
\bpf
For $m=1,\ldots,N$ and $\bx\in \R^3\setminus \widetilde \Gamma_m$,
\beq\label{eq:DLPsolid3d2}
 \int_{\widetilde \Gamma_m} \frac{\partial \Phi(\bx,\by)}{\partial n(\by)} \rd s(\by) = \pm \frac{\alpha_m(\bx)}{4\pi},
\eeq
where the $+$ sign is taken if $\bx$ is on the side of $\widetilde \Gamma_m$ to which the normal points, otherwise the $-$ sign is taken.
Thus, for $1\leq j,m\leq N$,
\begin{align}
D\psi_m(\bx) = \frac{1}{4\pi|\widetilde \Gamma_m|^{1/2}} (-1)^{m+1}\mathrm{sign}(m-j)\alpha_m(\bx),  \quad \bx\in \widetilde \Gamma_j,\label{Dphi3d}
\end{align}
and \eqref{eq:DNform3d} follows.  Let $\Gamma^*:= \cup_{j=1}^N \widetilde \Gamma_j$. As $\theta\to 0$, the vertices of each quadrilateral $\widetilde \Gamma_j$ tend to the corresponding vertices of the unit square $[0,1]\times\{0\}\times[0,1]$ (see Remark \ref{rem:closing}), so that $|\widetilde \Gamma_j|\to 1$ and $\alpha_m(\bx)\to 2\pi$ for almost all $\bx\in \Gamma^*\setminus \widetilde \Gamma_m$. Thus, as $\theta\to 0$, for $j\neq m$,
$$
|\widetilde \Gamma_j|\sim |\widetilde \Gamma_m| \to 1 \quad \mbox{and} \quad \int_{\widetilde \Gamma_j} \alpha_m \, \rd s\to 2\pi
$$
by the dominated convergence theorem, so that $d_{jm}\to 1/2$,
implying \eqref{eq:limitDA3d}.
\epf

\

\begin{proof}[Proof of Theorem \ref{thm:Q2Poly}] Suppose $n\in \NN$ with $n\geq 2$, and that $\Gamma=\Gamma_{\theta,n}$, with $\theta\in (0,\pi]$, and $D$ is the double-layer potential operator on $\Gamma$. By Theorem \ref{thm:loccon},
$$
W_\ess(D) = \bigcup_{\bx\in V} \overline{W(D_\bx)},
$$
where $V$ is the set of vertices of $\Gamma$ and $\Gamma_\bx$ is the cone at $\bx$ given by \eqref{eqn:cone}. In particular, since $\bx^*:=\bze$ is a vertex,
$$
W_\ess(D)\supset W(D_{\bx^*}).
$$
But, recalling Definition \ref{def:widetildeGamma} and setting $N=2n-1$, we have
$$
\Gamma_{\bx^*}\supset \Gamma^* := \bigcup_{j=1}^{N}\widetilde \Gamma_m,
$$
so that, by \eqref{eq:ANIncl}, $W(D_N)\subset W(D_{\bx^*})$, where $D_N$ is defined by \eqref{eq:DNform3d}. But, by Lemma \ref{lem:ANDLP3-d}, $2D_N\to B_N$ as $\theta\to 0$, so, by Corollary \ref{cor:CN}, for every $\epsilon>0$ there exists $\theta_0>0$ such that
$$
\left[\epsilon-\sqrt{(N+1)/2},\sqrt{(N+1)/2}-\epsilon\right]\,\cup \,\left\{z\in \C:|z|\leq \sqrt{\frac{N-1}{2}}-\epsilon\right\} \subset W(2D_N)= 2W(D_N),
$$
for $0<\theta\leq \theta_0$, and the result follows since $W_\ess(D)$ is convex.
\end{proof}

\section{The essential norm of the double-layer operator on polyhedra in weighted spaces of continuous functions (proof of Theorem \ref{thm:Q3})}

 If the bounded Lipschitz domain $\Omega_-\subset \R^d$ is a polygon or polyhedron, then $D$ is also a bounded linear operator on $C(\Gamma)$ (equipped with the usual supremum norm). Indeed this is true (e.g., \cite[\S4]{We:09}, \cite[Chapter 4, \S2]{Ma:91}) if and only if\footnote{We note that, in the 2-d case, $\Omega_-$ in Figure \ref{fig:GammaEpsOmega} is an example of a bounded Lipschitz domain for which \eqref{eq:necsuff} does not hold.  For, using the notation of Definition \ref{def:Gamma2d} and \eqref{eq:DLPsolidangle}, for every $m\in \NN$ the side $\Gamma_m\subset \Gamma_\beta\subset \Gamma$ subtends the same angle $\alpha_m(\bze)$ at $\bx=\bze$, so that \eqref{eq:necsuff} blows up as $\bx\to \bze$.}
\begin{equation} \label{eq:necsuff}
\sup_{\bx\in \R^d\setminus\Gamma} \int_{\Gamma} \left|\frac{\partial \Phi(\bx,\by)}{\partial n(\by)}\right| \rd s(\by) < \infty,
\end{equation}
in which case the norm of $D$ as an operator on $C(\Gamma)$ is
$$
\|D\|_{C(\Gamma)} = \sup_{\bx\in \Gamma} \int_{\Gamma} \left|\frac{\partial \Phi(\bx,\by)}{\partial n(\by)}\right| \rd s(\by).
$$
When $\phi\in C(\Gamma)$ and \eqref{eq:necsuff} holds, the integrals \eqref{eq:DD'}/\eqref{eq:DD'2} are well-defined for all $\bx\in \Gamma$, and the function $D\phi$ so-defined is equal almost everywhere to a continuous function; indeed $D\phi$ is continuous if the definition of $D\phi$ is adjusted on a set of (surface) measure zero to read
\beq\label{eq:DCG}
D \phi(\bx) := \left(\Theta(\bx)-\half\right)\phi(\bx) + \int_\Gamma \pdiff{\Phi(\bx,\by)}{n(\by)} \phi(\by)\, \rd s(\by),
\quad \bx\in \Gamma,
\eeq
where
$$
\Theta(\bx) := \lim_{\delta\to 0} \frac{|B_\delta(\bx)\cap \Omega_-|}{|B_\delta(\bx)|}, \quad \bx\in \Gamma,
$$
is the {\em density of $\Omega_-$ at $\bx$} (e.g., \cite[Lemma 2.9]{Kral:80}, \cite{KrMe:00}), with $\Theta(\bx)=1/2$ almost everywhere on $\Gamma$ (everywhere that the normal $\bn$ is well-defined).
An explicit expression, similar in flavour to Theorem \ref{lem:local}, is also known for the essential norm of $D$ on $C(\Gamma)$, that
 (\cite[\S4]{Kral:80},  \cite[Chapter 4, Theorem 10]{Ma:91}, or \cite[Theorem 4.1]{We:09})
\begin{equation} \label{eq:essCG}
\|D\|_{C(\Gamma),\mathrm{ess}} = \lim_{\delta\to 0}\sup_{\bx\in \Gamma} \int_{\Gamma\cap B_\delta(\bx)} \left|\frac{\partial \Phi(\bx,\by)}{\partial n(\by)}\right| \rd s(\by).
\end{equation}

Suppose now that $w\in L^\infty(\Gamma)$ satisfies \eqref{eq:wbound} for some $c_->0$ and define by \eqref{eq:wnormDef}  the weighted norm $\|\cdot\|_{C_w(\Gamma)}$ on $C(\Gamma)$, this norm equivalent to the supremum norm. Generalisations of \eqref{eq:essCG} to the case when $C(\Gamma)$ is equipped with a weighted norm $\|\cdot\|_{C_w(\Gamma)}$ have been discussed by Kr\'al and Wendland \cite{KrWe88}, Kr\'al and Medkov\'a \cite{KrMe:00}, and Wendland \cite{We:09}, who state the formula \eqref{eq:ENweighted} below for cases where the weight $w$ is lower semi-continuous (in particular see \cite[Theorem 18]{KrMe:00}), when the $\esssup$ in \eqref{eq:ENweighted} can be replaced by a supremum. So that our results apply to the class of weights considered by Hansen \cite{Ha:01} (see the discussion in \S\ref{sec:open}), we prove that the formula \eqref{eq:ENweighted} holds when $w\in L^\infty(\Gamma)$, a more general class of weights than in \cite{KrWe88,KrMe:00} (see, e.g.,  \cite[Theorem 5]{Zink:65}). Our proof starts from that of \eqref{eq:essCG} in \cite[\S4]{Kral:80} for the unweighted case, this also the starting point for the proof when $w$ is lower semicontinuous in \cite{KrMe:00}. But the proof for the case $w\in L^\infty(\Gamma)$ has new difficulties. We need to design the proof so as to avoid computing the action of a general Borel measure on our $L^\infty(\Gamma)$ weight (this action can be sensibly defined when $w$ is lower semicontinous). Furthermore, the proofs in \cite{Kral:80,KrMe:00} implicitly (and trivially) reverse the order of suprema over $\phi\in C(\Gamma)$ and $x\in \Gamma$; to justify this when the supremum over $x\in \Gamma$ is replaced by an essential supremum we prove first the following lemma\footnote{If $f:S\times T\to \R$ then, trivially, $\sup_{s\in S}\sup_{t\in T} f(s,t) = \sup_{(s,t)\in S\times T} f(s,t)=\sup_{t\in T}\sup_{s\in S} f(s,t)$. But this need not hold if a $\sup$ is replaced by an $\esssup$. E.g., if $S=T=\R$ (equipped with Lebesgue measure) and $f(s,t):= 1$ if $s=t$, $:= 0$ otherwise, then $\sup_{s\in S} \esssup_{t\in T} f(s,t) = 0$ but $\esssup_{t\in T} \sup_{s\in S} f(s,t) = 1$.}.

\begin{lemma} \label{lem:swap} Suppose that $z\in L^\infty(\Gamma)$ and, for some index set $S$, $\{f_\phi:\phi\in S\}$ is a bounded subset of $C(\Gamma)$. Then
$$
A:=\sup_{\phi\in S} \,\esssup_{\bx\in \Gamma} \left(|z(\bx)f_\phi(\bx)|\right) =  \esssup_{\bx\in \Gamma} \left(|z(\bx)|\sup_{\phi\in S}|f_\phi(\bx)|\right)=:B.
$$
\end{lemma}
\begin{proof} Given  $\epsilon>0$ we can choose $\psi\in S$ so that  $\esssup_{\bx\in \Gamma} |(z(\bx)f_\psi(\bx)|)\geq A-\epsilon$. But $|f_\psi(\bx)| \leq \sup_{\phi\in S}|f_\phi(\bx)|$, for every $\bx\in \Gamma$, so it follows that $A-\epsilon \leq B$. Since this holds for all $\epsilon>0$, $B\geq A$.

Conversely, given $\epsilon>0$ there exists $G\subset \Gamma$ with surface measure $|G| >0$ such that
$|z(\bx)|$ $\sup_{\phi\in S}|f_\phi(\bx)|\geq B-\epsilon$, for each $\bx\in G$. Thus, for each $\bx\in G$ there exists $\psi_\bx\in S$ such that $|z(\bx)f_{\psi_\bx}(\bx)|\geq B-2\epsilon$, for $\bx\in G$. By Lusin's theorem (e.g., \cite[\S2.24]{RudRCA}) there exists $\widetilde z\in C(\Gamma)$ and $\widetilde G\subset \Gamma$ such that $z(\bx)=\widetilde z(\bx)$ for all $\bx\in \Gamma\setminus \widetilde G$ and $|\widetilde G| < |G|$. Thus $|\widehat G|>0$, where $\widehat G:= G\setminus \widetilde G$, and $|\widetilde z(\bx)f_{\psi_\bx}(\bx)|\geq B-2\epsilon$, for $\bx\in \widehat G$. Since each $\widetilde zf_{\psi_\bx}$ is continuous, for every $\bx\in \widehat G$ there exists $\varepsilon(\bx)>0$ such that $|\widetilde z(\by)f_{\psi_\bx}(\by)|\geq B-3\epsilon$, $\by\in \Gamma\cap B_{\varepsilon(\bx)}(\bx)$. Let
$$
O := \bigcup_{\bx\in \widehat G} \Gamma\cap B_{\varepsilon(\bx)}(\bx).
$$
Then  $\{ \Gamma\cap B_{\varepsilon(\bx)}(\bx):\bx\in \widehat G\}$ is an open cover for $O\supset \widehat G$ (in the Euclidean topology on $\R^d$ restricted to $\Gamma$) which has a countable subcover $\{ \Gamma\cap B_{\varepsilon(\bx_n)}(\bx_n):n\in \NN\}$, with each $\bx_n\in \widehat G$ (to see this, use that the Euclidean topology is second countable, or that $O = \cup_{n\in \NN} G_n$, where each $G_n:= \{\bx\in O:\mathrm{dist}(\bx,\Gamma\cap \partial O)\leq n^{-1}\}$ is compact and so has a finite subcover). Thus
$$
0<|\widehat G|  = \left| \bigcup_{n=1}^\infty \widehat G\cap B_{\varepsilon(\bx_n)}(\bx_n)\right| \leq \sum_{n=1}^\infty \left|\widehat G\cap B_{\varepsilon(\bx_n)}(\bx_n)\right|.
$$
Thus, for some $m\in \NN$, $\left|\widehat G\cap B_{\varepsilon(\bx_m)}(\bx_m)\right|>0$. Since
$$
|z(\by)f_{\psi_{\bx_m}}(\by)|=|\widetilde z(\by)f_{\psi_{\bx_m}}(\by)|\geq B-3\epsilon, \quad \by\in \widehat G\cap B_{\varepsilon(\bx_m)}(\bx_m),
$$
$A\geq \esssup_{\by\in \Gamma}|z(\by)f_{\psi_{\bx_m}}(\by)| \geq B-3\epsilon$. Since this holds for all $\epsilon>0$ the proof is concluded.
\end{proof}

\begin{theorem} \label{thm:ENweighted} Suppose that $\Omega_-$ is a bounded Lipschitz domain with boundary $\Gamma$, that \eqref{eq:necsuff} holds, and that $w\in L^\infty(\Gamma)$ satisfies \eqref{eq:wbound} for some $c_->0$. Then
\begin{equation} \label{eq:ENweighted}
\|D\|_{C_w(\Gamma), \mathrm{ess}} = \lim_{\delta\to 0}\esssup_{\bx\in \Gamma} \int_{\Gamma\cap B_\delta(\bx)} \frac{w(\by)}{w(\bx)}\left|\frac{\partial \Phi(\bx,\by)}{\partial n(\by)}\right| \rd s(\by).
\end{equation}
\end{theorem}
\begin{proof}
Given $\delta >0$ choose $\epsilon\in (0,\delta)$ and  $\chi\in C(\Gamma\times\Gamma)$ such that $0\leq \chi(\bx,\by)\leq 1$ for $\bx,\by\in \Gamma$, $\chi(\bx,\by)=1$ for $|\bx-\by|\geq \delta$, and $\chi(\bx,\by)=0$ for $|\bx-\by|\leq \epsilon$. Define $K_\delta:C(\Gamma)\to C(\Gamma)$ by
\beq\label{eq:Kdelta}
K_\delta \phi(\bx) := \int_\Gamma \chi(\bx,\by)\pdiff{\Phi(\bx,\by)}{n(\by)} \phi(\by)\, \rd s(\by),
\quad \bx\in \Gamma.
\eeq
Then $K_\delta$ is an integral operator with kernel $\sum_{j=1}^d k_j(\bx,\by)\bn_j(\by)$, with each $k_j$ continuous, and so is compact on $C(\Gamma)$. Further, for $\phi\in C(\Gamma)$,
\begin{eqnarray*}
\|(D-K_\delta)\phi\|_{C_w(\Gamma)}&\leq &\esssup_{\bx\in \Gamma} \int_{\Gamma\cap B_\delta(\bx)} \frac{|\phi(\by)|}{w(\bx)}\left|\frac{\partial \Phi(\bx,\by)}{\partial n(\by)}\right| \rd s(\by) \leq  R_\delta \|\phi\|_{C_w(\Gamma)},
\end{eqnarray*}
where
$$
R_\delta := \esssup_{\bx\in \Gamma}\int_{\Gamma\cap B_\delta(\bx)} \frac{w(\by)}{w(\bx)}\left|\frac{\partial \Phi(\bx,\by)}{\partial n(\by)}\right| \rd s(\by).
$$
Thus $\|D\|_{C_w(\Gamma), \mathrm{ess}}\leq \|(D-K_\delta)\|_{C_w(\Gamma)}\leq R_\delta$. Since this holds for every $\delta>0$, $\|D\|_{C_w(\Gamma),\ess}\leq \lim_{\delta\to 0}R_\delta$.

Let $C^\prime(\Gamma)$ denote the dual space of $C(\Gamma)$, the space of regular complex Borel measures on $\Gamma$ (see, e.g., \cite[Theorem 6.19]{RudRCA}). Arguing as in \cite[p.~107]{Kral:80}, to see that $\|D\|_{C_w(\Gamma),\ess}\geq \lim_{\delta\to 0}R_\delta$, it is enough to show that $\|D-K\|_{C_w(\Gamma)}\geq \lim_{\delta\to 0}R_\delta$ for every finite rank operator $K$ on $C(\Gamma)$, since this set of operators is dense in the space of compact operators on $C(\Gamma)$
\footnote{Recall that there exists a sequence $(R_n)_{n\in \NN}$ of bounded linear operators on $C(\Gamma)$ that converges strongly to the identity with the range of $R_n$ finite-dimensional for each $n$, so that, for every compact operator $K$, $R_nK$ has finite rank and (see, e.g., \cite[Theorem 10.10]{Kr:14}) $\|R_nK-K\|_{C(\Gamma)}\to 0$ as $n\to\infty$. Explicitly, see, e.g., \cite[p.~186]{Jo:82}, we can take $R_n\phi(\bx) := \sum_{\bx\in G}\chi_\bx \phi(\bx)$, $\phi\in C(\Gamma)$, where $G$ and $\chi_\bx$ are as defined in and above \eqref{eqn:pou2}, with $\varepsilon=1/n$.}. Further, if $K$ is finite rank then
$$
K\phi = \sum_{j=1}^N \phi_j\int_\Gamma \phi \,\rd \mu_j, \quad \phi\in C(\Gamma),
$$
for some $N\in \NN$, $\phi_j\in C(\Gamma)$, $\mu_j\in C^\prime(\Gamma)$, $j=1,...,N$.
Moreover, as in \cite[p.~108]{Kral:80}, we can assume that $\{\bx\in \Gamma:|\mu_j|(\{\bx\})>0\}$ is finite since any $\mu\in C^\prime(\Gamma)$ can be approximated arbitrarily well in norm by such a measure. Thus it is enough to consider the case that $\mu_j = \mu_j^c + \mu_j^d$, where each $\mu^d_j$ is a finite sum of Dirac delta measures and each $\mu_j^c$ is continuous, i.e.~$|\mu_j^c|(\{\bx\})= 0$ for each $\bx\in \Gamma$, which implies also that $|\mu_j^c|(\Gamma\cap B_\delta(\bx))\to 0$ as $\delta\to 0$ for each $\bx\in \Gamma$, and hence, using the compactness of $\Gamma$, also that
\begin{equation} \label{eq:delta}
\lim_{\delta \to0}\sup_{\bx \in \Gamma}|\mu^c_j|(\Gamma\cap B_\delta(\bx)) \to 0.
\end{equation}

Let $K^c:C(\Gamma)\to C(\Gamma)$ denote the ``continuous'' part of $K$, meaning that
$$
K^c\phi := \sum_{j=1}^N \phi_j\int_\Gamma \phi \,\rd \mu^c_j, \quad \phi\in C(\Gamma);
$$
similarly, let $K^d:= K-K^c$ denote the ``discrete" part of $K$.
We show that $\|D-K\|_{C_w(\Gamma)}\geq \lim_{\delta\to 0}R_\delta$ by showing that  $\|D-K^c\|_{C_w(\Gamma)}\geq \lim_{\delta\to 0}R_\delta$ and that
\begin{equation} \label{eq:KKc}
\|D-K\|_{C_w(\Gamma)}\geq \|D-K^c\|_{C_w(\Gamma)}.
\end{equation}

To see that \eqref{eq:KKc} holds, let $\Delta:=\{\bx\in \Gamma: \sum_{j=1}^N|\mu_j^d|(\{\bx\})>0\}$, a finite subset of $\Gamma$, and let $C_\Delta(\Gamma):=\{\phi\in C(\Gamma):\phi(\bx)=0 \mbox{ for } \bx\in \Delta\}$.
Then, since $K^d\phi=0$ for all $\phi \in C_\Delta(\Gamma)$,
\begin{eqnarray}\label{eq:first}
\|D-K\|_{C_w(\Gamma)}& \geq& \sup_{\stackrel{\phi \in C_\Delta(\Gamma)}{\|\phi\|_{C_w(\Gamma)}} \leq 1} \|(D-K)\phi\|_{C_w(\Gamma)}
=\sup_{\stackrel{\phi \in C_\Delta(\Gamma)}{\|\phi\|_{C_w(\Gamma)}} \leq 1} \|(D-K^c)\phi\|_{C_w(\Gamma)}.
\end{eqnarray}
Thus \eqref{eq:KKc} holds if we can show that
\begin{equation} \label{eq:68a}
\sup_{\stackrel{\phi \in C_\Delta(\Gamma)}{\|\phi\|_{C_w(\Gamma)}} \leq 1} \|(D-K^c)\phi\|_{C_w(\Gamma)}\geq \sup_{\|\phi\|_{C_w(\Gamma)} \leq 1} \|(D-K^c)\phi\|_{C_w(\Gamma)}= \|D-K^c\|_{C_w(\Gamma)}.
\end{equation}
Given $\epsilon>0$, choose $\phi\in C(\Gamma)$ with $\|\phi\|_{C_w(\Gamma)}\leq 1$ such that
$$
\|(D-K^c)\phi\|_{C_w(\Gamma)}\geq \|D-K^c\|_{C_w(\Gamma)} -\epsilon.
$$
Then there exists $G\subset \Gamma$ of positive measure with $\mathrm{dist}(G,\Delta)>0$ such that $\Theta(\bx)=1/2$, \eqref{eq:wbound} holds, and
$$
\left|\frac{1}{w(\bx)}(D-K^c)\phi(\bx)\right|\geq \|D-K^c\|_{C_w(\Gamma)}-2\epsilon, \quad \mbox{for } \bx\in G.
$$
Further, for every $\delta>0$ there exists $\phi_\delta\in C_\Delta(\Gamma)$ with $|\phi_\delta(\bx)|\leq |\phi(\bx)|$, $\bx\in \Gamma$, so that $\|\phi_\delta\|_{C_w(\Gamma)}\leq 1$, and such that  $\mathrm{supp}(\phi-\phi_\delta)\subset \cup_{\bx\in \Delta} (\Gamma\cap B_\delta(\bx))$. This set is disjoint from $G$ if $\delta$ is small enough, so that, for some $\delta^\prime>0$ and all sufficiently small $\delta>0$, $D(\phi-\phi_\delta)(\bx)=K_{\delta^\prime}(\phi-\phi_\delta)(\bx)$, for $\bx\in G$, where $K_{\delta^\prime}$ is defined by \eqref{eq:Kdelta}. It is easy to see that  $\|K_{\delta^\prime}(\phi-\phi_\delta)\|_{C(\Gamma)}\to 0$ as $\delta\to 0$; moreover, $\|K^c(\phi-\phi_\delta)\|_{C(\Gamma)}\to 0$ as $\delta\to 0$, since $\lim_{\delta\to 0}|\mu_j^c|(\Gamma\cap B_\delta(\bx))=0$, for $j=1,...,N$ and each $\bx\in \Delta$. Thus
$$
\left|\frac{1}{w(\bx)}(D-K^c)\phi_\delta(\bx)\right|\geq \|D-K^c\|_{C_w(\Gamma)}-3\epsilon, \quad \bx\in G,
$$
if $\delta$ is small enough, so that
$$
\sup_{\stackrel{\phi \in C_\Delta(\Gamma)}{\|\phi\|_{C_w(\Gamma)}} \leq 1} \|(D-K^c)\phi\|_{C_w(\Gamma)} \geq \esssup_{\bx\in \Gamma} \left|\frac{1}{w(\bx)}(D-K^c)\phi_\delta(\bx)\right|\geq \|D-K^c\|_{C_w(\Gamma)}-3\epsilon.
$$
 Since this holds for all $\epsilon>0$, \eqref{eq:68a} holds and \eqref{eq:KKc} follows.

To see that $\|D-K^c\|_{C_w(\Gamma)}\geq \lim_{\delta\to 0}R_\delta$, note that, using Lemma \ref{lem:swap}, which applies since $D-K^c$ is a bounded operator on $C(\Gamma)$,
\begin{eqnarray*}
\|D-K^c\|_{C_w(\Gamma)} &=& \sup_{\|\phi\|_{C_w(\Gamma)}\leq 1} \esssup_{x\in \Gamma}\left|\frac{1}{w(\bx)}(D-K^c)\phi(\bx)\right|\\
& = &  \esssup_{x\in \Gamma}\left(\frac{1}{w(\bx)}\sup_{\|\phi\|_{C_w(\Gamma)}\leq 1}|(D-K^c)\phi(\bx)|\right)\\
& = &  \esssup_{x\in \Gamma}\left(\frac{1}{w(\bx)}\sup_{\|\phi\|_{C_w(\Gamma)}\leq 1}|(\lambda_\bx-\mu_\bx^c)(\phi)|\right),
\end{eqnarray*}
where
$$
\lambda_\bx (\psi) := \int_\Gamma \pdiff{\Phi(\bx,\by)}{n(\by)} \psi(\by)\, \rd s(\by), \quad \mu^c_\bx(\psi) :=  \sum_{j=1}^N \phi_j(\bx)\int_\Gamma \psi \, \rd\mu^c_j, \quad \psi \in C(\Gamma).
$$
Now, for $\bx\in \Gamma$ and $\delta>0$, where $S_\delta(\bx):=  \{\phi\in C(\Gamma): \|\phi\|_{C_w(\Gamma)}\leq 1, \, \mathrm{supp}(\phi)\subset B_\delta(\bx)\}$,
\begin{eqnarray*}
\sup_{\|\phi\|_{C_w(\Gamma)}\leq 1}|(\lambda_\bx-\mu_\bx^c)(\phi)|&\geq &\sup_{\phi\in S_\delta(\bx)}|(\lambda_\bx-\mu_\bx^c)(\phi)|\\
&\geq &\sup_{\phi\in S_\delta(\bx)}|\lambda_\bx(\phi)| -|\mu_\bx^c|(\Gamma\cap B_\delta(\bx))\\
& = & \int_{\Gamma\cap B_\delta(\bx)}  w(\by)\left|\frac{\partial \Phi(\bx,\by)}{\partial n(\by)}\right| \rd s(\by)  -|\mu_\bx^c|(\Gamma\cap B_\delta(\bx)),
\end{eqnarray*}
this last line following as a corollary of Lusin's theorem (see \cite[p.~56]{RudRCA}) and the dominated convergence theorem. Thus
\begin{eqnarray*}
\|D-K^c\|_{C_w(\Gamma)} &\geq & R_\delta - c_-^{-1} \sup_{\bx\in \Gamma} |\mu_\bx^c|(\Gamma\cap B_\delta(\bx)), \quad \delta>0.
\end{eqnarray*}
That $\|D-K\|_{C_w(\Gamma)} \geq \lim_{\delta\to 0} R_\delta$ follows from this bound, \eqref{eq:KKc}, and   \eqref{eq:delta}.
\end{proof}

In the following lemma we use again, in the case that $\Gamma$ is Lipschitz polyhedral, the notation $\Gamma_\bx$ of \eqref{eqn:cone} for the cone that coincides with $\Gamma$ in a neighbourhood of $\bx\in \Gamma$.

\begin{lemma} \label{lem:mostwork} Suppose that $\Omega_-$ is a bounded Lipschitz polyhedron with boundary $\Gamma$ and that, for some $\bx\in \Gamma$, some relatively open $\Gamma^\times\subset \Gamma_\bx \cap \Gamma$, and some $C>0$,
\begin{equation} \label{infG}
\inf_{\bx\in \Gamma^\times} \int_{\Gamma^\times} \left|\frac{\partial \Phi(\bx,\by)}{\partial n(\by)}\right| \rd s(\by)\geq C.
\end{equation}
Then, for every $w\in L^\infty(\Gamma)$ that satisfies \eqref{eq:wbound} for some $c_->0$,
$$
\|D\|_{C_w(\Gamma),\ess} \geq C.
$$
\end{lemma}
\begin{proof} Suppose that \eqref{infG} holds and, without loss of generality, suppose that $\bx=\bze$. Then, making a change of variables as in the proof of Lemma \ref{lem:Vcommute}, we see that
\begin{equation} \label{eq:infG2}
\inf_{\bx\in \beta\Gamma^\times} \int_{\beta\Gamma^\times} \left|\frac{\partial \Phi(\bx,\by)}{\partial n(\by)}\right| \rd s(\by)\geq C,
\end{equation}
for all $\beta>0$. Given $\delta>0$, choose $\beta>0$ so that $\beta\Gamma^\times \subset \Gamma \cap B_{\delta/2}(\bze)$. Then
\begin{eqnarray*}
& & \esssup_{\bx\in B_{\delta/2}(\bze)\cap \Gamma}\int_{\Gamma\cap B_\delta(\bx)} \frac{w(\by)}{w(\bx)}\left|\frac{\partial \Phi(\bx,\by)}{\partial n(\by)}\right| \rd s(\by) \\ & &\geq
\esssup_{\bx\in \beta\Gamma^\times}\left(\frac{1}{w(\bx)}\int_{\beta\Gamma^\times} w(\by)\left|\frac{\partial \Phi(\bx,\by)}{\partial n(\by)}\right| \rd s(\by)\right)\geq C,
\end{eqnarray*}
since the last integral above is $\geq C\,\essinf_{\by\in \beta\Gamma^\times} w(\by)$, for all $\bx \in \beta \Gamma^\times$. Thus, for all $\delta>0$,
$$
\esssup_{\bx\in \Gamma}\int_{\Gamma\cap B_\delta(\bx)} \frac{w(\by)}{w(\bx)}\left|\frac{\partial \Phi(\bx,\by)}{\partial n(\by)}\right| \rd s(\by) \geq C,
$$
and the result follows from Theorem \ref{thm:ENweighted}.
\end{proof}

\

\begin{proof}[Proof of Theorem \ref{thm:Q3}] Suppose that $n\in \NN$ with $n\geq 2$, $\theta \in (0,\pi/4]$ and $\Omega_-:= \Omega_{\theta,n}$, the open book polyhedron as in Definition \ref{def:Gamma3-d}, and note that $\bze$ is one of the vertices. Relabelling the ``front'' and ``back'' pages as in Definition \ref{def:widetildeGamma}, let $\Gamma^*:= \cup_{m=1}^{2n} \widetilde \Gamma_m$, so that $\Gamma^*$ is contained in the cone $\Gamma_\bze$ with vertex $\bze$. Let
$$
\Gamma^\times := \{\bx\in \Gamma^*: (x_1-1/2)^2+(x_3+1/2)^2<1/64\}.
$$
Then $\Gamma^\times= \cup_{m=1}^{2n} \Gamma_m^\times$, where each $\Gamma_m^\times\subset \widetilde \Gamma_m$ is an ellipse, in particular
$$
\Gamma_1^\times =  \{\bx=(x_1,x_2,x_3): (x_1-1/2)^2+(x_3+1/2)^2<1/64, \,x_2=0\}
$$
is the circular disc of radius $1/8$ in the plane $x_2=0$ centred on $x_1=1/2$, $x_3=-1/2$.
Further (cf.\ Lemma \ref{lem:ANDLP3-d}), for each $\bx\in \Gamma_j^\times\subset \Gamma^\times$, $j=1,\ldots, 2n$,
$$
\alpha(\bx) := 4\pi \int_{\Gamma^\times} \left|\frac{\partial \Phi(\bx,\by)}{\partial n(\by)}\right| \rd s(\by)
$$
is the sum of the solid angles subtended at $\bx$ by the $2n-1$ ellipses $\Gamma_m^\times$, with $m\in \{1,\ldots,2n\}$, $m\neq j$. As $\theta\to 0$ (``closing the book'', Remark \ref{rem:closing}), each ellipse $\Gamma_m^\times$ approaches the disc $\Gamma_1^\times$, so that $\Gamma^\times$ comprises asymptotically $2n$ circular discs of radius $1/8$ that are concentric and approximately parallel. Thus, as $\theta\to 0$,
$$
\inf_{\bx\in \Gamma^\times} 4\pi \int_{\Gamma_m^\times} \left|\frac{\partial \Phi(\bx,\by)}{\partial n(\by)}\right| \rd s(\by) \to \pi, \quad m=1,\ldots, 2n,
$$
this being the solid angle subtended by the disc $\Gamma_1^\times$ at each point of the circle $\{\bx=(x_1,x_2,x_3):(x_1-1/2)^2+(x_3+1/2)^2=1/64, \, x_2=p\}$, in the limit $p\to 0$.
Thus,  given any $\epsilon >0$ there exists $\theta_0\in (0,\pi/4]$ such that
$$
\inf_{\bx\in \Gamma^\times} \int_{\Gamma^\times} \left|\frac{\partial \Phi(\bx,\by)}{\partial n(\by)}\right| \rd s(\by) = \inf_{\bx\in \Gamma^\times} \frac{\alpha(\bx)}{4\pi} \geq \frac{2n-1}{4} - \epsilon, \quad 0<\theta\leq \theta_0.
$$
The result follows by applying Lemma \ref{lem:mostwork}.
\end{proof}

\section*{Acknowledgements}

This paper is dedicated, on the occasion of his 85th birthday, to Wolfgang Wendland (Stuttgart) who has had a leading role in our PDE and BIE community for many years. In particular, we thank Wolfgang for his insightful survey paper \cite{We:09}, which prompted the current work, and for many enjoyable and illuminating discussions of second-kind integral equations on non-smooth domains, dating back
to the Joint IMA/SIAM Conference on the State of the Art in Numerical Analysis in 1986.
The authors thank Johannes Elschner (WIAS, Berlin), Raffael Hagger and Karl-Mikael Perfekt (both University of Reading), and Eugene Shargorodsky (King's College London) for a number of very useful discussions.
EAS was supported by UK Engineering and Physical Sciences Research Council grant EP/R005591/1.

\footnotesize{
\bibliographystyle{plain}

}
\end{document}